\documentclass[a4paper, oneside]{amsart}

\usepackage{geometry}
\RequirePackage[l2tabu, orthodox]{nag}
\usepackage[all, warning]{onlyamsmath}
\usepackage{amsmath}
\usepackage{amsthm}
\usepackage{amsfonts}
\usepackage{amssymb}
\usepackage{amscd}
\usepackage{mathrsfs}
\usepackage{bm}
\usepackage{bbm}
\usepackage{stmaryrd}
\usepackage{braket}
\usepackage[all]{xy}
\usepackage{pb-diagram}
\usepackage{graphicx}
\usepackage{url}
\theoremstyle{plain}
\newtheorem{thm}{Theorem}[subsection]
\newtheorem{lem}[thm]{Lemma}
\newtheorem*{claim}{Claim}
\newtheorem{prop}[thm]{Proposition}
\newtheorem{cor}[thm]{Corollary}
\newtheorem*{aim}{Aim}
\newtheorem{dfn-lem}[thm]{Definition-Lemma}
\newtheorem{ex-thm}[thm]{Example}

\theoremstyle{definition}
\newtheorem{dfn}[thm]{Definition}
\newtheorem{ex}[thm]{Example}

\newtheorem{rmk}[thm]{Remark}

\theoremstyle{remark}
\newcommand{\Z}{\mathbb{Z}}
\newcommand{\Q}{\mathbb{Q}}
\newcommand{\R}{\mathbb{R}}
\newcommand{\C}{\mathbb{C}}

\newcommand{\bbD}{\mathbb{D}}

\newcommand{\bfS}{\mathbf{S}}
\newcommand{\bfT}{\mathbf{T}}

\newcommand{\bfx}{\mathbf{x}}

\newcommand{\bmi}{\bm{i}}

\newcommand{\mcH}{\mathcal{H}}

\newcommand{\mcN}{\mathcal{N}}
\newcommand{\mcO}{\mathcal{O}}

\newcommand{\mcX}{\mathcal{X}}

\newcommand{\msC}{\mathscr{C}}

\newcommand{\msF}{\mathscr{F}}
\newcommand{\msG}{\mathscr{G}}

\newcommand{\msI}{\mathscr{I}}

\newcommand{\mfa}{\mathfrak{a}}

\newcommand{\mfz}{\mathfrak{z}}

\DeclareMathOperator{\Hom}{Hom}

\DeclareMathOperator{\Map}{Map}

\DeclareMathOperator{\Stab}{Stab}

\DeclareMathOperator{\sgn}{sgn}
\DeclareMathOperator{\id}{id}

\DeclareMathOperator{\dd}{d}%differential
\DeclareMathOperator{\supp}{supp}
\DeclareMathOperator{\Span}{Span}

\newcommand{\re}{\operatorname{Re}}

\newcommand{\pr}{\mathrm{pr}}
\newcommand{\ev}{\mathrm{ev}}
\newcommand{\diag}{\mathrm{diag}}
\newcommand{\isomto}{\stackrel{\sim}{\longrightarrow}}
\newcommand{\isomfrom}{\stackrel{\sim}{\longleftarrow}}
\newcommand{\bs}{\backslash}

\newcommand{\ra}{\rightarrow}
\newcommand{\Ra}{\Rightarrow}
\newcommand{\Lra}{\Leftrightarrow}
\newcommand{\hra}{\hookrightarrow}
\newcommand{\tp}[1]{{}^t\!#1} %matrix transpose
\newcommand{\setm}{\!-\!}
\newcommand{\abs}[1]{\lvert#1\rvert}%bracket
\newcommand{\tup}[2]{#1_1, \dots,  #1_{#2}}%tuple

\numberwithin{equation}{section}
\newcommand{\Yo}{Y^{\circ}}
\newcommand{\PPC}{\mathbb{P}^{g-1}(\mathbb{C})}
\newcommand{\Rpos}{\mathbb{R}_{>0}}

\newcommand{\brk}[1]{\langle#1\rangle}%bracket

\newcommand{\XQ}{X_{\Q}}
\newcommand{\YQ}{\Q^g \setm \{0\}}
\newcommand{\Qw}{Q}

\DeclareMathOperator{\loc}{loc}
\newcommand{\shfC}{\C}

\begin{document}
\title{Shintani-Barnes cocycles and values of the zeta functions of algebraic number fields}
\author{Hohto Bekki}
\address{Department of Mathematics, Faculty of Science and Technology, Keio University, 3-14-1 Hiyoshi, Kohoku-ku, Yokohama, Kanagawa, 223-8522, Japan}
\curraddr{}
%\email{hohto.bekki@keio.jp}
\email{bekki@math.keio.ac.jp}
\thanks{}
%\subjclass[2010]{11R42, 11F75}
%\date{\today}
\dedicatory{}

\begin{abstract}
In this paper, we construct a new Eisenstein cocycle called the Shintani-Barnes cocycle which specializes in a uniform way to the values of the zeta functions of general number fields at positive integers.
Our basic strategy is to generalize the construction of the Eisenstein cocycle presented in the work of Vlasenko and Zagier by using some recent techniques developed by Bannai, Hagihara, Yamada, and Yamamoto in their study of the polylogarithm for totally real fields. 
We also closely follow the work of Charollois, Dasgupta, and Greenberg. In fact, one of the key ingredients in this paper which enables us to deal with general number fields is the introduction of a new technique called the ``exponential perturbation'' which is a slight modification of the $Q$-perturbation studied in their work. 
\end{abstract}

\maketitle

\tableofcontents

\section{Introduction}\label{sect:intro}

%\subsection{Motivation}
It is classically known that the Hecke integral formula~\cite{H17:Ube} expresses the zeta function of a number field of degree $g$ as an integral of the Eisenstein series over a certain torus orbit on the locally symmetric space for $SL_g(\Z)$. 

In some special cases, typically in the case where the number field is totally real, it is known that such an integral formula has a cohomological interpretation, and this often enables us to access the algebraic properties of the special values of the zeta function.
More precisely, one can construct a certain $(g-1)$-cocycle on $SL_g(\Z)$ which can be thought as an algebraic counterpart of the Eisenstein series, and a $(g-1)$-cycle on $SL_g(\Z)$ which can be thought as an algebraic counterpart of the torus orbit, so that their paring gives the value of the zeta function of a given totally real number field.
%
%
%
%Such a cocycle is often called the Eisenstein cocycle, and actually many different kinds of Eisenstein cocycles have been constructed and studied by Harder~\cite{MR892187}, Sczech~\cite{S93:Eis}, Nori~\cite{N95:Som}, Solomon~\cite{MR1631700}, Hill~\cite{H07:Shi}, Vlasenko-Zagier~\cite{VZ13:Hig}, Charollois-Dasgupta-Greenberg~\cite{CDG15:Int}, Beilinson-Kings-Levin~\cite{BKL18:Top}, Bergeron-Charollois-Garcia~\cite{BCG20:Tra}, Fl\'{o}rez-Karabulut-Wong~\cite{MR3991431}, Lim-Park~\cite{lim2019milnor}, Bannai-Hagihara-Yamada-Yamamoto~\cite{BHYY19:Can}, Sharifi-Venkatesh~\cite{sharifi2020eisenstein}, and so on.
Such a cocycle is often called the Eisenstein cocycle. Actually many different kinds of Eisenstein cocycles have been constructed and studied by Harder~\cite{MR892187}, Sczech~\cite{S93:Eis}, Nori~\cite{N95:Som}, Solomon~\cite{MR1631700}, Hill~\cite{H07:Shi}, Vlasenko-Zagier~\cite{VZ13:Hig}, Charollois-Dasgupta-Greenberg~\cite{CDG15:Int}, Beilinson-Kings-Levin~\cite{BKL18:Top}, Bergeron-Charollois-Garcia~\cite{BCG20:Tra}, Fl\'{o}rez-Karabulut-Wong~\cite{MR3991431}, Lim-Park~\cite{lim2019milnor}, Bannai-Hagihara-Yamada-Yamamoto~\cite{BHYY19:Can}, Sharifi-Venkatesh~\cite{sharifi2020eisenstein}, and so on, and various applications have been obtained.

However, the number fields treated in these previous works are basically limited to totally real fields or totally imaginary fields. 
The aim of this paper is to propose a new formulation in which we can treat all number fields in a uniform way.

\subsection{Shintani cocycles}

Among these many kinds of construction of the Eisenstein cocycle, a method we use in this paper is called Shintani's method, and the Eisenstein cocycles constructed by Shintani's method are often called the Shintani cocycles
\footnote{The use of these terminologies seems to depend on the authors. We adopt this convention in this paper.}
, cf. \cite{MR1631700}, \cite{H07:Shi}, \cite{CDG15:Int}, \cite{lim2019milnor}, \cite{BHYY19:Can}.
Roughly speaking, a Shintani cocycle is constructed as a family of objects (e.g., functions, formal power series, distributions, etc.) indexed by rational cones in $\R^g$. 
Therefore, what we do in this paper are basically the following:
\begin{enumerate}
\item Define a certain object ``$\psi_{C}$'' for each rational cone $C \subset \R^g$.
\item Prove that the family $(\psi_C)_C$ satisfies the ``cocycle relation''. 
\item Prove that the cohomology class defined by $(\psi_C)_C$ specializes to the special values of the zeta function of a given number field. 
\end{enumerate}

Let $g, k\geq 1$ be integers. 
In this paper, we say that a matrix $Q \in GL_g(\Q)$ is \textit{irreducible} if its characteristic polynomial is irreducible over $\Q$.  
In Section \ref{sect:shi}, for a rational open cone
\[
C_I =\sum_{i=1}^g \Rpos \alpha_i \subset \R^g
\]
generated by $I =(\alpha_1, \dots, \alpha_g) \in (\Q^g \setm \{0\})^g$, and an irreducible matrix $Q\in GL_g(\Q)$, we consider a holomorphic function 
\begin{align*}
\psi_{kg, I}^Q(y) := \sgn (I) \sum_{\substack{x \in C_{I}^Q \cap \Z^g \setm \{0\}}} \frac{1}{\brk{x,y}^{g+kg}} % \in \Gamma(V_I, \msF_{kg}), % \subset \Gamma(V_I, \Omega^{g-1}_{\C^g \setm \{0\}}), 
\end{align*}
on
\[
%V_I:=\{ y \in \C^g \ |\ \re (\brk{\alpha_i, y})>0, \ i=1, \dots, g  \}, 
\big\{ y \in \C^g \ |\  \exists \lambda \in \C^{\times},  1 \leq \forall i \leq g, \  \re (\brk{\alpha_i, \lambda y})>0  \big\} \subset \C^g \setm \{0\}, 
\]
where
\begin{itemize}
\item $\sgn(I)=\sgn (\det (\alpha_1, \dots, \alpha_g)) \in \{0, \pm 1\}$, 
%\item $\brk{x,y}, \brk{\alpha_i,y}$ are the scalar product, 
\item $\brk{\ , \ }$ denotes the scalar product, 
\item $C_I^Q$ is the ``exponential $Q$-closure'' of the cone $C_I$ (Section \ref{subsect:pert}). 
%\item $y \in V_I \subset \C^g$, and $\brk{x,y}$ is the scalar product, 
\end{itemize}
%cf. Definition \ref{dfn:psi}. 
%
Then we prove that the collection $(\psi_{{kg}, I}^Q)_{I,Q}$ defines a class 
\[
[\Psi_{kg}] \in H^{g-1} \big(\Yo, SL_g(\Z), \msF_{kg}^{\Xi} \big)
\]
of the equivariant cohomology of a certain $SL_g(\Z)$-equivariant sheaf $\msF_{kg}^{\Xi}$ on $\Yo:=\C^g \setm \bmi \R^g$, cf. Section \ref{sect:sp and sh} and Theorem \ref{thm:SB class}. 
%%of the equivariant cohomology of a certain $SL_g(\Z)$-equivariant sheaf $\msF_{kg}^{\Xi}$ on $\Yo:=\C^g \setm \bmi \R^g$ which is related to the sheaf $\Omega^{g-1}_{\PPC}$ of holomorphic $(g-1)$-forms on $\PPC$.  %cf. Section \ref{sect:sp and sh} and Theorem \ref{thm:SB class}. 
We call our Shintani cocycle the Shintani-Barnes cocycle because the function $\psi_{kg,I}^Q(y)$ is essentially the Barnes zeta function. %, cf. Remark \ref{rmk:decomp eis}. % associated to $C_I^Q$. 

%\subsection{Special values of the zeta functions}

Then for a number field $F/\Q$ of degree $g$; a fractional ideal $\mfa \subset F$; and a continuous map $\chi: F_{\R}^{\times} \ra \Z$, we construct a specialization map \eqref{eqn:spec}: 
\[
H^{g-1}(\Yo, SL_g(\Z), \msF_{kg}^{\Xi}) \ra H_{sing}^{g-1}(F_{\R}^{\times}/\mcO_{F,+}^{\times}, \C) \ra \C, 
\]
using a certain integral operator. 
The image of the Shintani-Barnes cocycle $[\Psi_{kg}]$ under this specialization map can be computed using the classical Hurwitz formula (Proposition \ref{prop:Hur formula}, Example \ref{ex:Hur formula}) and a version of the Shintani cone decomposition (Proposition \ref{prop:sfd}). As a result, we prove that the class $[\Psi_{kg}]$ maps to the value of the partial zeta function: 
\[
%\frac{(k!)^g\sqrt{|D_{F}|} N\mfa^{-k}}{(g+gk-1)!} \zeta_{F, +}(\mfa^{-1}, 1+k)
%%(20201012)\pm \frac{(k!)^g}{(g+gk-1)!}\frac{\sqrt{D_{\mcO_F}}}{N\mfa^{k}} \zeta_{\mcO_F, +}(\mfa^{-1}, 1+k) 
\pm \frac{\sqrt{D_{\mcO_F}}N\mfa (k!)^g}{(g+gk-1)!} 
\zeta_{\mcO_F}(\bm{\varepsilon}^{k+1}\chi, \mfa^{-1}, k+1)
\]
under the specialization map, where $\bm{\varepsilon}:F_{\R}^{\times} \ra \{\pm 1\}$ is the sign character, cf. Theorem \ref{thm:main thm}. 
%Cf. Theorem \ref{thm:main thm} and Remark \ref{rmk:main thm}. 

The idea of using the Barnes zeta functions is based on the work of Vlasenko and Zagier~\cite{VZ13:Hig} dealing with the values of the zeta functions of real quadratic fields at positive integers, and the idea of constructing the Shintani cocycle as a \v{C}ech cocycle of an equivariant sheaf is based on the work of Bannai, Hagihara, Yamada, and Yamamoto~\cite{BHYY19:Can} in which the higher dimensional polylogarithm associated to a totally real field is studied. 
Moreover, the concept of the exponential $Q$-closure $C_I^Q$ of a cone $C_I$ is a slight modification of the $Q$-closure studied by Charollois, Dasgupta, Greenberg~\cite{CDG15:Int}, and Yamamoto~\cite{Y10:On-}. We use irreducible matrices $Q\in GL_g(\Q)$ instead of ``irrational vectors'' used in \cite{CDG15:Int}. 
These three ideas are the main ingredients in this paper which enable us to deal with general number fields.

\subsection{Structure of the paper}
%\subsection{Outline? of the paper}
%We briefly summarize the structure of the paper. 
%
The first four sections (Section \ref{sect:prelim} to Section \ref{sect:cone and pert}) are devoted to prepare some tools that are necessary for the definition of the Shintani-Barnes cocycle. 
More precisely, in Section \ref{sect:prelim} we review some elementary facts about irreducible matrices of $GL_g(\Q)$ and their relationship to number fields. 
In Section \ref{sect:sp and sh} we introduce the sheaves $\msF_d$ and $\msF_{d}^{\Xi}$ on $\Yo= \C^g \setm \bmi \R^g$, and examine the basic properties of these sheaves. 
Then in Section \ref{sect:equiv coh} we compute the equivariant cohomology groups of these sheaves using the equivariant \v{C}ech complex. 
In Section \ref{sect:cone and pert} we introduce the notion of the exponential perturbation, and prove the cocycle relation satisfied by rational cones. 
Based on these preparations, in Section \ref{sect:shi} we give the definition of the Shintani-Barnes cocycle. 

The remaining sections (Section \ref{sect:int} and Section \ref{sect:spec}) are devoted to show that we can obtain the special values of the zeta functions as a specialization of the Shintani-Barnes cocycle. 
In Section \ref{sect:int} we first introduce a certain integral operator, and construct the first half of the specialization map. 
In Section \ref{sect:spec} we finish the construction of the specialization map using a version of the Shintani cone decomposition, and finally prove the main result Theorem \ref{thm:main thm}.

%applications will be studied in the future works. 

%\clearpage

\section{Preliminaries}\label{sect:prelim}

\noindent
\textbf{Convention.}
\begin{itemize}
\item Throughout the paper we fix an integer $g \geq 1$. 
\item For a ring $R$,  a vector $x \in R^g$ is always regarded as a column vector, and the matrix algebra $M_g(R)$ acts on $R^g$ by the matrix multiplication from the left. 
\item For $x_1, \dots, x_g \in R^g$, we often regard $(x_1, \dots, x_g)$ as a $g\times g$-matrix whose columns are $x_1, \dots, x_g$. 
\item For $\gamma \in M_g(R)$, its transpose is denoted by $\tp{\gamma} \in M_g(R)$. 
\item The bracket
\[
\brk{\ ,\ }: R^g \times R^g \ra R; (x,y) \mapsto \brk{x,y}=\tp x y
\]
denotes the scalar product. 
\item If $A$ and $B$ are sets, then $A \setm B$ denotes the relative complement of $B$ in $A$. %For sets $A, B$, the relative complement of $B$ in $A$ is denoted by $A \setm B$.  
\item Let $\{S_{\lambda}\}_{\lambda \in \Lambda}$ be a family of sets. For $s \in \prod_{\lambda \in \Lambda} S_{\lambda}$, the $\lambda$-component of $s$ is often denoted by $s_{\lambda} \in S_{\lambda}$. 
%we often express the $\lambda$-component of $s$ as $s_{\lambda} \in S_{\lambda}$. %, and $s^{(\mu)}$
\end{itemize}

%\section{Preliminaries on irreducible matrices}
\subsection{Irreducible matrices}\label{subsect:irr}
In this subsection we review some basic facts about irreducible matrices of $GL_g(\Q)$. We say that a matrix $Q \in  GL_g(\Q)$ is \textit{irreducible} \textit{over $\Q$} if the characteristic polynomial of $Q$ is an irreducible polynomial over $\Q$. We often drop ``over $\Q$'' if it is obvious from the context.  Let 
\begin{align*}
\Xi :=\Big\{Q \in GL_g(\Q) \ \Big| \  \text{irreducible over $\Q$}\Big\}
\end{align*}
denote the set of irreducible matrices of $GL_g(\Q)$. 
The group $GL_g(\Q)$ acts on $\Xi$ by the conjugate action. For $Q \in \Xi$ and $\gamma \in GL_g(\Q)$, let 
\[
[\gamma](Q):=\gamma Q \gamma^{-1} \in \Xi
\]
denote this conjugate action. 

Now, for $Q \in \Xi$, let 
\[
\Gamma_Q:=\Stab_{SL_g(\Z)}(Q)=\left\{\gamma \in SL_g(\Z) \ \big|\  [\gamma](Q)=\gamma Q \gamma^{-1}=Q \right\}
\]
denote the subgroup of $SL_g(\Z)$ stabilizing $Q$. Moreover, let 
\begin{align*}
F_Q &:=\Q[Q] \subset M_g(\Q),\\
\mcO_Q &:=F_Q \cap M_g(\Z) \subset F_Q
\end{align*}
denote the subalgebras of $M_g(\Q)$ generated by $Q$ over $\Q$ and its ``$M_g(\Z)$-part'' respectively. 

%The following lemma summarizes the basic facts about irreducible matrices of $GL_g(\Q)$. 
\begin{lem}\label{lem:tor}
Let $Q \in \Xi$, and let $f_Q(X) \in \Q[X]$ be the characteristic polynomial of $Q$.
\begin{enumerate}
\item $Q$ has $g$ distinct eigenvalues in $\C$, and hence $Q$ is diagonalizable in $GL_g(\C)$. 
\item There are no non-zero proper $Q$-stable $\Q$-subspace of $\Q^g$. 
\item For any non-zero vector $x \in \Q^g \setm \{0\}$, the map
\begin{align*}
F_Q \isomto \Q^g; \ \gamma \mapsto \gamma x
\end{align*}
is an isomorphism of $\Q$-vector spaces.
\item The $\Q$-algebra $F_Q$ is a field of degree $g$ over $\Q$, and we have
\[
N_{F_Q/\Q}(\gamma) = \det \gamma 
\]
for $\gamma \in F_Q$, where $N_{F_Q/\Q}$ is the norm of the field extension $F_Q/\Q$. 
\item We have
\[
F_Q=\{\gamma \in M_g(\Q) \mid \gamma Q=Q \gamma  \}.
\]
\item We have
\[
\Gamma_Q =\{ \gamma \in \mcO_Q \mid N_{F_Q/\Q}(\gamma)=1 \} \subset \mcO_Q^{\times}, 
\]
i.e., $\Gamma_Q$ is the norm one unit group of $\mcO_Q$. 
\item The action of $\Gamma_Q$ on $\Q^g \setm \{0\}$ is free, i.e., for any $x \in \Q^g \setm \{0\}$ and $\gamma \in \Gamma_Q$, we have $\gamma x=x$ if and only if $\gamma =1$. 
%
%\item $Q$ has $g$ distinct eigenvalues in $\C$ and $g$ distinct eigenvectors in $\C^g$. 
%\item The $\Q$-algebra $F_Q$ is a field of degree $g$ over $\Q$.
%\item The natural map
%\[
%\Q[X]/(f_Q(X)) \isomto F_Q; X \mapsto Q
%\]
%is an isomorphism of $\Q$-algebras. In particular $f_Q(X)$ is the minimal polynomial of $Q$.
%\item $f_Q(X)$ is an irreducible polynomial in $\Q[X]$. In particular $f_Q(X)$ has $g$ distinct roots and $Q$ is a diagonalizable matrix in $GL_g(\C)$.
%\item Let $\gamma \in F_Q$, $\lambda \in \C$, and $x \in \Q^g \setm \{0\}$ such that $\gamma x =\lambda x$, i.e., $x$ is an eigenvector of $\gamma$ with eigenvalue $\lambda$. Then we have $\lambda \in \Q^{\times}$ and $\gamma =\lambda$, i.e., $x$ is a scalar matrix whose diagonal entries are all $\lambda$.
%\item The action of $\Gamma_Q$ on $\XQ$ is free, i.e., for any $\overline{\alpha} \in \XQ$, we have
%\[
%\Stab_{\Gamma_Q}(\overline{\alpha})=\{\gamma \in \Gamma_Q \mid \gamma \overline{\alpha}=\overline{\alpha}\}=\{1\}.
%\]
\end{enumerate}
\end{lem}

\begin{proof}
(1) is obvious since $f_Q(X)$ is an irreducible polynomial over $\Q$. 
(2) is also clear because if $V \subset \Q^g$ is a $Q$-stable $\Q$-subspace, then the characteristic polynomial of $Q|_{V}$ divides $f_Q(X)$. 

(3),(4) First, since $x\neq 0$, the image of the map 
\[
F_Q \ra \Q^g; \gamma \mapsto \gamma x
\]
is a non-zero $Q$-stable $\Q$-subspace. Hence by (2), this map is surjective. 
Now, again since $f_Q(X)$ is an irreducible polynomial over $\Q$, we see that $F_Q \simeq \Q[X]/(f_Q(X))$ is a field of degree $g$ over $\Q$. Therefore, by comparing the dimension, we find that the above map is an isomorphism. 
The identity $N_{F_Q/\Q}(\gamma)=\det \gamma$ is nothing but the definition of the norm. 

(5) Let $F_Q'$ denote the right hand side. The inclusion $F_Q \subset F_Q'$ is obvious. We compare the dimension. First we have
\[
F_Q' \otimes_{\Q} \C \subset F_Q'':=\{\gamma \in M_g(\C) \mid \gamma Q=Q \gamma \}. 
\]
%Then by (1), the right hand side is simultaneously diagonalizable in $M_g(\C)$ and isomorphic to the space of diagonal matrices. Therefore we find 
Then by (1), the right hand side $F_Q''$ is simultaneously diagonalizable in $M_g(\C)$. Therefore, $F_Q''$ is isomorphic to the space of diagonal matrices. Thus we find 
\[
\dim_{\Q} F_Q' =\dim_{\C} F_Q' \otimes_{\Q}\C \leq \dim_{\C} F_Q''=g=\dim_{\Q}F_Q, 
\]
and hence we obtain $F_Q=F_Q'$. 

(6) follows directly from (4) and (5).  

(7) By (5) we see that $\Gamma_Q \subset F_Q^{\times}$, and by (3) and (4), we see that $F_Q^{\times}$ acts freely on $\Q^g \setm \{0\}$. 
\end{proof}

%\noindent
%\textbf{Relation to number fields} 

%\subsection{Relation to number fields}
%\subsection{Torus embedding}
\subsection{Review on number fields}\label{subsect:number field}
%In this subsection we review some basic facts about number fields which will be needed in Section \ref{}. 
%
%We do not need the contents of this subsection until Section \ref{}. 
In this subsection we take a closer look at the relationship between irreducible matrices and number fields. 
%We need the contents of this subsection later in Section \ref{sect:spec}. 

Let $F/\Q$ be a number field of degree $g$, and let 
\[
\tau_1, \dots, \tau_g: F \hookrightarrow \C
\]
be the field embeddings of $F$ into $\C$.  
Let $\mcO \subset F$ be an order in $F$, i.e., $\mcO \subset F$ is a subring which is a finitely generated $\Z$-module and generates $F$ over $\Q$.
Let $\mfa \subset F$ be a proper fractional $\mcO$-ideal, i.e., $\mfa \subset F$ is a finitely generated $\mcO$-submodule such that
\begin{align}
\{\alpha \in F \mid \alpha \mfa \subset \mfa \}=\mcO. \label{eqn:order}
\end{align}
Let $w_1, \dots, w_g \in \mfa$ be a basis of $\mfa$ over $\Z$, and put
\begin{align*}
w &:=\tp{(w_1, \dots, w_g)} \in F^g, \\
w^{(i)} &:= \tau_i(w)=\tp{(\tau_i(w_1), \dots, \tau_i(w_g))} \in \C^g
\end{align*}
for $i=1, \dots, g$.
We define the norm polynomial $N_w(x)=N_w(x_1, \dots, x_g) \in \Q[x_1, \dots, x_g]$ with respect to this basis by
\[
N_w(x):=\prod_{i=1}^g \brk{x, w^{(i)}} \in \Q[x_1, \dots, x_g], 
\]
where $x=(x_1, \dots, x_g)$. 
%to be the ``norm'' with respect to this basis. 
The situation can be summarized in the following diagram:
\begin{align*}
\xymatrix@=10pt{
%x \ar@{|->}[d]& \in & \Z^g \ar[d]^-{\wr}& \subset & \Q^g \ar[d]^-{\wr} \ar[drr]^-{N_w}&&\\
%\brk{x,w} & \in & \mfa & \subset & F  \ar[rr]_-{N_{F/\Q}}&& \Q\\
x \ar@{|->}[d]& \in & \Z^g \ar[d]^-{\wr}& \subset & \Q^g \ar[d]^-{\wr} \ar[drr]^-{N_w}&&\\
\brk{x,w} & \in & \mfa & \subset & F  \ar[rr]_-{N_{F/\Q}}&& \Q .\\
}
\end{align*}
Moreover, let
\[
\rho_w: F \ra M_g(\Q)
\]
be the regular representation of $F$ with respect to the basis $w_1, \dots, w_g$, i.e., for $\alpha \in F$ and $x \in \Q^g$ we have
\begin{align}\label{eqn:reg rep}
\brk{\rho_w(\alpha) x, w}=\alpha\brk{x,w}=\brk{x,\alpha w} \in F.
\end{align}

\noindent
\textbf{Dual objects.}
Let $w_1^*, \dots, w_g^* \in F$ be the dual basis of $w_1, \dots, w_g$ with respect to the field trace $Tr_{F/\Q}$, i.e.,
\begin{align*}
Tr_{F/\Q}(w_iw_j^*)=\delta_{ij}
=
\begin{cases}
0 & \text{ if } i\neq j \\
1 & \text{ if } i=j. 
\end{cases}
\end{align*}
Then it is easy to see that $w_1^*, \dots, w_g^*$ form a $\Z$-basis of a proper fractional $\mcO$-ideal
\[
\mfa^*:=\{\alpha \in F \mid  Tr_{F/\Q}(\alpha \mfa) \subset \Z\}.
\]
We define %$w^*, w^{*(1)}, \dots, w^{*(g)}, N_{w^*}(x)$, and $\rho_{w^{*}}$ 
\begin{align*}
w^* &:=\tp{(w^*_1, \dots, w^*_g)} \in F^g, \\
w^{*(i)} &:= \tau_i(w^*)=\tp{(\tau_i(w^*_1), \dots, \tau_i(w^*_g))} \in \C^g, \\
N_{w^*}(x) &:=\prod_{i=1}^g \brk{x, w^{*(i)}} \in \Q[x_1, \dots, x_g], \\
\rho_{w^*} &: F \ra M_g(\Q)
\end{align*}
in the same way as above starting from the dual basis  $w_1^*, \dots, w_g^*$.

%Let $\theta \in F^{\times}$ be an element such that $F=\Q(\theta)$. Put $Q=\rho_w(\theta) \in GL_g(\Q)$. 
\begin{lem}\label{lem:rev}
Let $\theta \in F^{\times}$ be an element such that $F=\Q(\theta)$. Put $Q=\rho_w(\theta) \in GL_g(\Q)$. 
%Let $\theta \in F^{\times}$ be an element such that $F=\Q(\theta)$. %Put $Q=\rho_w(\theta) \in GL_g(\Q)$. 
\begin{enumerate}
%\item We have $\rho_w(\theta) \in \Xi$. Conversely, any element of $\Xi$ can be obtained in this way. 
\item We have $\Qw \in \Xi$. Conversely, any element of $\Xi$ can be obtained in this way. 
\item The regular representation $\rho_w : F\ra M_g(\Q)$ induces isomorphisms
\begin{align*}
\xymatrix@R=0pt{
F \ar[r]^-{\stackrel{\rho_w}{\sim}}&F_{\Qw} \\
\cup&\cup\\
\mcO \ar[r]^-{\sim}&\mcO_{\Qw} \\
\cup&\cup\\
\mcO^1 \ar[r]^-{\sim}&\Gamma_{\Qw} \\
}
\end{align*}
where $\mcO^1:=\{u \in \mcO^{\times} \mid N_{F/\Q}(u)=1\}$ is the norm one unit group of $\mcO$. 
\item $w^{*(1)}, \dots, w^{*(g)} \in \C^g$ are the dual basis of $w^{(1)}, \dots, w^{(g)} \in \C^g$ with respect to the scalar product $\brk{\ , \ }$, i.e., we have
\[
\brk{w^{*(i)}, w^{(j)}} =\delta_{ij}. 
\]
\item For $\alpha \in F$, we have 
\[
\rho_{w^*}(\alpha) = \tp{\rho_w(\alpha)}. 
\]
\item Let $\alpha \in F$. Then $w^{(i)}$ is an eigenvector of $\tp{\rho_w(\alpha)}$ with eigenvalue $\tau_i(\alpha)$. % for $i=1, \dots, g$.
%\item For $\alpha \in F$, the vector $w \in F^g$ is an eigenvector of $\tp{\rho_w(\alpha)}$ with eigenvalue $\alpha$. In particular, for $i=1, \dots, g$, the vector $w^{(i)} \in \C^g$ is an eigenvector of $\tp{\rho_w(\alpha)}$ with eigenvalue $\tau_i(\alpha)$. 
%
\item Let $\alpha \in F$. Then $w^{*(i)}$ is an eigenvector of $\rho_w(\alpha)$ with eigenvalue $\tau_i(\alpha)$. % for $i=1, \dots, g$.
\item For $\gamma \in \Gamma_{\Qw}$, we have
\begin{align*}
N_w(\gamma x) &=N_w(x),\\
N_{w^*}(\tp{\gamma}x) &=N_{w^*}(x).
\end{align*} 
\end{enumerate}
\end{lem}
\begin{proof}
(1) Since $\theta$ generates $F$, the characteristic polynomial of $\Qw=\rho_w(\theta)$ is irreducible, and hence $\Qw \in \Xi$. The latter half of the statement follows from Lemma \ref{lem:tor} (3), (4). 
Indeed, for $Q \in \Xi$, fix a non-zero vector $x \in \Q^g$ and take a basis $w_1, \dots, w_g \in F_Q$ corresponding to the standard basis of $\Q^g$ via the isomorphism 
\[
F_Q \isomto \Q^g; \ \gamma \mapsto \gamma x. 
\]
Let $\mfa \subset F_Q$ be the subset corresponding to $\Z^g \subset \Q^g$ under this isomorphism. 
Then we easily see that $\mfa$ is a proper $\mcO_Q$-ideal and that $\rho_w$ is the natural inclusion $F_Q \hookrightarrow M_g(\Q)$. Hence we find that $Q=\rho_w(Q)$.  

%(2) and (3) are clear from the definition. Note that we need (\ref{eqn:order}) to show $\mcO \simeq \mcO_{\Qw}$. 
(2) The first isomorphism $F \isomto F_{\Qw}$ is obvious. The second isomorphism follows from (\ref{eqn:order}), and the third follows from Lemma \ref{lem:tor} (6).  

(3) Put
\begin{align*}
W &:= (w^{(1)}, \dots, w^{(g)}) = (\tau_{j}(w_i))_{ij} \in M_g(\C), \\
W^*&:=(w^{*(1)}, \dots, w^{*(g)}) = (\tau_{j}(w_i^*))_{ij} \in M_g(\C). 
\end{align*}
Then by definition, we have
\begin{align}\label{eqn:W}
W \tp{W^*} = \big( Tr_{F/\Q}(w_iw_j^*) \big)_{ij}=1 \in M_g(\C), 
\end{align}
and hence
\[
\big( \brk{w^{*(i)}, w^{(j)}} \big)_{ij} =\tp{W^*} W =1. 
\]

(4), (5), (6) First, by \eqref{eqn:reg rep}, we have
\[
\brk{x, \alpha w}= \brk{\rho_w(\alpha) x, w}= \brk{x, \tp{\rho_w(\alpha)} w} \in F
\]
for all $x \in \Q^g$. Therefore, we find that $\alpha w = \tp{\rho_w(\alpha)} w \in F^g$. By applying $\tau_{i}$, we obtain (5). In particular, we have
\begin{align}\label{eqn:W2}
W 
%\begin{pmatrix}
%\tau_1(\alpha) & & \\
%& \ddots & \\
%&& \tau_g(\alpha)
%\end{pmatrix}
\diag(\tau_1(\alpha), \dots, \tau_g(\alpha))
=\tp{\rho_w(\alpha)} W, 
\end{align}
where $\diag(\tau_1(\alpha),\dots, \tau_g(\alpha)) \in M_g(\C)$ is the diagonal matrix with diagonal entries $\tau_1(\alpha), \dots, \tau_g(\alpha)$. 
Similarly, we have 
\begin{align}\label{eqn:W3}
W^* 
%\begin{pmatrix}
%\tau_1(\alpha) & & \\
%& \ddots & \\
%&& \tau_g(\alpha)
%\end{pmatrix}
\diag(\tau_1(\alpha), \dots, \tau_g(\alpha))
=\tp{\rho_{w^*}(\alpha)} W^*. 
\end{align}
On the other hand, by using \eqref{eqn:W} and \eqref{eqn:W2} we also find that 
\[
\diag(\tau_1(\alpha), \dots, \tau_g(\alpha)) \tp{W^*} = \tp{W^*}\tp{\rho_w(\alpha)}, 
\]
and hence by taking the transpose, we have
\begin{align}\label{eqn:W4}
W^* \diag(\tau_1(\alpha), \dots, \tau_g(\alpha)) =\rho_w(\alpha) W^*. 
\end{align}
By comparing \eqref{eqn:W3} and \eqref{eqn:W4}, we obtain (4) and (6). 

(7) follows from (2), (5), and (6). Indeed, take $u \in \mcO^1$ such that $\rho_w(u)=\gamma$. Then we have
\[
N_w(\gamma x)=\prod_{i=1}^g \brk{\gamma x, w^{(i)}} =\prod_{i=1}^g \brk{x, \tp{\rho_w(u)} w^{(i)}}=N_{F/\Q}(u)N_w(x)=N_w(x).  
\]
The statement for $N_{w^*}(x)$ can be proved similarly. 
\end{proof}

%\begin{dfn}[Partial zeta function]
%\end{dfn}

%\subsection{Equivariant sheaves on $\Yo$}

%\subsection{A Simplicial set and the \v{C}ech complexes}

\section{The space $\Yo$ and the sheaves $\msF_d, \msF_d^{\Xi}$}\label{sect:sp and sh}
%\section{The space $\Yo$ and the sheaf $\msF_d$}

%setting
%vanishing 
%cech
%%for general sheaf
%%(injective) resolution (for Fd)

\subsection{Definitions}\label{subsect:dfns}
Let $\PPC= (\C^g \setm \{0\})/ \C^{\times}$ be the complex projective $(g-1)-$space and let 
\[
\pi_{\C}: \C^g \setm \{0\} \ra \PPC
\]
be the natural projection. We define an open subset $\Yo$ of $\C^g \setm \{0\}$ by 
\begin{align*}
\Yo := \C^g \setm \bmi \R^g \subset \C^g \setm \{0\}, 
\end{align*}
where $\bmi \in \C$ is the imaginary unit. 
The group $GL_g(\Q)$ acts on $\C^g \setm \{0\}, \Yo, \PPC$ by the matrix action from the left. 
%
%Furthermore, for each integer $d \geq 0$, we define sheaves $\msF_d, \msF_d^{\Xi}$ on $\Yo$ by
For an integer $d\geq 0$, we define a sheaf $\msF_d$ on $\Yo$ as
\begin{align*}
\msF_d:= \pi_{\C}^{-1} \Omega^{g-1}_{\PPC}(-d) \big|_{\Yo},
\end{align*}
where $\Omega^{g-1}_{\PPC}(-d)$ is the $-d$-th Serre twist of the sheaf $\Omega^{g-1}_{\PPC}$ of holomorphic $(g-1)$-forms on $\PPC$, and $\pi_{\C}^{-1}$ is the inverse image functor of sheaves. 
Furthermore, we define %a sheaf $\msF_d^{\Xi}$ on $\Yo$ as
%\begin{align*}
%\msF_d^{\Xi} := \prod_{Q \in \Xi} \msF_d
%\end{align*}
%to be the product of the copies of $\msF_d$ over the set $\Xi$ of irreducible matrices of $GL_g(\Q)$. 
\begin{align*}
\msF_d^{\Xi} := \underline{\Hom} \big(\underline{\Z[\Xi]}, \msF_d \big) \simeq  \prod_{Q \in \Xi} \msF_d,  
\end{align*}
%where $\Z[\Xi]$ is the free abelian group generated by $\Xi$ regarded as a constant sheaf on $\Yo$, and $\underline{\Hom}$ is the sheaf $\Hom$. 
where $\underline{\Z[\Xi]}$ is the constant sheaf associated to the free abelian group $\Z[\Xi]$ generated by the set $\Xi$ of irreducible matrices of $GL_g(\Q)$, and $\underline{\Hom}$ is the sheaf $\Hom$. 
For $Q \in \Xi$, let 
\begin{align}\label{eqn:ev}
\ev_Q: \msF_d^{\Xi} \ra \msF_d
\end{align}
denote the evaluation map at $Q$. See Remark \ref{rmk:sheaf1} below. 

\begin{rmk}\label{rmk:sheaf1}
\begin{enumerate}
\item 
More generally, for a sheaf $\msF$ (of abelian groups) on $\Yo$, we define
\begin{align*}
\msF^{\Xi} := \underline{\Hom} \big(\underline{\Z[\Xi]}, \msF \big).  % \simeq  \prod_{Q \in \Xi} \msF_d,  
\end{align*}
Note that for an open subset $U \subset \Yo$, we have
\begin{align*}\label{eqn:FXi map}
\Gamma \Big(U,  \underline{\Hom} \big(\underline{\Z[\Xi]}, \msF \big) \Big)
&=\Hom \big( \underline{\Z[\Xi]}|_U, \msF|_U \big) \notag \\
&= \Hom \big(\Z[\Xi], \Gamma(U, \msF)\big) \\
&= \Map \big(\Xi, \Gamma(U, \msF)\big). \notag
\end{align*}
Then the evaluation map $\ev_Q: \msF^{\Xi} \ra \msF$ is given by
\[
\ev_Q: \Gamma (U,  \msF^{\Xi})= \Map \big(\Xi, \Gamma(U, \msF)\big) \ra \Gamma(U, \msF); \phi \mapsto \phi(Q). 
\]
\item By (1) we also see that $\msF^{\Xi} \simeq \prod_{Q \in \Xi} \msF$. 
\item The sheaf $\msF_d^{\Xi}$ is an analogue of the group $\mcN$ considered in \cite{CDG15:Int}. 
\end{enumerate}
\end{rmk}

\begin{rmk}\label{rmk:sheaf}
%\begin{enumerate}
%\item %It is well known that
The sections of the sheaf $\Omega^{g-1}_{\PPC}(-d)$ on an open subset $U \subset \PPC$ can be described as follows. First, let $\omega$ be a holomorphic $(g-1)$-form on $\C^g \setm \{0\}$ defined by 
\begin{align*}
\omega(y_1, \dots, y_g):= \sum_{i=1}^g (-1)^{i-1} y_i dy_1 \wedge \cdots \wedge \check{dy_i} \wedge \cdots \wedge dy_g 
\end{align*}
for $y=\tp{(y_1, \dots, y_g)} \in \C^g \setm \{0\}$, where $\check{dy_i}$ means that $dy_i$ is omitted. Then we have
\begin{align}\label{eqn:hol forms}
\Gamma \big(U, \Omega^{g-1}_{\PPC}(-d) \big) \simeq \Bigg\{ f\omega \ \Bigg| \
\begin{split}
f: \text{ holomorphic function on }\pi_{\C}^{-1}(U) \\
\text{s.t. } f(\lambda y)=\lambda^{-g-d} f(y),\ \forall \lambda \in \C^{\times}
\end{split}
\Bigg\}.
\end{align}
%In this paper we may regard this as a definition of the sheaf $\Omega^{g-1}_{\PPC}(-d)$. 
%See Appendix \ref{}. 
In this paper we use this as a definition of the sheaf $\Omega^{g-1}_{\PPC}(-d)$. 
%\marginpar{appendix?}
%
%
%
%\item The sheaf $\msF_d^{\Xi}$ can be written as
%\begin{align*}
%\msF_d^{\Xi} = \underline{\Hom} \big(\underline{\Z[\Xi]}, \msF_d \big), 
%\end{align*}
%where $\underline{\Z[\Xi]}$ is the constant sheaf associated to the free abelian group $\Z[\Xi]$ generated by $\Xi$, and $\underline{\Hom}$ is the sheaf $\Hom$. 
%\item 
%Note that for an open subset $U \subset \Yo$, we have
%\end{enumerate}
\end{rmk}

The sheaf $\Omega^{g-1}_{\PPC}(-d)$ has a natural $GL_g(\Q)$-equivariant structure via the pull-back of differential forms.
Since $\pi_{\C}$ is a $GL_g(\Q)$-equivariant map, this induces $GL_g(\Q)$-equivariant structures on $\msF_d, \msF_d^{\Xi}$. We describe these $GL_g(\Q)$-equivariant structures more explicitly in Section \ref{subsect:equiv str}. 

%The aim of Section \ref{sect:sp and sh} and Section \ref{} is to compute the 

%\subsection{Some vanishing results}
\subsection{A vanishing result}

The aim of this subsection is to compute the cohomology group $H^q \big(U,  \pi_{\C}^{-1} \Omega^{g-1}_{\PPC}(-d) \big)$ for convex open subsets $U \subset \C^g \setm \{0\}$. Actually, we will show that 
%In this subsection we compute the cohomology groups $H^q \big(U,  \pi_{\C}^{-1} \Omega^{g-1}_{\PPC}(-d) \big)$ for convex open subsets $U \subset \C^g \setm \{0\}$. Actually we show that 
\[
H^q \big(U,  \pi_{\C}^{-1} \Omega^{g-1}_{\PPC}(-d) \big)=0 %  \text{ for } q \geq 1.
\]
for $q \geq 1$, and also give an explicit description of $H^0 \big(U,  \pi_{\C}^{-1} \Omega^{g-1}_{\PPC}(-d) \big)$. 

Let
\begin{align*}
\bbD:= \{z \in \C \mid \re (z) >0\}
\end{align*}
be the right-half plane.
We start with the following elementary lemma.
\begin{lem}\label{section lemma}
Let $X$ be a paracompact manifold, and let $\pr_1: X \times \bbD \ra X$ be the first projection. Let $U \subset X\times \bbD$ be an open subset such that for any $x \in X$, the set
\begin{align*}
\{ z \in \bbD \mid (x,z) \in U\}
\end{align*}
is a non-empty convex subset of $\bbD$. Then there exists a continuous section $s: X \ra U$ of $\pr_1|_U:U\ra X$ such that $s \circ \pr_1$ is homotopic to the identity map $\id_U$ over $X$, i.e., there exists a continuous map
\[
h: [0,1] \times U \ra U
\]
such that $h(0, u)=s \circ \pr_1(u)$, $h(1, u)=u$, and $\pr_1 \circ h(t,u)=\pr_1(u)$ for $t \in [0,1], u \in U$. 
\end{lem}
\begin{proof}
%Although this lemma seems to be a well-known fact, here we give a direct proof.
%
In order to construct a section, it suffices to construct a continuous map
\[
f: X \ra \bbD
\]
such that $(x, f(x)) \in U$ for all $x \in X$. 
First, by the assumption, for each $x \in X$ we can take $z_x \in \bbD$ such that $(x, z_x) \in U$. Then there exist an open neighborhood $U_x \subset X$ of $x$ and an open neighborhood $V_x \subset \bbD$ of $z_x$ such that $U_x \times V_x \subset U$. 
Since $X=\bigcup_{x \in X} U_x$ and $X$ is paracompact, there exists a subset $\Lambda \subset X$ such that $\{U_{\lambda}\}_{\lambda \in \Lambda}$ is a locally finite open covering of $X$. 
Note that for $x \in U_{\lambda}$, we have 
\[
(x,z_{\lambda}) \in U_{\lambda} \times V_{\lambda} \subset U. 
\]

By using the paracompactness once again, there exists a partition of unity with respect to the open covering $\{U_{\lambda}\}_{\lambda \in \Lambda}$, i.e., a collection $\{\phi_{\lambda}\}_{\lambda \in \Lambda}$ of continuous maps
\[
\phi_{\lambda}: X \ra [0,1]
\]
such that $\supp (\phi_{\lambda}) \subset U_{\lambda}$ and $\sum_{\lambda \in \Lambda} \phi_{\lambda}(x) =1$ for all $x \in X$.
Put
\begin{align*}
f:=\sum_{\lambda \in \Lambda} z_{\lambda}\phi_{\lambda}: X \ra \bbD.
\end{align*}
Then by the convexity assumption, we see that
\[
(x, f(x)) =\left( x, \sum_{\lambda \in \Lambda} z_{\lambda}\phi_{\lambda}(x) \right) \in U
\]
for all $x \in X$. Thus we obtain a section
\[
s: X \ra U; x \mapsto (x, f(x)).
\]
Again by the convexity assumption, we see that $s \circ \pr_1$ is homotopic to the identity map $\id_U$ over $X$. 
Indeed, 
\[
h: [0,1] \times U \ra U; \  (t, (x, z)) \mapsto (x, tz+(1-t)f(x))
\]
gives a homotopy between $s \circ \pr_1$ and $\id_U$ over $X$. 
\end{proof}

\begin{lem}\label{section lemma 2}
Let $U \subset \C^g \setm \{0\}$ be a convex open subset.
%Then the projection $\pi_{\C}|_U: U \ra \pi_{\C}(U)$ has a continuous section $s: \pi_{\C}(U) \ra U$ such that $s \circ \pi_{\C}|_U$ is homotopic to the identity map $\id_U$ on U.
\begin{enumerate}
\item There exists $x \in \C^g \setm \{0\}$ such that $U \subset V_x:=\{ y \in \C^g \setm \{0\} \mid \re (\brk{x,y}) >0 \}$.
\item The projection $\pi_{\C}|_U: U \ra \pi_{\C}(U)$ has a continuous section $s: \pi_{\C}(U) \ra U$ such that $s \circ \pi_{\C}|_U$ is homotopic to the identity map $\id_U$ over $\pi_{\C}(U)$. 
%20210106
\item The image $\pi_{\C}(U)$ is a Stein manifold. 
\end{enumerate}
\end{lem}
\begin{proof}
(1) By the so-called hyperplane separation theorem \cite[Theorem 3.4 (a)]{R91:Fun} applied to $U$ and $\{0\}$, there exist $x \in \C^g \setm \{0\}$ and $\mu \in \R$ such that
\begin{align*}
0= \re (\brk{x, 0}) \leq \mu <\re (\brk{x, y})
\end{align*}
for all $y \in U$, hence $U \subset V_x=\{ y \in \C^g \setm \{0\} \mid \re (\brk{x,y}) >0 \}$.

(2) We first construct a section $s_x: \pi_{\C}(V_x) \ra V_x$ of $\pi_{\C}|_{V_x}$ as follows. 
%For $z \in \pi_{\C}(V_x)$, we define $s_x(z) \in V_x$ to be the unique element such that $\pi_{\C}(s_x(z))=z$ and $\brk{x, s_x(z)}=1$. Clearly, $s_x$ is a continuous map.
%0106
Set
\[
V_x^1:=\{ y \in \C^g \setm \{0\} \mid \brk{x,y} =1 \} \subset V_x. 
\]
Then we easily see that $\pi_{\C}|_{V_x^1}: V_x^1 \isomto \pi_{\C}(V_x)$ is a biholomorphism. Thus we define 
\[
s_x:=(\pi_{\C}|_{V_x^1})^{-1}: \pi_{\C}(V_x) \isomto V_x^1 \subset V_x
\]
to be the inverse map of $\pi_{\C}|_{V_x^1}$, which is clearly a section of $\pi_{\C}|_{V_x}$. 
Then we have the following trivialization of $\pi_{\C}|_{V_x}$:
\begin{align*}
\xymatrix{
\varphi: \pi_{\C}(V_x)\times \bbD \ar[r]^(0.65){\sim} \ar[dr]_{\pr_1}  & V_x \ar[d]^{\pi_{\C}|_{V_x}}; & \hspace{-12mm} (z, \lambda) \mapsto \lambda s_x(z) \\
&\pi_{\C}(V_x)&
}
\end{align*}
Therefore, it suffices to construct a continuous section $s'$ of
\begin{align*}
p:= \pr_1|_{\varphi^{-1}(U)}: \varphi^{-1}(U) \stackrel{\pr_1}{\longrightarrow} \pi_{\C}(U).
\end{align*}
such that $s' \circ p$ is homotopic to $\id_{\varphi^{-1}(U)}$ over $\pi_{\C}(U)$. 
By Lemma \ref{section lemma}, it suffices to show the following:
\begin{claim}
For any $z \in \pi_{\C}(U)$, the set
\begin{align*}
%\{\lambda \in \bbD \mid (z,\lambda) \in \varphi^{-1}(U) \}
%20210106
\bbD_z:= \{\lambda \in \bbD \mid (z,\lambda) \in \varphi^{-1}(U) \}
\end{align*}
is a non-empty convex subset of $\bbD$.
\end{claim}
\begin{proof}[Proof of Claim]
Let $z \in \pi_{\C}(U)$. The set $\bbD_z$ is obviously non-empty. 
Suppose $\lambda, \lambda' \in \bbD_z$, 
%\[
%(z,\lambda), (z, \lambda') \in \varphi^{-1}(U),  
%\]
i.e., $\lambda s_x(z), \lambda' s_x(z) \in U$. 
%0106
%We must show that $(z, t\lambda +(1-t)\lambda') \in \varphi^{-1}(U)$ for $t\in [0,1]$. 
%Indeed, for $t\in [0,1]$, we have $(t\lambda +(1-t)\lambda') s_x(z) \in U$ because $U$ is convex, and hence $(z, t\lambda +(1-t)\lambda') \in \varphi^{-1}(U)$. %This shows the claim. 
Then for $t\in [0,1]$, we have $(t\lambda +(1-t)\lambda') s_x(z) \in U$ because $U$ is convex, and hence $t\lambda +(1-t)\lambda' \in \bbD_z$. %$(z, t\lambda +(1-t)\lambda') \in \varphi^{-1}(U)$. %This shows the claim. 
\end{proof}
%The following elementary fact completes the proof.
%

%(3) Since $\pi_{\C}$ is an open map, the above argument shows that $\pi_{\C}(U)$ is a connected open subset of 
%(3) This follows from \cite[Proposition 4.6.3, Theorem 4.6.8]{MR2311920}. 
%Indeed, since $U$ is convex it is $\C$ convex. Therefore, by applying \cite[Theorem 4.6.8]{MR2311920} to $X=U$, $L(y)=\brk{x,y}$, $z_0=0$, we find that $\pi_{\C}(U)$ is linearly convex open subset of 
%\[
%\pi_{\C}(V_x) \simeq V_x^1 \simeq \C^{g-1}. 
%\]
(3) From the above argument, we see that $\pi_{\C}(U)$ is an open subset of 
\[
\pi_{\C}(V_x) \simeq V_x^1 \simeq \C^{g-1}. 
\]
Since every pseudoconvex open subset of $\C^{g-1}$ is a Stein manifold, cf. \cite[Theorem 4.2.8, Example after Definition 5.1.3]{MR0344507}, it suffices to see that $\pi_{\C}(U)$ is pseudoconvex. 
This follows, for example, from \cite[Proposition 4.6.3, Theorem 4.6.8]{MR2311920}. (Use \cite[Theorem 4.6.8]{MR2311920} for $X=U$, $z_0=0$, and $L(y)=\brk{x,y}$. Note that a convex set $U$ is obviously $\C$ convex.)
%20210113
\end{proof}

%\begin{rmk}
%As for the proof of  Lemma \ref{section lemma 2} (3), it is also possible to prove directly that $\pi_{\C}(U)$ is a domain of holomorphy (or equivalently, pseudoconvex) using \cite[Theorem 3.4 (a)]{R91:Fun}. 
%\end{rmk}

\begin{prop}\label{local vanishing}
Let $U \subset \C^g \setm \{0\}$ be a convex open subset.
\begin{enumerate}
\item The natural map
\[
H^q \big(\pi_{\C}(U), \Omega^{g-1}_{\PPC}(-d) \big) \isomto H^q \big(U, \pi_{\C}^{-1} \Omega^{g-1}_{\PPC}(-d) \big)
\]
 is an isomorphism for all $q \geq 0$. 
%In the following we identify these cohomology groups via this isomorphism.  
% \item Under the above isomorphism, we have
 \item Under this identification, we have
 \begin{align*}
 \Gamma \big(U, \pi_{\C}^{-1} \Omega^{g-1}_{\PPC}(-d) \big)
%\simeq
=
 \Bigg\{ f\omega \ \Bigg| \
\begin{split}
f: \text{ holomorphic function on } \pi_{\C}^{-1}(\pi_{\C}(U)) \\
\text{s.t. } f(\lambda y)=\lambda^{-g-d} f(y),\ \forall \lambda \in \C^{\times}
\end{split}
\Bigg\}. 
 \end{align*}
 \item For all $q\geq 1$, we have
\[
H^q \big(U, \pi_{\C}^{-1} \Omega^{g-1}_{\PPC}(-d) \big)=0.
\]
\end{enumerate}
\end{prop}
\begin{proof}
(1) follows from Lemma \ref{section lemma 2} (2) and \cite[Corollary 2.7.7 (ii)]{KS90:She}.

(2) follows directly from (1) and the description of $\Omega^{g-1}_{\PPC}(-d)$, cf. Remark \ref{rmk:sheaf}.

(3) By Lemma \ref{section lemma 2} (3), we know that $\pi_{\C}(U)$ is a Stein manifold. 
Moreover, $ \Omega^{g-1}_{\PPC}(-d)$ is a coherent sheaf on $\PPC$. 
Therefore, (3) follows from (1) and Cartan's Theorem B, cf. \cite[Theorem 7.4.3]{MR0344507}. %, because $ \Omega^{g-1}_{\PPC}(-d)$ is a coherent sheaf. 
%By Lemma \ref{section lemma 2} (3), we know that $\pi_{\C}(U)$ is a Stein manifold. Therefore, (3) follows from (1) and Cartan's Theorem B, cf. \cite[Theorem 7.4.3]{MR0344507}
%As $\pi_{\C}(U)$ is a Stein manifold, cf. Lemma \ref{section lemma 2},  and $ \Omega^{g-1}_{\PPC}(-d)$ is a coherent sheaf.
\end{proof}

\subsection{$GL_g(\Q)$-equivariant structures}\label{subsect:equiv str}

In this subsection we explicitly describe the $GL_g(\Q)$-  \!\!\! equivariant structures on $\msF_d, \msF_d^{\Xi}$. 
%In this subsection we directly define the $GL_g(\Q)$-  \!\!\! equivariant structures on $\msF_d, \msF_d^{\Xi}$. 
%In this subsection we define the $GL_g(\Q)$-equivariant structures on $\msF_d, \msF_d^{\Xi}$. 

%Let $\msF$ be a sheaf (of abelian groups) on $\Yo$, and let $G \subset GL_g(\Q)$ be a subgroup. 
%In this paper we define a $G$-equivariant structure on $\msF$ to be a collection $\{[\gamma]\}_{\gamma \in G}$ of isomorphisms %
In this paper, for a subgroup $G \subset GL_g(\Q)$ and a sheaf $\msF$ (of abelian groups) on $\Yo$, we define a $G$-equivariant structure on $\msF$ to be a collection $\{[\gamma]\}_{\gamma \in G}$ of isomorphisms 
\[
[\gamma]: \msF \isomto (\tp{\gamma})_{\ast} \msF
\]
subject to the conditions
\begin{enumerate}
\item[(i)] $[1] = \id_{\msF}$, 
\item[(ii)] $[\gamma_1\gamma_2]=(\tp{\gamma_2})_{\ast}[\gamma_1] \circ [\gamma_2], \ \forall \gamma_1, \gamma_2 \in G$. %
\end{enumerate}
Here, $\tp{\gamma}$ is the transpose matrix of $\gamma$, and $(\tp{\gamma})_{\ast} \msF$ (resp. $(\tp{\gamma_2})_{\ast}[\gamma_1]$) is the direct image of $\msF$ (resp. $[\gamma_1]$) with respect to the map $\tp{\gamma}: \Yo \ra \Yo$ (resp. $\tp{\gamma_2}: \Yo \ra \Yo$).

%The trivial examples which we use later are the constant sheaves $\C$ and $\Z[\Xi]$
%For example if $M$ is a $G$-module, then the associated constant sheaf $\underline{M}$ has a natural $G$-equivariant structure
%\[
%[\gamma]: \underline{M} \isomto (\tp{\gamma})_{\ast} \underline{M}; m \mapsto \gamma m. 
%\]

%For example the conjugate action of $GL_g(\Q)$ on $\Xi$ induces a $GL_g(\Q)$-equivariant structure on the constant sheaf $\underline{\Z[\Xi]}$. More concretely $\{[\gamma]\}_{\gamma \in GL_g(\Q)}$ is given by %we have
%\[
%[\gamma]: \Gamma \big(U, \underline{\Z[\Xi]} \big)=\Z[\Xi] \isomto \Gamma \big(\tp{\gamma}^{-1}U, \underline{\Z[\Xi]} \big)=\Z[\Xi]; Q \mapsto [\gamma](Q)=\gamma Q \gamma^{-1}
%\]
%for a connected open subset $U \subset \Yo$, $Q \in \Xi$.  

%First, note that if $\msF$ is a $G$-eqivariant sheaf, the sheaf
%First, note that a $G$-equivariant structure $\{[\gamma]\}_{\gamma \in G}$ on $\msF$ together with the conjugate action of $GL_g(\Q)$ on $\Xi$ naturally induces a $G$-equivariant structure on 
%\[
%\msF^{\Xi}:=\prod_{Q \in \Xi} \msF
%\]
%as follows. 

%Now, the $GL_g(\Q)$-equivariant structure on $\msF_d$ can be defined as follows. 
The $GL_g(\Q)$-equivariant structure on $\msF_d$ can be defined as follows. 
First, by Proposition \ref{local vanishing} (2), we have
 \begin{align*}
 \Gamma(U, \msF_d)=
 \Bigg\{ f \omega \ \Bigg| \
\begin{split}
f: \text{ holomorphic function on } \pi_{\C}^{-1}(\pi_{\C}(U)) \\
\text{s.t. } f(\lambda y)=\lambda^{-g-d} f(y),\ \forall \lambda \in \C^{\times}
\end{split}
\Bigg\}
\end{align*}
for a convex open subset $U \subset \Yo$, where
\[
\omega(y_1, \dots, y_g):= \sum_{i=1}^g (-1)^{i-1} y_i dy_1 \wedge \cdots \wedge \check{dy_i} \wedge \cdots \wedge dy_g.
\]

\begin{lem}\label{action on omega}
For $\gamma \in GL_g(\Q)$, we have
\[
\omega(\gamma y) =\det (\gamma) \omega(y).
\]
\end{lem}
\begin{proof}
It suffices to prove the identity for elementary matrices $\gamma$. The case of elementary matrices can be checked easily, and we omit the proof. %note on 2020/3/13 
\end{proof}

\begin{dfn}\label{dfn gamma on U}
For $\gamma \in GL_g(\Q)$ and a convex open subset $U \subset \Yo$, let $[\gamma]_U$ denote the pull-back map
\begin{align*}
[\gamma]_U: \ &\Gamma(U, \msF_d) \stackrel{\sim}{\longrightarrow} \Gamma(U, (\tp{\gamma})_{\ast}\msF_d) = \Gamma((\tp{\gamma})^{-1} U, \msF_d); \\
&f(y)\omega(y) \longmapsto f(\tp{\gamma}y)\omega(\tp{\gamma}y)=\det (\gamma) f(\tp{\gamma}y)\omega(y).
\end{align*}
Here $f(\tp{\gamma}y)$ is regarded as a holomorphic function of $y \in  (\tp{\gamma})^{-1}\pi_{\C}^{-1}(\pi_{\C}(U)) =  \pi_{\C}^{-1}(\pi_{\C}((\tp{\gamma})^{-1}U))$. %Since convex open subsets form a basis open subsets of $\Yo$
We may drop the subscript $U$ and write as $[\gamma]=[\gamma]_U$ if there is no confusion. 
\end{dfn}

\begin{lem}\label{lem:equiv str}
\begin{enumerate}
\item Let $V, U\subset \Yo$ be convex open subsets such that $V\subset U$, and let $s \in \Gamma(U, \msF_d)$ be a section. Then we have
\[
[\gamma]_U(s)|_V=[\gamma]_V(s|_V).
\]
\item The collection $\{[\gamma]_U \mid U\subset \Yo: \text{ convex open}\}$ defines an isomorphism of sheaves
\[
[\gamma]: \msF_d \stackrel{\sim}{\longrightarrow} (\tp{\gamma})_{\ast}\msF_d.
\]
\item %The collection $\{[\gamma]\}_{\gamma \in GL_g(\Q)}$ satisfies the conditions (i), (ii) of an equivariant structure. %In particular, this defines a $GL_g(\Q)$-equivariant structure on $\msF_d$.
The collection $\{[\gamma]\}_{\gamma \in GL_g(\Q)}$ defines a $GL_g(\Q)$-equivariant structure on $\msF_d$. 
\end{enumerate}
\end{lem}
\begin{proof}
(1) is clear, and (2) follows from (1) since convex open subsets form a basis of open subsets of $\Yo$.
We prove (3). The condition (i) is obvious.

%Let $U\subset \Yo$ be a convex open subset, and let $s=f \omega\in \Gamma(U, \msF_d)$ be a section. Then for $\gamma_1, \gamma_2 \in GL_g(\Q)$ and $y \in \pi_{\C}^{-1}(\pi_{\C}(U))$, we have
Let $U\subset \Yo$ be a convex open subset, and let $s(y)=f(y) \omega(y)\in \Gamma(U, \msF_d)$ be a section. Then for $\gamma_1, \gamma_2 \in GL_g(\Q)$, we have
\begin{align*}
(\tp{\gamma_2})_{\ast}[\gamma_1] \circ [\gamma_2](s(y))
&= [\gamma_1]_{\tp{\gamma_2}^{-1}U}\circ [\gamma_2]_U (s(y)) \\
%&=(\tp{\gamma_2})_{\ast}[\gamma_1] \left( [\gamma_2]_U(s(y))\right) \\
%&=(\tp{\gamma_2})_{\ast}[\gamma_1] (s(\tp{\gamma_2}y)) \\
&=[\gamma_1]_{\tp{\gamma_2}^{-1}U} (s(\tp{\gamma_2}y)) \\
&=s(\tp{\gamma_2}\tp{\gamma_1}y) \\
&=[\gamma_1\gamma_2](s(y)).
\end{align*}
%This completes the proof.
%The compatibility of $[\gamma]_U$ with inclusion maps $V \subset U$ is clear. Now, since convex open subsets form a basis open subsets of $\Yo$, the collection $\{[\gamma]_U\}_U$ defines an isomorphism of sheaves.
%
%The conditions (i), (ii) follow directly from the definition of $[\gamma]_U$.  
%This shows the condition (ii), and completes the proof. 
Since convex open subsets form a basis of open subsets of $\Yo$, this shows the condition (ii). 
\end{proof}

%\begin{dfn}\label{dfn:equ}
%We define a $GL_g(\Q)$-equivariant structure on $\msF_d$ to be the collection $\{[\gamma]\}_{\gamma \in GL_g(\Q)}$ in Lemma \ref{lem:equiv str}.
%The $GL_g(\Q)$-equivariant structure on $\msF_d$ is defined to be the collection $\{[\gamma]\}_{\gamma \in GL_g(\Q)}$ in Lemma \ref{lem:equiv str}.
%\end{dfn}

Hence we obtain the $GL_g(\Q)$-equivariant structure on $\msF_d$. 
Next we describe the $GL_g(\Q)$-equivariant structure on $\msF_d^{\Xi}$. First, note that the conjugate action %of $GL_g(\Q)$ on $\Xi$:
\[
[\gamma]: \Z[\Xi] \isomto \Z[\Xi]; Q \mapsto [\gamma](Q)=\gamma Q\gamma^{-1} %,\ Q \in \Xi, \gamma \in GL_g(\Q)
\]
%for $\gamma \in \Xi, Q \in \Xi$ naturally induces a $GL_g(\Q)$-equivariant structure on the associated constant sheaf $\underline{\Z[\Xi]}$. 
of $GL_g(\Q)$ on $\Z[\Xi]$ naturally induces a $GL_g(\Q)$-equivariant structure on the associated constant sheaf $\underline{\Z[\Xi]}$. 
%Taking the Remark \ref{rmk:sheaf1} into account, we see that $\msF_d^{\Xi}=\underline{\Hom} \big(\underline{\Z[\Xi]}, \msF_d \big)$ has a natural $GL_g(\Q)$-equivariant structure induced from those of $\underline{\Z[\Xi]}$ and $\msF_d$. %, i.e., for an open subset $U \subset \Yo$, and a section $\phi \in \Gamma(U, \msF_d^{\Xi})=\Hom$
%%%By Remark \ref{rmk:sheaf1}, we see that $\msF_d^{\Xi}=\underline{\Hom} \big(\underline{\Z[\Xi]}, \msF_d \big)$ has a natural $GL_g(\Q)$-equivariant structure induced from those of $\underline{\Z[\Xi]}$ and $\msF_d$. %, i.e., for an open subset $U \subset \Yo$, and a section $\phi \in \Gamma(U, \msF_d^{\Xi})=\Hom$
%
%
%
Therefore, for a $GL_g(\Q)$-equivariant sheaf $\msF$, the sheaf
\[
\msF^{\Xi} =\underline{\Hom}\big(\underline{\Z[\Xi]}, \msF \big) 
\]
has a natural $GL_g(\Q)$-equivariant structure induced from those of $\underline{\Z[\Xi]}$ and $\msF$.  In particular, we obtain a $GL_g(\Q)$-equivariant structure on $\msF_d^{\Xi}$. 

More concretely, for an open subset $U \subset \Yo$ and a section 
\[
\phi \in \Gamma(U, \msF^{\Xi}) = \Map (\Xi, \Gamma(U,\msF)), 
\]
cf. Remark \ref{rmk:sheaf1}, the $GL_g(\Q)$-equivariant structure on $\msF^{\Xi}$ can be computed as % is given by
\[
[\gamma](\phi)(Q)=[\gamma](\phi([\gamma^{-1}](Q))=[\gamma](\phi(\gamma^{-1} Q \gamma))
\]
for $\gamma \in GL_g(\Q)$ and $Q \in \Xi$. 
%
%For $Q \in \Xi$, the evaluation map
In particular, we see that for $Q \in \Xi$, the evaluation map
\[
\ev_Q: \msF^{\Xi} \ra \msF, 
\]
cf. Remark \ref{rmk:sheaf1}, is a $\Gamma_Q$-equivariant map, where $\Gamma_Q=\Stab_{SL_g(\Z)}(Q) \subset SL_g(\Z)$ is the stabilizer of $Q \in \Xi$ in $SL_g(\Z)$.

%%%%%%%%%%%%%%%%%%%%%%%%%%%%%%%%%%%%%%%%%%%%%%%%%%%%%%%%%%%%%%%%%%%%%%%%%%%%%%%%%%%%%%%%%%%%%%%%%%%%%%%%%%%%%%%%%%%%%%%%
\section{Equivariant cohomology}\label{sect:equiv coh}
%review equiv str for general sheaf (moved to section sp and sh)

%cech
%%for general sheaf
%%(injective) resolution (for Fd)
Recall that $\Gamma_Q=\Stab_{SL_g(\Z)}(Q) \subset SL_g(\Z)$ denotes the stabilizer of $Q \in \Xi$ in $SL_g(\Z)$. 
In this section we compute the equivariant cohomology groups  
\begin{align*}
&H^q(\Yo, \Gamma_Q, \msF_d), \\
&H^q(\Yo, SL_g(\Z), \msF_d^{\Xi})
\end{align*}
using the equivariant \v{C}ech complexes, cf. Corollary \ref{cor:equiv coh}. 
We closely follow the argument in \cite{BHYY19:Can}.

Here, for a subgroup $G \subset GL_g(\Q)$, the equivariant cohomology
\[
H^q(\Yo, G, - ): \mathbf{Sh}(\Yo, G) \ra \mathbf{Ab}
\]
is defined to be the right derived functor of the $G$-invariant global section functor
\[
\Gamma(\Yo, G, -): \mathbf{Sh}(\Yo, G) \ra \mathbf{Ab}; \ \msF \mapsto \Gamma(\Yo, \msF)^G,
\]
where $ \mathbf{Sh}(\Yo, G)$ is the category of $G$-equivariant sheaves on $\Yo$, $\mathbf{Ab}$ is the category of abelian groups, and $\Gamma(\Yo, \msF)^G$ is the $G$-invariant part of the global section $\Gamma(\Yo, \msF)$.

\subsection{Open covering}

%In this subsection we introduce a special $GL_g(\Q)$-stable open covering $\{V_{\alpha}\}_{\alpha \in \XQ}$ of $\Yo$. The meaning of this open covering will be revealed in Section \ref{}. 
In this subsection we introduce a certain $GL_g(\Q)$-stable open covering of $\Yo$. 
For $\alpha \in \C^g \setm \{0\}$, we define an open subset $V_{\alpha} \subset \C^g \setm \{0\}$ by
\[
V_{\alpha} :=\{y \in \C^g \mid \re (\brk{\alpha,y} ) >0\} \subset \C^g \setm \{0\}. 
\]
%
%Now, let 
Clearly, $V_{\alpha} \subset \C^g \setm \{0\}$ is a convex open subset. 
Let 
\begin{align*}
\XQ := \Q^g \setm \{0\}
%\mcX := \Q^g \setm \{0\}
\end{align*}
denote the set of all non-zero rational vectors on which $GL_g(\Q)$ acts by the matrix multiplication from the left. 
Then we easily see that 
\[
\Yo = \bigcup_{\alpha \in \XQ} V_{\alpha}. 
\]
Let $\mcX_{\Q}:=\{V_{\alpha}\}_{\alpha \in \XQ}$ denote this open covering of $\Yo$.  
For $r \geq 0$, $I=(\alpha_1, \dots, \alpha_r) \in (\XQ)^r$, set
\begin{align*}
V_I := \bigcap_{i=1}^r V_{\alpha_i} =\{y \in \Yo \mid \re (\brk{\alpha_i, y})>0 \ \forall i\} \subset \Yo.
\end{align*}
%We assume $V_{\emptyset}=\Yo$ in the case $r=0$.
In the case $r=0$, we assume $(\XQ)^0=\{\emptyset\}$ and $V_{\emptyset}=\Yo$. 
Let 
\begin{align*}
j_I: V_I \hra \Yo
\end{align*}
denote the inclusion map.

First, we show that $\mcX_{\Q}=\{V_{\alpha}\}_{\alpha \in \XQ}$ is a $GL_g(\Q)$-stable open covering. 
Note that the group $GL_g(\Q)$ acts diagonally on $(\XQ)^r$. For $I=(\alpha_1, \dots, \alpha_r) \in (\XQ)^r$ and $\gamma \in GL_g(\Q)$, let
\[
\gamma I = (\gamma \alpha_1, \dots, \gamma \alpha_r) \in (\XQ)^r
\]
denote this diagonal action of $\gamma$ on $I$.
\begin{lem}\label{lem stable cov}
For $r\geq 0$, $I=(\alpha_1, \dots, \alpha_r) \in (\XQ)^r$, and $\gamma \in GL_g(\Q)$, we have
\[
V_{\gamma I} = \tp{\gamma}^{-1}V_I.
\]
In other words, we have the following commutative diagram:
\begin{align*}
\xymatrix{
V_I \  \ar@{^{(}->}[r]^{j_I} \ar[d]_{\tp{\gamma}^{-1}}^{\wr} & \Yo \ar[d]^{\tp{\gamma}^{-1}}\\
V_{\gamma I} \  \ar@{^{(}->}[r]^{j_{\gamma I}} &\Yo
}
\end{align*}
%In particular, the open covering $\mcX_{\Q}=\{V_{\alpha}\}_{\alpha \in \XQ}$ is a $GL_g(\Q)$-stable open covering. 
\end{lem}

\begin{proof}
For $y \in \Yo$, we have $y \in V_{\gamma I}$ if and only if
\[
0< \re (\brk{\gamma \alpha_i, y}) = \re (\brk{\alpha_i, \tp{\gamma} y})
\]
for all $i \in \{1, \dots, r\}$. This proves the lemma.
\end{proof}

%%%%%%%%%%%%%%%%%%%%%%%%%%%%%%%%%%%%%%%%%%%%%%%%%%%%%%%%%%%%%%%%%%%%%%%%%%%%%%%%%%%%%%%%%%%%%%%%%%%%%%%%%%%%%%%%%%%%%%%%

\subsection{The equivariant \v{C}ech complex}\label{subsect:cech cpx}

%\begin{dfn}
%Let $\msF$ be a $GL_g(\Q)$-equivariant sheaf on $\Yo$. We define the $GL_g(\Q)$-equivariant ``sheaf'' \v{C}ech complex
Let $\msF$ be a $GL_g(\Q)$-equivariant sheaf on $\Yo$. We consider the $GL_g(\Q)$-equivariant ``sheaf \v{C}ech complex''
\begin{align*}
\msC^{\bullet}(\mcX_{\Q}, \msF): \msC^0(\mcX_{\Q}, \msF) \stackrel{\dd^0}{\longrightarrow} \msC^1(\mcX_{\Q}, \msF)  \stackrel{\dd^1}{\longrightarrow} \msC^2(\mcX_{\Q}, \msF)  \stackrel{\dd^2}{\longrightarrow}  \cdots
\end{align*}
%associated to $\mcX_{\Q}$ and $\msF$ 
defined as follows. For $q \geq 0$, put %, where
\[
\msC^q(\mcX_{\Q}, \msF):=\prod_{I \in (\XQ)^{q+1}} j_{I \ast}j_I^{-1} \msF,
\]
where $j_{I \ast}$ (resp. $j_I^{-1}$) is the direct image (resp. inverse image) functor induced by the inclusion map $j_I: V_I \hookrightarrow \Yo$.
By Lemma \ref{lem stable cov}, the $GL_g(\Q)$-equivariant structure 
\[
[\gamma]: \msF \isomto (\tp{\gamma})_{\ast} \msF
\]
of $\msF$ induces isomorphisms
\begin{align*}
\begin{cases}
[\gamma]: j_{I \ast}j_I^{-1} \msF \isomto j_{I \ast}j_{I}^{-1} (\tp{\gamma})_{\ast} \msF \simeq  (\tp{\gamma})_{\ast} j_{\gamma I \ast}j_{\gamma I}^{-1} \msF, \\
[\gamma]: \msC^{q}(\mcX_{\Q}, \msF) \isomto (\tp{\gamma})_{\ast} \msC^{q}(\mcX_{\Q}, \msF). 
\end{cases}
\end{align*}
We easily see that this defines a $GL_g(\Q)$-equivariant structure on $\msC^{q}(\mcX_{\Q}, \msF)$. 
%More concretely, for an open subset $U \subset \Yo$, and a section
More concretely, for an open subset $U \subset \Yo$ and a section
\[
s=(s_I)_{I \in (\XQ)^{q+1}} \in \Gamma(U, \msC^q(\mcX_{\Q}, \msF))=\prod_{I \in (\XQ)^{q+1}}\Gamma(U\cap V_I, \msF), 
\]
we have
\begin{align}\label{eqn gamma sI}
%([\gamma](s))_I=[\gamma]_{U\cap V_{\gamma^{-1}I}}(s_{\gamma^{-1} I}),
([\gamma](s))_I=[\gamma](s_{\gamma^{-1} I}),
\end{align}
where $([\gamma](s))_I$ is the $I$-th component of $[\gamma](s)$.

%and the differential 
The differential map
\[
\dd^q: \msC^q(\mcX_{\Q}, \msF) \ra \msC^{q+1}(\mcX_{\Q}, \msF)
\]
is given by
\begin{align*}%\label{dfn differential}
(\dd^q(s))_{(\alpha_0, \dots, \alpha_{q+1})} = \sum_{i=0}^{q+1} (-1)^i s_{(\alpha_0, \dots, \check{\alpha_i}, \dots, \alpha_{q+1})}|_{U\cap V_{(\alpha_0, \dots, \alpha_{q+1})}} 
\end{align*}
for an open subset $U \subset \Yo$ and a section $s=(s_I)_{I \in (\XQ)^{q+1}} \in \Gamma(U, \msC^q(\mcX_{\Q}, \msF))$. 
Here $\check{\alpha_i}$ means that $\alpha_i$ is omitted. %, and hence $(\alpha_0, \dots, \check{\alpha_i}, \dots, \alpha_{q+1}) \in (\XQ)^{q+1}$. 
%\[
%s=(s_I)_{I \in (\XQ)^{q+1}} \in \Gamma(U, \msC^q(\mcX_{\Q}, \msF))=\prod_{I \in (\XQ)^{q+1}}\Gamma(U\cap V_I, \msF). 
%\]
Moreover, there is a map
\[
\dd^{-1} : \msF \ra \msC^0(\mcX_{\Q}, \msF)=\prod_{\alpha \in \XQ} j_{\alpha \ast}j_{\alpha}^{-1} \msF
\]
induced by the natural maps $\msF \ra  j_{\alpha \ast}j_{\alpha}^{-1} \msF$. % ($\alpha \in \XQ$).
%\end{dfn}

%The following fact is well-known.
Then we have the following. 
\begin{lem}\label{lem:hom to 0}
\begin{enumerate}
\item For $q \geq -1$, the differential map $\dd^q$ is a $GL_g(\Q)$-equivariant map, i.e., $[\gamma]\circ \dd^q=\dd^q \circ [\gamma]$ for $\gamma \in GL_g(\Q)$. 
\item 
For any $\alpha_0 \in \XQ$, the sequence
\[
0 \longrightarrow \msF \big|_{V_{\alpha_0}} \stackrel{\dd^{-1}}{\longrightarrow} \msC^0(\mcX_{\Q}, \msF)\big|_{V_{\alpha_0}}  \stackrel{\dd^0}{\longrightarrow} \msC^1(\mcX_{\Q}, \msF)\big|_{V_{\alpha_0}}
%\stackrel{\dd^1}{\longrightarrow} \msC^2(\mcX_{\Q}, \msF_d)\big|_{V_{a_0}}
\longrightarrow \cdots
\]
is homotopic to zero.
In particular, the sequence
\[
0 \longrightarrow \msF \stackrel{\dd^{-1}}{\longrightarrow} \msC^0(\mcX_{\Q}, \msF)  \stackrel{\dd^0}{\longrightarrow} \msC^1(\mcX_{\Q}, \msF)
\longrightarrow \cdots
\]
is an exact sequence of $GL_g(\Q)$-equivariant sheaves since $\Yo=\bigcup_{\alpha_0 \in \XQ}V_{\alpha_0}$. 
\end{enumerate}
\end{lem}
\begin{proof}
(1) Let $U \subset \Yo$ be an open subset, and let 
\[
s=(s_I)_{I \in (\XQ)^{q+1}} \in \Gamma(U, \msC^q(\mcX_{\Q}, \msF))=\prod_{I \in (\XQ)^{q+1}}\Gamma(U\cap V_I, \msF)
\]
be a section. 
Let $J=(\alpha_0, \dots, \alpha_{q+1}) \in (\XQ)^{q+2}$, and put $J^{(i)}:=(\alpha_0, \dots, \check{\alpha_i}, \dots, \alpha_{q+1}) \in (\XQ)^{q+1}$ for $i=0, \dots, q+1$. 
%Then, for $J=(\alpha_0, \dots, \alpha_{q+1}) \in (\XQ)^{q+2}$, we have 
Then we have 
\begin{align*}
(\dd^q([\gamma](s)))_{J}
&= \sum_{i=0}^{q+1} (-1)^i  [\gamma](s_{\gamma^{-1} J^{(i)}})|_{\tp{\gamma}^{-1}U\cap V_{J}}  \\
&=\sum_{i=0}^{q+1} (-1)^i  [\gamma](s_{\gamma^{-1} J^{(i)}}|_{U\cap V_{\gamma^{-1}J}} ) \\
&=[\gamma]\left(\sum_{i=0}^{q+1} (-1)^i  s_{\gamma^{-1} J^{(i)}}|_{U\cap V_{\gamma^{-1}J}} \right) \\
&=([\gamma](\dd^q(s)))_J. 
\end{align*}
%where $J^{(i)}:=(\alpha_0, \dots, \check{\alpha_i}, \dots, \alpha_{q+1})$. %This shows (1). 

(2) 
%See \cite[{Lemma 02FU}]{stacks-project}. Although \cite[{Lemma 02FU}]{stacks-project} proves only the exactness of the sequence, its argument essentially shows the statement in this lemma.
See \cite[Th\'eor\`eme 5.2.1]{MR0345092} or \cite[{Lemma 02FU}]{stacks-project}. Although they prove only the exactness of the sequence, we can prove the statement in this lemma using essentially the same argument. 
Cf. also \cite[Lemma 2.8.2, Remark 2.8.3]{KS90:She}. 
\end{proof}
%MR0345092
%\marginpar{appendix?}

%next:
%equivariant structure (before lemma hom to 0?)
%global section, usual cech cpx

%Now recall that $\underline{\Z[\Xi]}$ has a natural $GL_g(\Q)$-equivariant structure induced from the conjugate action of $GL_g(\Q)$ on $\Xi$, and hence we have an additive functor

By applying the additive functor 
\[
\underline{\Hom}\big(\underline{\Z[\Xi]}, -\big): \mathbf{Sh}(\Yo, GL_g(\Q)) \ra  \mathbf{Sh}(\Yo, GL_g(\Q)); \ \msG \mapsto \msG^{\Xi} := \underline{\Hom}\big(\underline{\Z[\Xi]}, \msG \big), 
\]
we obtain the following.  
\begin{cor}\label{lem:FXi resol}
The sequence
\[
0 \longrightarrow \msF^{\Xi} \stackrel{\dd^{-1}}{\longrightarrow} \msC^0(\mcX_{\Q}, \msF)^{\Xi}  \stackrel{\dd^0}{\longrightarrow} \msC^1(\mcX_{\Q}, \msF)^{\Xi}
\longrightarrow \cdots
\]
is an exact sequence of $GL_g(\Q)$-equivariant sheaves. 
\end{cor}

\begin{proof}
Since the homotopy is preserved by the additive functor, by Lemma \ref{lem:hom to 0} (2), we see that for any $\alpha_0 \in \XQ$, the sequence
\[
0 \longrightarrow \msF^{\Xi} \big|_{V_{\alpha_0}} \stackrel{\dd^{-1}}{\longrightarrow} \msC^0(\mcX_{\Q}, \msF)^{\Xi}\big|_{V_{\alpha_0}}  \stackrel{\dd^0}{\longrightarrow} \msC^1(\mcX_{\Q}, \msF)^{\Xi}\big|_{V_{\alpha_0}}
%\stackrel{\dd^1}{\longrightarrow} \msC^2(\mcX_{\Q}, \msF_d)\big|_{V_{a_0}}
\longrightarrow \cdots
\]
is homotopic to zero, and hence exact. %This proves the corollary since $\Yo=\bigcup_{\alpha_0 \in \Xi}V_{\alpha_0}$. 
\end{proof}

%\begin{rmk}
%In this paper we use the above ``unrestricted'' \v{C}ech complex instead of the ``alternating'' \v{C}ech complex consisting of alternating sections. In fact, it is also known that they are homotopically equivalent \cite[{Lemma 01FM}]{stacks-project}.
%\end{rmk}

Now, by taking the global section, set
\begin{align*}
C^q(\mcX_{\Q}, \msF) &:= \Gamma \big(\Yo, \msC^q(\mcX_{\Q}, \msF) \big) = \prod_{I \in (\XQ)^{q+1}} \Gamma (V_I, \msF). %, \\
%C^q(\mcX_{\Q}, \msF)^{\Xi} &:= \Gamma \big(\Yo, \msC^q(\mcX_{\Q}, \msF)^{\Xi} \big)\\
%&= \Map \big(\Xi, C^q(\mcX_{\Q}, \msF) \big). 
\end{align*}
Then we obtain a complex %$C^{\bullet}(\mcX_{\Q}, \msF)$ 
\begin{align*}
%(\Cpx(\mcX_{\Q}, \msF_d), \dd^{\bullet})
C^{\bullet}(\mcX_{\Q}, \msF): C^0(\mcX_{\Q}, \msF) \stackrel{\dd^0}{\longrightarrow} C^1(\mcX_{\Q}, \msF)  \stackrel{\dd^1}{\longrightarrow} C^2(\mcX_{\Q}, \msF)  \stackrel{\dd^2}{\longrightarrow} \cdots
\end{align*}
of $GL_g(\Q)$-modules. Note that this is the usual \v{C}ech complex associated to the open covering $\mcX_{\Q}$. 
Furthermore, set
\begin{align*}
C^q(\mcX_{\Q}, \msF)^{\Xi} := \Gamma \big(\Yo, \msC^q(\mcX_{\Q}, \msF)^{\Xi} \big) = \Map \big(\Xi, C^q(\mcX_{\Q}, \msF) \big). 
\end{align*}
Then we obtain another complex
\begin{align*}
%(\Cpx(\mcX_{\Q}, \msF_d), \dd^{\bullet})
C^{\bullet}(\mcX_{\Q}, \msF)^{\Xi}: C^0(\mcX_{\Q}, \msF)^{\Xi} \stackrel{\dd^0}{\longrightarrow} C^1(\mcX_{\Q}, \msF)^{\Xi}  \stackrel{\dd^1}{\longrightarrow} C^2(\mcX_{\Q}, \msF)^{\Xi}  \stackrel{\dd^2}{\longrightarrow}  \cdots
\end{align*}
of $GL_g(\Q)$-modules. 
For $Q \in \Xi$, the evaluation map
\begin{align}\label{eqn:ev cech}
\ev_Q: C^{\bullet}(\mcX_{\Q}, \msF)^{\Xi} \ra C^{\bullet}(\mcX_{\Q}, \msF)
\end{align}
is a $\Gamma_Q$-equivariant morphism of complexes.

%%%%%%%%%%%%%%%%%%%%%%%%%%%%%%%%%%%%%%%%%%%%%%%%%%%%%%%%%%%%%%%%%%%%%%%%%%%%%%%%%%%%%%%%%%%%%%%%%%%%%%%%%%%%%%%%%%%%%%%%

\subsection{Acyclicity}
%%The aim of this subsection is to prove the acyclicity of the sheaf \v{C}ech complexes $\msC^{\bullet}(\mcX_{\Q}, \msF_d)$ and $\msC^{\bullet}(\mcX_{\Q}, \msF_d)^{\Xi}$, cf. Proposition \ref{prop:acyc2}. 
The aim of this subsection is to prove the acyclicity of the sheaves $\msC^{q}(\mcX_{\Q}, \msF_d)$ and $\msC^{q}(\mcX_{\Q}, \msF_d)^{\Xi}$, cf. Proposition \ref{prop:acyc2}. 
Then we can compute the equivariant cohomology groups $H^q(\Yo, \Gamma_Q, \msF_d)$ and $H^q(\Yo, SL_g(\Z), \msF_d^{\Xi})$ using the \v{C}ech complexes $C^{\bullet}(\mcX_{\Q}, \msF_d)$ and $C^{\bullet}(\mcX_{\Q}, \msF_d)^{\Xi}$, cf. Corollary \ref{cor:equiv coh}. 
%%Then we can compute the equivariant cohomology groups $H^q(\Yo, \Gamma_Q, \msF_d)$ and $H^q(\Yo, SL_g(\Z), \msF_d^{\Xi})$ using the \v{C}ech complexes defined in Section \ref{subsect:cech cpx}. %, cf. Corollary \ref{cor:equiv coh}. 

\begin{lem}\label{lem:acyc0}
Let $r\geq 1$, $I=(\alpha_1, \dots, \alpha_r) \in (\XQ)^r$.
\begin{enumerate}
\item For all $q \geq 1$, we have
\[
H^q(V_I, \msF_d)=0.
\]
\item For all $q \geq 1$, we have
\[
R^q j_{I \ast} \left( j_I^{-1} \msF_d \right)=0,
\]
where $R^q j_{I \ast}$ (resp. $j_I^{-1}$) is the higher direct image (resp. inverse image) functor induced by the inclusion map $j_I: V_I \hookrightarrow \Yo$.
\item For any open subset $U \subset \Yo$ and $q \geq 0$, we have an isomorphism
\[
H^q(U,  j_{I \ast}  j_I^{-1} \msF_d)  \isomto H^q(U\cap V_I, \msF_d).
\]
\end{enumerate}
\end{lem}

\begin{proof}
(1) follows directly from Proposition \ref{local vanishing} (3) since $V_I$ is convex.

(2) Let $x \in \Yo$. Since convex open subsets form a basis of open subsets of $\Yo$, we have
\begin{align*}
\left( R^q j_{I \ast} \left( j_I^{-1} \msF_d \right)\right)_x &= \varinjlim_{\substack{x \in U: \\ \text{convex}}} H^q(U \cap V_I, j_I^{-1}\msF_d) \\
&=\varinjlim_{\substack{x \in U: \\ \text{convex}}} H^q(U\cap V_I, \msF_d)=0.
\end{align*}
Here the last vanishing follows from Proposition \ref{local vanishing} (3). This proves (2).

(3) follows from (2) and the Leray spectral sequence.
%This completes the proof.
\end{proof}

\begin{prop}\label{prop:acyc1}
%\begin{enumerate}
%\item 
For $q \geq 0$, the sheaves $\msC^q(\mcX_{\Q}, \msF_d)$ and $\msC^q(\mcX_{\Q}, \msF_d)^{\Xi}$ are $\Gamma (\Yo, -)$-acyclic, i.e., we have
\begin{align*}
H^p\big(\Yo, \msC^q(\mcX_{\Q}, \msF_d)\big)=0, \\%\ \forall p \geq 1.
H^p\big(\Yo, \msC^q(\mcX_{\Q}, \msF_d)^{\Xi}\big)=0
\end{align*}
for $p \geq 1$. 
%\end{enumerate}
\end{prop}

\begin{proof}
We imitate the argument in \cite[Proposition 3.4, Lemma 3.5]{BHYY19:Can}. 
For $I \in (\XQ)^{q+1}$, put $\msF_I:= j_{I \ast}j_I^{-1} \msF_d$, and let
\[
0 \ra \msF_I \ra \msI_I^{\bullet}
\]
be an injective resolution of $\msF_I$. First we show that
\begin{align}\label{resolution1}
0 \ra \msC^q(\mcX_{\Q}, \msF_d)=\prod_{I \in(\XQ)^{q+1}} \msF_I \ra \prod_{I \in(\XQ)^{q+1}} \msI_I^{\bullet}, %\\
\end{align}
\begin{align}\label{resolution1Xi}
%0 \ra &\msC^q(\mcX_{\Q}, \msF_d)^{\Xi}=\prod_{Q \in \Xi} \prod_{I \in(\XQ)^{q+1}} \msF_I \ra \prod_{\Q \in \Xi}\prod_{I \in(\XQ)^{q+1}} \msI_I^{\bullet}
0 \ra \msC^q(\mcX_{\Q}, \msF_d)^{\Xi}=\left( \prod_{I \in(\XQ)^{q+1}} \msF_I\right)^{\Xi} \ra \left(\prod_{I \in(\XQ)^{q+1}} \msI_I^{\bullet}\right)^{\Xi}
\end{align}
%are both injective resolutions. It is clear that $\prod_{I \in(\XQ)^{q+1}} \msI_I^{\bullet}$ and $\left(\prod_{I \in(\XQ)^{q+1}} \msI_I^{\bullet}\right)^{\Xi} \simeq \prod_{Q \in \Xi}\prod_{I \in(\XQ)^{q+1}} \msI_I^{\bullet}$ are injective sheaves because they are product of injective sheaves. 
are both injective resolutions of $\msC^q(\mcX_{\Q}, \msF_d)$ and $\msC^q(\mcX_{\Q}, \msF_d)^{\Xi}$ respectively. 
It is clear that $\prod_{I \in(\XQ)^{q+1}} \msI_I^{p}$ and 
\begin{align*}
%&\prod_{I \in(\XQ)^{q+1}} \msI_I^{\bullet}, \\
\left(\prod_{I \in(\XQ)^{q+1}} \msI_I^{p}\right)^{\Xi} \simeq \prod_{Q \in \Xi}\prod_{I \in(\XQ)^{q+1}} \msI_I^{p}, 
\end{align*}
are injective sheaves because they are products of injective sheaves, cf. Remark \ref{rmk:sheaf1} (2). 
We must show the exactness of \eqref{resolution1} and \eqref{resolution1Xi}. 
Let $U \subset \Yo$ be any convex open subset. By Lemma \ref{lem:acyc0} (3) and Proposition \ref{local vanishing} (3), we have
\[
H^p(U, \msF_I) \isomto H^p(U \cap V_I, \msF_d) = 0
\]
for $p \geq 1$. Therefore, we find that
\[
0 \ra \msF_I(U) \ra \msI_I^{\bullet}(U)
\]
is exact because $H^p(U, \msF_I)$ is the cohomology of this complex.  
Hence, %we see that %the products
\begin{align*}
0 \ra \prod_{I \in(\XQ)^{q+1}} \msF_I(U) \ra \prod_{I \in(\XQ)^{q+1}} \msI_I^{\bullet}(U), 
\end{align*}
\begin{align*}
%0 \ra \left(\prod_{I \in(\XQ)^{q+1}} \msF_I(U)\right)^{\Xi} \ra \left(\prod_{I \in(\XQ)^{q+1}} \msI_I^{\bullet}(U)\right)^{\Xi}
0 \ra \prod_{Q \in \Xi}\prod_{I \in(\XQ)^{q+1}} \msF_I(U) \ra \prod_{Q \in \Xi}\prod_{I \in(\XQ)^{q+1}} \msI_I^{\bullet}(U)
\end{align*}
are also exact.
Since convex open subsets of $\Yo$ form a basis of open subsets, we obtain the exactness of \eqref{resolution1} and \eqref{resolution1Xi}. 

Then for $p \geq 1$, we have
\begin{align*}
H^p\big(\Yo, \msC^q(\mcX_{\Q}, \msF_d)\big)&\simeq H^p\left(\Gamma\left(\Yo, \prod_{I \in(\XQ)^{q+1}} \msI_I^{\bullet}\right)\right) \\
&\simeq \prod_{I \in(\XQ)^{q+1}}H^p(\Gamma(\Yo, \msI_I^{\bullet})) \\
&\simeq \prod_{I \in(\XQ)^{q+1}}H^p(\Yo, \msF_I) \\
&\simeq \prod_{I \in(\XQ)^{q+1}} H^p(V_I, \msF_d)=0, 
\end{align*}
and similarly, 
\begin{align*}
H^p\big(\Yo, \msC^q(\mcX_{\Q}, \msF_d)^{\Xi}\big) &\simeq H^p\left(\Gamma\left(\Yo, \left(\prod_{I \in(\XQ)^{q+1}} \msI_I^{\bullet}\right)^{\Xi} \right)\right) \\
&\simeq \prod_{Q \in \Xi}\prod_{I \in(\XQ)^{q+1}}H^p(\Gamma(\Yo, \msI_I^{\bullet})) \\
&\simeq \prod_{Q \in \Xi} \prod_{I \in(\XQ)^{q+1}} H^p(V_I, \msF_d)=0. 
\end{align*}
This completes the proof. 
%https://ncatlab.org/nlab/show/chain+homology+and+cohomology#RespectForDirectSum
\end{proof}

\begin{prop}\label{prop:acyc2}
\begin{enumerate}
\item Let $Q \in \Xi$. For $q \geq 0$, the sheaf $\msC^q(\mcX_{\Q}, \msF_d)$ is $\Gamma(\Yo, \Gamma_Q, -)$-acyclic, i.e., we have
\[
H^p \big(\Yo, \Gamma_Q,  \msC^q(\mcX_{\Q}, \msF_d)\big)=0
\]
for $p \geq 1$. In particular, the complex
\begin{align*}
0 \ra \msF_d \stackrel{\dd^{-1}}{\ra} \msC^{\bullet}(\mcX_{\Q}, \msF_d) \label{eqn:resol 2}
\end{align*}
gives a $\Gamma(\Yo, \Gamma_Q, -)$-acyclic resolution of $\msF_d$. 
\item For $q \geq 0$, the sheaf $\msC^q(\mcX_{\Q}, \msF_d)^{\Xi}$ is $\Gamma(\Yo, SL_g(\Z), -)$-acyclic, i.e., we have
\[
H^p\big(\Yo, SL_g(\Z),  \msC^q(\mcX_{\Q}, \msF_d)^{\Xi}\big)=0
\]
for $p \geq 1$. In particular, the complex
\begin{align*}
0 \ra \msF_d^{\Xi} \stackrel{\dd^{-1}}{\ra} \msC^{\bullet}(\mcX_{\Q}, \msF_d)^{\Xi} \label{eqn:resol 2}
\end{align*}
gives a $\Gamma(\Yo, SL_g(\Z), -)$-acyclic resolution of $\msF_d^{\Xi}$.  
\end{enumerate}
\end{prop}

\begin{proof}
(1) First note that the functor $\Gamma(\Yo, \Gamma_Q, -)$ is a composition of two left exact functors $\Gamma(\Yo, -)$ and $(-)^{\Gamma_Q}$. Moreover, $\Gamma(\Yo, -)$ sends injective objects to injective objects. Therefore, we have a spectral sequence
\[
E_2^{ab}=H^a (\Gamma_Q, H^b(\Yo,  \msC^q(\mcX_{\Q}, \msF_d))) \Ra H^{a+b}(\Yo, \Gamma_Q,  \msC^q(\mcX_{\Q}, \msF_d)),
\]
where $H^a(\Gamma_Q, -)$ is the usual group cohomology of $\Gamma_Q$.
Now by Proposition \ref{prop:acyc1}, we already have
\[
H^b(\Yo,  \msC^q(\mcX_{\Q}, \msF_d))=0 , \ \forall b\geq 1.
\]
Therefore, it suffices to show
\[
H^a(\Gamma_Q, \Gamma(\Yo,  \msC^q(\mcX_{\Q}, \msF_d)))=H^a(\Gamma_Q, C^q(\mcX_{\Q}, \msF_d))=0, \ \forall a \geq 1.
\]
Actually, we will prove that $C^q(\mcX_{\Q}, \msF_d)$ is a co-induced $\Gamma_Q$-module.
First, recall that
\[
C^q(\mcX_{\Q}, \msF_d)=\prod_{I \in (\XQ)^{q+1}} \Gamma(V_I, \msF_d),
\]
and that $\Gamma_Q$ acts freely on $(\XQ)^{q+1}$ by Lemma \ref{lem:tor} (7).
Let $A \subset (\XQ)^{q+1}$ be a system of representatives of $\Gamma_Q \bs (\XQ)^{q+1}$, and set
\[
M:=\prod_{I \in A} \Gamma(V_I, \msF_d).
\]
Then recall that the $GL_g(\Q)$-equivariant structure on $\msF_d$ gives an isomorphism
\begin{align}\label{eqn:co-ind}
[\gamma]: \Gamma(V_I, \msF_d) \isomto \Gamma((\tp{\gamma})^{-1}V_I, \msF_d) =\Gamma(V_{\gamma I}, \msF_d)
\end{align}
for each $I \in (\XQ)^{q+1}$ and $\gamma \in GL_g(\Q)$, cf. Lemma \ref{lem stable cov}.
Therefore, for each $\gamma \in \Gamma_Q$, we have an isomorphism
\[
M=\prod_{I \in A} \Gamma(V_I, \msF_d) \isomto \prod_{I \in A} \Gamma(V_{\gamma I}, \msF_d); \  (s_I)_{I \in A
} \mapsto ([\gamma](s_I))_{I \in A},
\]
and hence we obtain an isomorphism
\begin{align*}
\Hom_{\Z} (\Z[\Gamma_Q], M)=\prod_{\gamma \in \Gamma_Q} M &\isomto C^q(\mcX_{\Q}, \msF_d)=\prod_{\gamma \in \Gamma_Q}\prod_{I \in A} \Gamma(V_{\gamma I}, \msF_d).
%\Hom_{\Z} (\Z[\Gamma_Q], M) &\isomto C^q(\mcX_{\Q}, \msF_d)=\prod_{I \in (\XQ)^{q+1}} \Gamma(V_I, \msF_d) %\\
%(m_{\gamma})_{\gamma} &\longmapsto ([\gamma](m_{\gamma}))_{\gamma}.
\end{align*}
Since this is clearly a $\Gamma_Q$-equivariant isomorphism, we see $C^q(\mcX_{\Q}, \msF_d)$ is a co-induced $\Gamma_Q$-module.
%This proves the acyclicity assertion.
%This completes the proof. 
%This completes the proof of (1). 

(2) can be proved similarly. First, by the spectral sequence
\[
E_2^{ab}=H^a(SL_g(\Z), H^b(\Yo,  \msC^q(\mcX_{\Q}, \msF_d)^{\Xi})) \Ra H^{a+b}(\Yo, SL_g(\Z),  \msC^q(\mcX_{\Q}, \msF_d)^{\Xi})
\]
and Proposition \ref{prop:acyc1}, it suffices to show 
\[
H^a \big(SL_g(\Z), C^q(\mcX_{\Q}, \msF_d)^{\Xi} \big)=0,  \ \forall a \geq 1. 
\]
Again, we will prove that 
\[
C^q(\mcX_{\Q}, \msF_d)^{\Xi} \simeq \prod_{Q \in \Xi} \prod_{I \in (\XQ)^{q+1}} \Gamma(V_I, \msF_d)
\]
is a co-induced $SL_g(\Z)$-module. 
Note that the action of $SL_g(\Z)$ on $\Xi \times (\XQ)^{q+1}$ is free.  Indeed, if we have 
\[
\gamma (Q, I)=([\gamma](Q),
\gamma I)=(Q, I),
\]
then it follows that $\gamma \in \Gamma_Q$, and hence $\gamma=1$ since the action of  $\Gamma_Q$ on $(\XQ)^{q+1}$ is free. 
Let $A' \subset \Xi \times (\XQ)^{q+1}$ be a system of representatives of $SL_g(\Z) \bs (\Xi \times (\XQ)^{q+1})$, and set 
\[
M':= \prod_{(Q, I) \in A'} \Gamma(V_I, \msF_d). 
\]
Then again by using \eqref{eqn:co-ind}, we obtain an isomorphism 
\begin{align*}
\Hom_{\Z} \big(\Z[SL_g(\Z)], M' \big)=\prod_{\gamma \in SL_g(\Z)} M' \isomto C^q(\mcX_{\Q}, \msF_d)^{\Xi}\simeq\prod_{\gamma \in SL_g(\Z)}\prod_{(Q, I) \in A'} \Gamma(V_{\gamma I}, \msF_d).
%\Hom_{\Z} (\Z[\Gamma_Q], M) &\isomto C^q(\mcX_{\Q}, \msF_d)=\prod_{I \in (\XQ)^{q+1}} \Gamma(V_I, \msF_d) %\\
%(m_{\gamma})_{\gamma} &\longmapsto ([\gamma](m_{\gamma}))_{\gamma}.
\end{align*} 
of $SL_g(\Z)$-modules. %This completes the proof. 
Thus we find that $C^q(\mcX_{\Q}, \msF_d)^{\Xi}$ is a co-induced $SL_g(\Z)$-module. 
\end{proof}

%\begin{rmk}
%More generally, using the same argument we can actually show that $\msC^q(\mcX_{\Q}, \msF_d)$ (resp. $\msC^q(\mcX_{\Q}, \msF_d)^{\Xi}$) is $\Gamma(\Yo, F_Q^{\times}, -)$-acyclic (resp. $\Gamma(\Yo, GL_g(\Q), -)$-acyclic), but in the following we only need the above statement. 
%\end{rmk}

\begin{cor}\label{cor:equiv coh}
\begin{enumerate}
\item Let $Q \in \Xi$. For $q \geq 0$, we have
\[
H^q (\Yo, \Gamma_Q, \msF_d) \simeq H^q \big( \Gamma\big(\Yo, \Gamma_Q, \msC^{\bullet}(\mcX_{\Q}, \msF_d) \big) \big)
= H^q \big(C^{\bullet}\big(\mcX_{\Q}, \msF_d \big)^{\Gamma_Q}\big), 
\]
where the second and third $H^q$ are the cohomology of complexes. 
\item For $q \geq 0$, we have
\begin{align*}
H^q (\Yo, SL_g(\Z), \msF_d^{\Xi}) &\simeq H^q \big( \Gamma\big(\Yo, SL_g(\Z), \msC^{\bullet}(\mcX_{\Q}, \msF_d)^{\Xi} \big) \big) \\
&= H^q \big( \Map_{SL_g(\Z)}\big(\Xi, C^{\bullet}\big(\mcX_{\Q}, \msF_d \big)\big)\big), 
\end{align*}
where $\Map_{SL_g(\Z)}(-,-)$ is the set of $SL_g(\Z)$-equivariant maps. 
\item For $Q \in \Xi$, we have the following commutative diagram
%the evaluation maps \eqref{eqn:ev}, \eqref{eqn:ev cech} induces a commutative diagram
\begin{align*}
%\xymatrix{
%H^q (\Yo, SL_g(\Z), \msF_d^{\Xi})  \ar[r]^(0.4){\sim}  \ar[d]_{\ev_Q} & H^q \big( \Map_{SL_g(\Z)}\big(\Xi, C^{\bullet}\big(\mcX_{\Q}, \msF_d \big)\big)\big) \ar[d]_{\ev_Q} \\
%H^q (\Yo, \Gamma_Q, \msF_d) \ar[r]^{\sim}   & H^q \big(C^{\bullet}\big(\mcX_{\Q}, \msF_d \big)^{\Gamma_Q}\big)
%}
\xymatrix@R=0pt{
H^q (\Yo, SL_g(\Z), \msF_d^{\Xi}) \ar[r]^-{\ev_Q}&H^q (\Yo, \Gamma_Q, \msF_d) \\
|\wr&|\wr \\
 H^q \big( \Map_{SL_g(\Z)}\big(\Xi, C^{\bullet}\big(\mcX_{\Q}, \msF_d \big)\big)\big)  \ar[r]^-{\ev_Q}&H^q \big(C^{\bullet}\big(\mcX_{\Q}, \msF_d \big)^{\Gamma_Q}\big)\\
}
\end{align*}
where the two $\ev_Q$ are the evaluation maps induced by \eqref{eqn:ev} and \eqref{eqn:ev cech}. 
\end{enumerate}
\end{cor}

We end this section by adding one more corollary concerning an operation which shifts the index $d \geq 0$ of $\msF_d$.  

\begin{cor}\label{cor:eq coh1.5}
  Let $P(y_1, \dots, y_g)\in \C[y_1, \dots, y_g]$ be a homogeneous polynomial of degree $d' \leq d$ such that
  \[
    P(\tp{\gamma}y)=P(y), \forall \gamma \in \Gamma_Q.
  \]
  Then the multiplication by $P$:
  \begin{align*}
    P: C^q(\mcX_{\Q}, \msF_d) \ra C^q(\mcX_{\Q}, \msF_{d-d'});\ (s_I(y))_{I\in (\XQ)^{q+1}} \mapsto (P(y)s_I(y))_{I \in (\XQ)^{q+1}}
  \end{align*}
  gives a $\Gamma_Q$-equivariant map of complexes, and hence induces a map
  \begin{align*}
    P: H^q(\Yo, \Gamma_Q, \msF_d) \ra H^q(\Yo, \Gamma_Q, \msF_{d-d'}).
  \end{align*}
\end{cor}

\begin{ex}\label{ex:norm}
A typical example of such a $\Gamma_Q$-invariant homogeneous polynomial $P$ is the norm polynomial $N_{w^*}$ defined in Section \ref{subsect:number field}, cf. Lemma \ref{lem:rev}. 
More generally, let $k \geq 1$ be an integer. Under the notations in Lemma \ref{lem:rev}, the $k$-th power $N_{w^*}^k$ of the norm polynomial $N_{w^*}$ is a $\Gamma_{\Qw}$-invariant homogeneous polynomial of degree $kg$. In particular, we have a map %Lemma ver
%More generally, let $k \geq 1$ be an integer. Under the notations in Section \ref{subsect:number field}, the $k$-th power $N_{w^*}^k$ of the norm polynomial $N_{w^*}$ is a $\Gamma_{\Qw}$-invariant homogeneous polynomial of degree $kg$. In particular we have a map %Section ver
\[
 N_{w^*}^k: H^q(\Yo, \Gamma_{\Qw}, \msF_{kg}) \ra H^q(\Yo, \Gamma_{\Qw}, \msF_{0}).
\]
\end{ex}

%%%%%%%%%%%%%%%%%%%%%%%%%%%%%%%%%%%%%%%%%%%%%%%%%%%%%%%%%%%%%%%%%%%%%%%%%%%%%%%%%%%%%%%%%%%%%%%%%%%%%%%%%%%%%%%%%%%%%%%%%%%%%%%%%%%

\section{Cones and the exponential perturbation}\label{sect:cone and pert}

In this section we introduce the notion of exponential perturbation which is a modification of the so-called upper closure or $Q$-perturbation (Colmez perturbation) used in \cite{Y10:On-}, \cite{BHYY19:Can}, and \cite{CDG15:Int}. This is one of the key ingredients which enables us to deal with general number fields.

For $r\geq 0$, $I=(\alpha_1, \dots, \alpha_r) \in (\R^g \setm \{0\})^r$, let 
\begin{align*}
%\begin{cases}
C_I&:=\sum_{i=1}^r \Rpos \alpha_i \subset \R^g,
%\hat C_I&:=\sum_{i=1}^r \R_{\geq 0} \alpha_i  \subset \R^g
%\end{cases}
\end{align*}
denote the \textit{open cone generated by} $\alpha_1, \dots, \alpha_r$.
In the case $r=0, I=\emptyset$, we assume $C_{\emptyset}:=\{0\}$.
\begin{rmk}
We follow the convention to call $C_I$ an ``open'' cone although it is not necessarily an open subset of $\R^g$. Note that, however, $C_I$ is open in $\Span_{\R} \{\alpha_1, \dots, \alpha_r\}$, where $\Span_{\R} \{\alpha_1, \dots, \alpha_r\} \subset \R^g$ is the $\R$-subspace spanned by $\alpha_1, \dots, \alpha_r$, cf. Lemma \ref{lem:rsoc}.
\end{rmk}

Recall that $\XQ:=\Q^g \setm \{0\}$ denotes the set of non-zero vectors of $\Q^g$. 
In this paper we fix the terminologies concerning cones as follows. %First recall that we have put $\XQ:=\Q^g \setm \{0\}$. 
\begin{dfn}
\begin{enumerate}
\item An open cone $C_I$ is called \textit{rational} if we can take $I \in (\XQ)^r$. % ($r \geq 1$).
%\item An open cone $C_I$ is called \textit{simplicial} if we can take $I=(\alpha_1, \dots, \alpha_r)$ with linearly independent $\alpha_1, \dots \alpha_r$.
\item An open cone $C_I$ is called \textit{simplicial} if $\alpha_1, \dots \alpha_r$ are linearly independent over $\R$.
%\item A \textit{rational constructible cone} in $\R^g$ is a subset which is a disjoint union of a finite number of rational simplicial open cones.
\item We refer to a subset of $\R^g$ which can be written as a disjoint union of a finite number of rational simplicial open cones as a \textit{rational constructible cone}. 
\end{enumerate} %
\end{dfn}

\subsection{The exponential perturbation}\label{subsect:pert}

Recall that 
\begin{align*}
\Xi =\Big\{Q \in GL_g(\Q) \ \Big| \  \text{irreducible over $\Q$}\Big\} 
\end{align*}
denotes the set of irreducible matrices of $GL_g(\Q)$, cf. Section \ref{subsect:irr}.

\begin{dfn}
For $Q \in \Xi$ and a subset $A \subset \R^g$, we define the $Q$-\textit{closure} $A^Q$ of $A$ as
\[
A^Q:=\left\{x \in \R^g \mid \exists \delta >0, \forall \varepsilon \in (0, \delta), \ \exp(\varepsilon Q)x \in A  \right\},
\]
where $\exp(\varepsilon Q) \in GL_g(\R)$ is the matrix exponential of $\varepsilon Q \in GL_g(\R)$.
\end{dfn}

\begin{rmk}
The $Q$-closure is defined by considering the perturbation of $x \in \R^g$ by $\exp(\varepsilon Q)$, and we call this process \textit{the exponential perturbation}. The Colmez perturbation used in \cite{CDG15:Int} is the perturbation of $x$ by the vectors $Q \in \R^g$ whose components are linearly independent over $\Q$.
\end{rmk}

\begin{lem}\label{lem:capQ}
Let $Q \in \Xi$.
\begin{enumerate}
\item Let $A, B \subset \R^g$ be subsets such that $A \subset B$. Then we have
\[
A^Q \subset B^Q.
\]
\item Let $A_1, \dots , A_m \subset \R^g$ be subsets. Then we have
\[
(A_1\cap \cdots \cap A_m)^Q = A_1^Q \cap \cdots \cap A_m^Q.
\]
In particular, if $A_1\cap \cdots \cap A_m=\emptyset$, then $A_1^Q \cap \cdots \cap A_m^Q=\emptyset$.
\end{enumerate}
\end{lem}

\begin{proof}
(1) is obvious. We prove (2). The inclusion $\subset$ is clear. We prove $\supset$. Let $x \in A_1^Q \cap \cdots \cap A_m^Q$.
Then by definition, there exist $\delta_1, \dots , \delta_m >0$ such that
\[
\exp((0, \delta_i)Q)x \subset A_i
\]
for $i=1, \dots, m$. Put $\delta:= \min \{\delta_1, \dots, \delta_m \}>0$. Then we have
\[
\exp((0,\delta)Q)x \subset A_1\cap \cdots \cap A_m,
\]
and hence $x \in (A_1\cap \cdots \cap A_m)^Q$.
\end{proof}

In the following, we study the $Q$-closure $C_I^Q$ of rational open cones $C_I$, which play an important role in the construction of our Shintani cocycle.

\begin{lem}\label{lem:gammaCQ}
For $r \geq 0$, $I=(\alpha_{1}, \dots, \alpha_{r}) \in (\XQ)^r$, $Q\in \Xi$, and $\gamma \in GL_g(\Q)$, we have
\[
\gamma (C_{\gamma^{-1}I}^Q)= C_I^{[\gamma](Q)},
\]
where $[\gamma](Q)=\gamma Q \gamma^{-1} \in \Xi$. %, cf. Section \ref{subsect:irr}.
%where $\gamma (C_{\gamma^{-1}I}^Q)$ is the translation of $C_{\gamma^{-1}I}^Q$ by $\gamma$ and $[\gamma](Q)=\gamma Q \gamma^{-1}$. %, cf. Section \ref{subsect:irr}.
\end{lem}

\begin{proof}
Indeed, for $x \in \R^g$ and $\varepsilon >0$, we see that %the condition $\exp(\varepsilon [\gamma](Q)x) \in C_I$ holds if and only if
\begin{align*}
\exp(\varepsilon [\gamma](Q))x \in C_I &\Leftrightarrow
\exp(\varepsilon \gamma Q \gamma^{-1})x \in C_I  \\
&\Leftrightarrow  \exp (\varepsilon Q) \gamma^{-1}x \in \gamma^{-1} (C_I) =C_{\gamma^{-1}I}.
\end{align*}
This proves the lemma.
\end{proof}

%%%%%%%%%%%%%%%%%%%%%%%%%%%%%%%%%%%%%%%%%%%%%%%%%%%%%%%%%%%%%%%%%%%%%%%%%%

\subsection{Rationality}\label{sect:rat}

%The aim of this subsection is to prove the following two propositions:
The aim of this subsection is to prove the following proposition:
\begin{prop}\label{prop:rat}
%Let $r \geq 0$, $I=(\alpha_{1}, \dots, \alpha_{r}) \in (\XQ)^r$ and $Q \in \Xi$.
Let $0 \leq r \leq g$, $I=(\alpha_{1}, \dots, \alpha_{r}) \in (\XQ)^r$, and $Q \in \Xi$.
\begin{enumerate}
%\item If $\dim_{\R} \Span_{\R}\{\alpha_1, \dots \alpha_r\} \leq g-1$, then $C_I^Q =\emptyset$.
\item Suppose $\dim_{\Q} \Span_{\Q}\{\alpha_1, \dots \alpha_r\} \leq g-1$, then %$C_I^Q =\emptyset$.
\begin{align*}
C_I^Q=
\begin{cases}
\{0\} & \text{ if } 0 \in C_I\\
\  \emptyset & \text{ if } 0 \notin C_I.
\end{cases}
\end{align*}
\item The $Q$-closure $C_I^Q$ of the rational open cone $C_I$ generated by $I$ is a rational constructible cone, i.e., a disjoint union of a finite number of rational simplicial open cones. 
\end{enumerate}
\end{prop}

%\begin{prop}\label{prop:rat2}
%Let $C_{I_1}, \dots, C_{I_m}$ be rational simplicial open cones, where $I_i \in (\XQ)^{r_i}$, $0 \leq r_i \leq g$ for $i=1, \dots , m$.
%Assume that $C_{I_i}\cap C_{I_j}=\emptyset$ for all $i\neq j$, and let
%\[
%P=\coprod_{i=1}^m C_{I_i},
%\]
%i.e., $P$ is a rational constructible cone. Let $Q \in \Xi$. Then we have
%\[
%P^Q=\coprod_{i=1}^m C_{I_i}^Q. %=\coprod_{\substack{i=1\\r_i=g}}^mC_{I_i}^Q. %, %.
%\]
%\end{prop}

%\begin{rmk}
%In fact the assumption $r \leq g$ in Proposition \ref{prop:rat} is superfluous, i.e., the statements hold true for any $r \geq 0$, $I=(\alpha_{1}, \dots, \alpha_{r}) \in (\XQ)^r$. However, since we do not need the case $r > g$ and it requires more complicated arguments, we assume $r \leq g$.
%\end{rmk}

In order to prove this proposition, we first prepare several lemmas. 
In the following, for $\alpha \in \R^g \setm \{0\}$, we put
\begin{align*}
U_{\alpha, \pm} &:=\{x \in \R^g \mid \pm\brk{x, \alpha}>0 \}, \\
H_{\alpha} &:=\{x \in \R^g \mid \brk{x, \alpha}=0\}.
\end{align*}
We start with recalling the following fact.

\begin{lem}[\cite{S76:On-} Section 1.2, \cite{H93:Ele} pp.~68--69, Lemma 1]\label{lem:polyhedron}~
\begin{enumerate}
%\item Let $W \subset \Q^g$ be a $\Q$-subspace, and let $l_1, \dots, l_m \in \YQ=\Q^g \setm \{0\}$. Then the subset
\item Let $W \subset \Q^g$ be a $\Q$-subspace, and let $l_1, \dots, l_m \in \YQ$. Then the subset
\[
X=\{x \in W\otimes_{\Q}\R \subset \R^g \mid \brk{x, l_i}>0, \  i=1, \dots, m \} \subset \R^g
\]
is a rational constructible cone.
\item Let $C,C' \subset \R^g$ be rational constructible cones. Then $C\cup C', C\cap C', C\setm C'$ are rational constructible cones.
\end{enumerate}
\end{lem}
\begin{proof}
See \cite[Lemma 2, Corollary to Lemma 2]{S76:On-} and \cite[pp.~68--69, Lemma 1]{H93:Ele}. 
%(1) follows from \cite[Corollary to Lemma 2]{S76:On-}, and (2) is exactly the statement of \cite[p.~68--69, Lemma 1]{H93:Ele}. 
%(1) follows from \cite[Lemma 2, Corollary to Lemma 2]{S76:On-}, and (2) is exactly the statement of \cite[p.~68--69, Lemma 1]{H93:Ele}. 
Although \cite{H93:Ele} is assuming that the total space is of the form $F\otimes_{\Q}\R$ for a number field $F$ and that $W$ is a subspace generated by elements in $F$, its proof does not use this special assumption.
\end{proof}

The following is the key lemma of this section.
\begin{lem}\label{lem:hypQ}
Let $Q \in \Xi$ and $\alpha \in \YQ$. For $k \geq 0$, put
\[
H_{\pm}^{(k)}:=\{x\in \R^g \mid \pm \brk{x, \tp{Q}^k\alpha}>0, \ \brk{x, \tp{Q}^j\alpha}=0,\  \text{for } 0\leq j \leq k-1 \}.
\]
Note that $H_{\pm}^{(0)}=U_{\alpha, \pm}$ by definition.
\begin{enumerate}
\item There exists $k_0 \geq 0$ such that $H_{\pm}^{(k)} = \emptyset$ for all $k \geq k_0+1$. Moreover, we have
\[
\R^g \setm \{0\} = \coprod_{k=0}^{k_0} \left(H_+^{(k)} \coprod H_-^{(k)} \right),
%\R^g \setm \{0\} = \bigsqcup_{k=0}^{k_0} \left(H_+^{(k)} \sqcup H_-^{(k)} \right),
\]
where $\coprod$ denotes the disjoint union.
%where $\sqcup$ and $\bigsqcup$ denote the disjoint union.
%\item For all $k \geq 1$, we have $H_+^{(k)} \subset (H_+^{(0)})^Q$ and $H_-^{(k)} \subset (H_-^{(0)})^Q$.
\item For all $k \geq 0$, the sets $H_+^{(k)}, H_-^{(k)}$ are rational constructible cones.
\item For all $k \geq 0$, we have $H_+^{(k)} \subset (H_+^{(0)})^Q=(U_{\alpha,+})^Q$ and $H_-^{(k)} \subset (H_-^{(0)})^Q=(U_{\alpha,-})^Q$.
\item We have $H_{\alpha}^Q =\{0\}$ and
\begin{align*}
\R^g \setm \{0\} = (U_{\alpha,+})^Q \coprod (U_{\alpha,-})^Q. 
\end{align*}
In particular, $\R^g = H_{\alpha}^Q \coprod (U_{\alpha,+})^Q \coprod (U_{\alpha,-})^Q$.
%\item We have $(U_{\alpha,+})^Q= \coprod_{k=0}^{k_0} H_+^{(k)}$, and $(U_{\alpha,-})^Q= \coprod_{k=0}^{k_0} H_-^{(k)}$. In particular, $(U_{\alpha,+})^Q$ and $(U_{\alpha,-})^Q$ are rational constructible cones.
\item We have
\begin{align*}
(U_{\alpha,+})^Q &= \coprod_{k=0}^{k_0} H_+^{(k)}, \\
(U_{\alpha,-})^Q &= \coprod_{k=0}^{k_0} H_-^{(k)}.
\end{align*}
In particular, $(U_{\alpha,+})^Q$ and $(U_{\alpha,-})^Q$ are rational constructible cones.
\end{enumerate}
\if0
In particular, we have
\begin{align*}
\R^g \setm \{0\} &= (U_{\alpha,+})^Q \coprod (U_{\alpha,-})^Q, \\
(U_{\alpha,+})^Q&= \coprod_{k=0}^{k_0} H_+^{(k)}, \\
(U_{\alpha,-})^Q&= \coprod_{k=0}^{k_0} H_-^{(k)}, \\
H_{\alpha}^Q=\emptyset,
\end{align*}
and hence $(H_+^{(0)})^Q$ and $(H_-^{(0)})^Q$ are rational constructible cones.
\fi
\end{lem}
\begin{proof}
For $k \geq 0$, put
\[
H^{(k)}:=\{x\in \R^g \mid \brk{x, \tp{Q}^j\alpha}=0,\  \text{for } 0\leq j \leq k-1 \}.
\]
Then we have a descending chain
\[
\R^g=H^{(0)} \supset H^{(1)} \supset H^{(2)} \supset \cdots
\]
of $\R$-vector spaces.
Note that the subspaces $H^{(k)}$ are all defined over $\Q$ since we have $\tp{Q}^j\alpha \in \YQ$ for $j \geq 0$.
Since $\R^g$ is a finite dimensional vector space, there exists $k_0 \geq 0$ such that $H^{(k)}=H^{(k_0+1)}$ for all $k \geq k_0+1$.
\begin{claim}
%We claim that $H^{(k_0+1)}=0$.
We have $H^{(k_0+1)}=0$.
\end{claim}
Indeed, let $x \in H^{(k_0+1)}=H^{(k_0+2)}$. Then we have
\[
\brk{Qx, \tp{Q}^j\alpha} =\brk{x, \tp{Q}^{j+1}\alpha} =0, \text{ for } 0 \leq j \leq k_0,
\]
and hence $Qx \in H^{(k_0+1)}$. Therefore, $H^{(k_0+1)}$ is a $Q$-stable subspace of $\R^g$ defined over $\Q$. Moreover, since $\alpha \neq 0$, we have
\[
H^{(k_0+1)} \subset H^{(1)} \subsetneq \R^g.
\]
%Therefore we obtain $H^{(k_0+1)}=0$ because $Q$ is irreducible over $\Q$.
Therefore, we obtain $H^{(k_0+1)}=0$ by Lemma \ref{lem:tor} (2). 

Now (1) follows from the fact
\[
H^{(k)}\setm H^{(k+1)} = H_+^{(k)} \coprod H_-^{(k)}, \ \forall k \geq 0, 
\]
and (2) follows from Lemma \ref{lem:polyhedron} (1).

%We prove (3).
(3)
Let $x \in H_+^{(k)}$. Then we have
\[
\brk{\exp(\varepsilon Q)x, \alpha}=\sum_{m\geq k}\frac{\brk{x, \tp{Q}^m\alpha}}{m!}\varepsilon^m.
\]
Now since $\brk{x, \tp{Q}^k\alpha}>0$, there exists $\delta >0$ such that
\[
\brk{\exp(\varepsilon Q)x, \alpha}=\sum_{m\geq k}\frac{\brk{x, \tp{Q}^m\alpha}}{m!}\varepsilon^m >0
\]
for all $\varepsilon \in (0, \delta)$. Hence $x \in (H_+^{(0)})^Q$. 
The inclusion $H_-^{(k)} \subset (H_-^{(0)})^Q$ can be proved similarly. %in the same way. %This completes the proof.

(4) First, by Lemma \ref{lem:capQ} (2), we see $(U_{\alpha,+})^Q \cap (U_{\alpha,-})^Q=\emptyset$, and $H_{\alpha}^Q \cap (U_{\alpha,\pm})^Q=\emptyset$.
On the other hand, we obviously have $0 \in H_{\alpha}^Q$, and hence $0 \notin (U_{\alpha,\pm})^Q$.
Therefore, by (1) and (3), we obtain
\[
\R^g\setm \{0\} \subset (U_{\alpha,+})^Q \coprod (U_{\alpha,-})^Q \subset \R^g\setm \{0\}.
\]
%Again, by $H_{\alpha}^Q \cap (U_{\alpha,\pm})^Q=\emptyset$, we also get $H_{\alpha}^Q =\{0\}$.
Thus we find $\R^g \setm \{0\}=(U_{\alpha,+})^Q \coprod (U_{\alpha,-})^Q$ and $H_{\alpha}^Q =\{0\}$. 

(5) The first part follows from (1), (3), and (4). Then the latter part follows from (2). %This completes the proof.
\end{proof}

\begin{lem}\label{lem:rsoc}
Let $I=(\alpha_1, \dots, \alpha_r) \in (\XQ)^r$ such that $\alpha_1, \dots, \alpha_r \in \Q^g \setm \{0\}$ are linearly independent. Note that we automatically have $r \leq g$.
\begin{enumerate}
\item There exist $\alpha^*_1, \dots, \alpha^*_r, \beta^*_1, \dots , \beta^*_{g-r} \in \YQ$ such that
\[
C_I=\left(\bigcap_{i=1}^rU_{\alpha^*_i,+}\right) \cap \left(\bigcap_{i=1}^{g-r}H_{\beta^*_i}\right).
\]
\item Let $Q \in \Xi$. Then we have
\[
\R^g=C_I^Q \coprod \left( \R^g \setm C_I \right)^Q.
\]
\end{enumerate}
\end{lem}

\begin{proof}
(1) Put $W:=\Span_{\Q}\{\alpha_1, \dots, \alpha_r \} \subset \Q^g$, and let $W^{\perp}\subset \Q^g$ be its orthogonal complement with respect to the scalar product $\brk{\ , \ }$. Let $\alpha^*_1, \dots, \alpha^*_r \in W$ be the dual basis of $\alpha_1, \dots, \alpha_r$ in $W$ with respect to $\brk{\ ,\ }$, i.e.,
\begin{align*}
\brk{\alpha_i, \alpha^*_j}=
\begin{cases}
1 & (i=j)\\
0 & (i \neq j)
\end{cases},
\end{align*}
and let $\beta^*_1, \dots , \beta^*_{g-r} \in W^{\perp}$ be a basis of $W^{\perp}$ over $\Q$.
Then $\alpha^*_1, \dots, \alpha^*_r, \beta^*_1, \dots , \beta^*_{g-r}$ satisfy the desired property.
Indeed, let $\beta_1, \dots, \beta_{g-r} \in W^{\perp}$ be the dual basis of $\beta^*_1, \dots , \beta^*_{g-r}$ in $W^{\perp}$, and let $x \in \R^g$. Since $\alpha_1, \dots, \alpha_r, \beta_1, \dots , \beta_{g-r}$ form a basis of $\R^g$, we have
\[
%x =\sum_{i=1}^rc_i \alpha_i +\sum_{i=1}^{g-r}d_i \beta_i,
x =\sum_{i=1}^rc_i \alpha_i +\sum_{j=1}^{g-r}d_j \beta_j,
\]
for some $c_i, d_j \in \R$. Then we have $x \in C_I$ if and only if
\[
\brk{x, \alpha^*_i}=c_i >0 \text{ and } \brk{x,\beta^*_j}=d_j=0, \ \forall i,j.
\]
%for any $i$. 
This proves (1).

(2) Using (1), we take $\alpha^*_1, \dots, \alpha^*_r, \beta^*_1, \dots , \beta^*_{g-r} \in \YQ$ such that
\begin{align}\label{eqn CI}
C_I=\left(\bigcap_{i=1}^rU_{\alpha^*_i,+}\right) \cap \left(\bigcap_{i=1}^{g-r}H_{\beta^*_i}\right). 
\end{align}
We then have
\begin{align*}%\label{eqn:5.2?}
\R^g \setm C_I = \bigcup_{i=1}^r \left(U_{\alpha^*_i,-} \cup H_{\alpha^*_i}\right) \cup \bigcup_{i=1}^{g-r}\left(U_{\beta^*_i,+}\cup U_{\beta^*_i,-} \right). 
\end{align*}
By taking the $Q$-closure and using Lemma \ref{lem:capQ} (1), we obtain
\begin{align}\label{eqn CI3}
\bigcup_{i=1}^r \left((U_{\alpha^*_i,-})^Q \cup H_{\alpha^*_i}^Q \right) \cup \bigcup_{i=1}^{g-r}\left((U_{\beta^*_i,+})^Q\cup (U_{\beta^*_i,-})^Q \right) \subset \left( \R^g \setm C_I \right)^Q. 
\end{align}
On the other hand, by \eqref{eqn CI}, Lemma \ref{lem:capQ} (2), and Lemma \ref{lem:hypQ} (4), we obtain
\begin{align}\label{eqn CI4}
\begin{split}
\R^g \setm C_I^Q&=\R^g \setm \Bigg( \left(\bigcap_{i=1}^r(U_{\alpha^*_i,+})^Q\right) \cap \left(\bigcap_{i=1}^{g-r}H_{\beta^*_i}^Q\right) \Bigg)\\
&= \bigcup_{i=1}^r \left((U_{\alpha^*_i,-})^Q \cup H_{\alpha^*_i}^Q \right) \cup \bigcup_{i=1}^{g-r}\left((U_{\beta^*_i,+})^Q\cup (U_{\beta^*_i,-})^Q \right).
\end{split}
\end{align}
Therefore,  by combining \eqref{eqn CI3} and \eqref{eqn CI4}, we find that $\R^g \setm C_I^Q  \subset \left( \R^g \setm C_I \right)^Q$, and hence $\R^g = C_I^Q \cup  \left( \R^g \setm C_I \right)^Q$.
Finally, since we have $C_I^Q \cap  \left( \R^g \setm C_I \right)^Q=\emptyset$ by Lemma \ref{lem:capQ} (2), we obtain $\R^g=C_I^Q \coprod \left( \R^g \setm C_I \right)^Q$. 
\end{proof}

Now we prove Proposition \ref{prop:rat}. 

\begin{proof}[Proof of Proposition \ref{prop:rat}]
(1) Since $\Span_{\Q} \{\alpha_1, \dots, \alpha_r\} \subsetneq \Q^g$, there exists $\beta \in \YQ$ such that
\[
C_I \subset \Span_{\R} \{\alpha_1, \dots, \alpha_r\} \subset H_{\beta}.
\]
Therefore, by Lemma \ref{lem:capQ} (1) and Lemma \ref{lem:hypQ} (4), we have either $C_I^Q=\emptyset$ or $C_I^Q=\{0\}$.
Then it is clear that $C_I^Q=\{0\}$ if and only if $0 \in C_I$. This proves (1).

(2) Since $\emptyset$ and $\{0\}$ are obviously rational constructible cones, we may assume $\alpha_1, \dots, \alpha_r$ generates $\R^g$. 
In particular, we have $r=g$ and $C_I$ is a rational simplicial open cone. By Lemma \ref{lem:rsoc} (1), there exist $\alpha^*_1, \dots, \alpha^*_g \in \YQ$ such that
\[
C_I=\bigcap_{i=1}^gU_{\alpha^*_i,+}.
\]
Then by Lemma \ref{lem:capQ} (2), we have
\[
C_I^Q= \bigcap_{i=1}^g\left(U_{\alpha^*_i,+}\right)^Q. 
\]
Now, we already know that $\left(U_{\alpha^*_i,+}\right)^Q$ is a rational constructible cone by Lemma \ref{lem:hypQ} (5), and hence $C_I^Q$ is also a rational constructible cone by Lemma \ref{lem:polyhedron} (2).
\end{proof}

\subsection{Cocycle relation}\label{subsec:coc}

\begin{dfn}
\begin{enumerate}
\item For a subset $A \subset \R^g$, let 
\begin{align*}
\mathbf 1_A: \R^g \ra \R; x \mapsto
\begin{cases}
0 & \text{ if } x \notin A\\
1 & \text{ if } x \in A
\end{cases}
\end{align*}
denote the characteristic function of $A$.
\item For $I=(\alpha_1, \dots, \alpha_g) \in (\R^g \setm \{0\})^{g}$, we set
\[
\sgn (I) := \sgn \det (\alpha_1, \dots, \alpha_g) \in \{-1, 0, 1\},
\]
where $(\alpha_1, \dots, \alpha_g)$ is regarded as an element in $M_g(\R)$.
We assume $\sgn 0:=0$.
%It is clear that $\sgn (I)$ is well-defined.
%\item Let $r\geq 1$ and $I=(\alpha_{1}, \dots, \alpha_{r}) \in (X_{\R})^r$, where $\alpha_1, \dots, \alpha_r \in \R^g \setm \{0\}$. We say $x \in \R^g$ is \textit{generic relative to} $I$ if $x$ is not contained in any $\R$-subspace generated by $g-1$ or less elements in $\{\alpha_1, \dots, \alpha_r\}$. It is also clear that this is well-defined for $I$.
\item Let $r\geq 1$ and $I=(\alpha_{1}, \dots, \alpha_{r}) \in (\R^g \setm \{0\})^r$. We say that $x \in \R^g$ is \textit{in general position relative to} $I$ if $x$ is not contained in any proper $\R$-subspace of $\R^g$ generated by a subset of $\{\alpha_1, \dots, \alpha_r\}$.
%Here, a proper $\R$-subspace is a subspace not equal to $\R^g$.
%It is also clear that this is well-defined for $I$.
\end{enumerate}
\end{dfn}

\begin{rmk}
The condition ``in general position relative to $I$'' is slightly more strict than the condition ``generic with respect to $\{\alpha_1, \dots, \alpha_r\}$'' in the sense of Yamamoto \cite[p.~471]{Y10:On-}. Actually, this difference is not important at all, but we adopt this definition since it is more useful in this paper. 
%\begin{enumerate}
%\item
%The condition ``in general position relative to $I$'' is slightly more strict than the condition ``generic with respect to $\{\alpha_1, \dots, \alpha_r\}$'' in the sense of Yamamoto \cite[p.~471]{Y10:On-}. Actually, this difference is not important at all, but we adopt this definition since it is more useful in this paper. %just to avoid cones in this definition. %make the definition simple.
%\item Note that if an $\R$-subspace $W \subset \R^g$ generated by a subset of $\{\alpha_1, \dots, \alpha_r\}$ is proper, i.e., $W \neq \R^g$, then $W$ is generated by $g-1$ or less elements in $\{\alpha_1, \dots, \alpha_r\}$.
%\end{enumerate}
\end{rmk}

\begin{lem}\label{lem:gen}
%Let $r\geq 1$, $I=(\alpha_{1}, \dots, \alpha_{r}) \in (\XQ)^r$, $x \in \R^g \setm \{0\}$, and $Q \in \Xi$. Then there exists $\delta>0$ such that for all $\varepsilon \in (0, \delta)$, $\exp (\varepsilon Q)x$ is in general position relative to $I$.
Let $r\geq 1$, $I=(\alpha_{1}, \dots, \alpha_{r}) \in (\XQ)^r$, $x \in \R^g \setm \{0\}$, and $Q \in \Xi$. Then there exists $\delta>0$ such that $\exp (\varepsilon Q)x$ is in general position relative to $I$ for all $\varepsilon \in (0, \delta)$.
\end{lem}

\begin{proof}
Let $W_1, \dots, W_m \subsetneq \R^g$ be all the proper $\R$-subspaces which can be generated by some subset of $\{\alpha_1, \dots, \alpha_r\}$. In particular, $y \in \R^g$ is in general position relative to $I$ if and only if %\marginpar{english}
\[
y \notin \bigcup_{j=1}^m W_j.
\]
%Then there exist 
Take $\beta_1, \dots, \beta_m \in \YQ$ such that $W_j \subset H_{\beta_j}$ for $j=1, \dots, m$. (See Section \ref{sect:rat} for the definition of $H_{\beta_j}$.)
Then by Lemma \ref{lem:hypQ} (4), for each $j$, there exists $\delta_j >0$ such that
\[
%\exp(\varepsilon Q)x \notin H_{\beta_j}, \ \forall \varepsilon \in (0,\delta_j).
\exp((0,\delta_j) Q)x \subset U_{\beta_j, +} \cup U_{\beta_j, -}= \R^g \setm H_{\beta_j}.
\]
Put $\delta:= \min \{\delta_1, \dots, \delta_m\}>0$. Then for all $\varepsilon \in (0,\delta)$, we have
\[
\exp (\varepsilon Q)x \notin \bigcup_{j=1}^m H_{\beta_j} \supset  \bigcup_{j=1}^m W_j,
\]
and hence $\exp (\varepsilon Q)x$ is in general position relative to $I$.
\end{proof}

%The aim of this subsection is to prove the following.
The following is the main proposition of this subsection.
\begin{prop}\label{prop:cocycle}
%%Let $J=(\alpha_{0}, \dots, \alpha_{g}) \in (\XQ)^{g+1}$, $Q \in \Xi$, and $y \in \R^g \setm \{0\}$. Assume that $\brk{\alpha_i, y}>0$ for all $i=0, \dots, g$. %For $i=0, \dots, g$, we put
Let $J=(\alpha_{0}, \dots, \alpha_{g}) \in (\XQ)^{g+1}$ and $Q \in \Xi$. Assume that there exists $y \in \R^g \setm \{0\}$ such that for all $i=0, \dots, g$ we have $\brk{\alpha_i, y}>0$. %For $i=0, \dots, g$, we put
%Let $J=(\alpha_{0}, \dots, \alpha_{g}) \in (\XQ)^{g+1}$, $Q \in \Xi$. Assume that there exists $y \in \R^g \setm \{0\}$ such that $\brk{\alpha_i, y}>0$ holds for all $i=0, \dots, g$. %For $i=0, \dots, g$, we put
%Then for any $x \in \R^g \setm \{0\}$, we have
Then we have
\[
%\sum_{i=0}^g (-1)^i \sgn (\alpha_0, \dots, \check{\alpha_i}, \dots, \alpha_g) \mathbf 1_{C_{(\alpha_{0}, \dots, \check{\alpha_{i}}, \dots, \alpha_{g})}^Q} (x)=0.
\sum_{i=0}^g (-1)^i \sgn (J^{(i)}) \mathbf 1_{C_{J^{(i)}}^Q} (x)=0
\]
for $x \in \R^g \setm \{0\}$, 
where $J^{(i)} = (\alpha_{0}, \dots, \check{\alpha_{i}}, \dots, \alpha_{g}) \in (\XQ)^g$. % for $i=0, \dots, g$.
\end{prop}

\begin{proof}
Take such $y \in \R^g \setm \{0\}$. 
We will reduce the problem to the ``generic case''. 
By Lemma \ref{lem:rsoc} (2), for each $i=0, \dots, g$, there exists $\delta_i>0$ such that
\[
\exp((0, \delta_i)Q)x \subset C_{J^{(i)}} \text{ or } \exp((0, \delta_i)Q)x \subset \R^g \setm C_{J^{(i)}}.
\]
In particular, we have
\[
\mathbf 1_{C_{J^{(i)}}^Q}(x)= \mathbf 1_{C_{J^{(i)}}} (\exp(\varepsilon Q)x), \ \forall \varepsilon \in (0,\delta_i).
\]
On the other hand, by Lemma \ref{lem:gen}, there exists $\delta>0$ such that for all $\varepsilon \in (0,\delta)$, $\exp (\varepsilon Q)x$ is in general position relative to $J$. 
Set $\varepsilon_0 :=\frac{1}{2} \min \set{\delta_0, \dots, \delta_g, \delta}$, and put $x' :=\exp (\varepsilon_0 Q)x$. Then we have:
\begin{itemize}
\item $\mathbf 1_{C_{J^{(i)}}^Q}(x)= \mathbf 1_{C_{J^{(i)}}}(x')$ for $i=0, \dots, g$, 
\item $x'$ is in general position relative to $J$.
\end{itemize}
Therefore, it suffices to prove
\begin{align}\label{eqn:cocycle}
\sum_{i=0}^g (-1)^i \sgn (J^{(i)}) \mathbf 1_{C_{J^{(i)}}} (x')=0 
\end{align}
for any $x'$ which is in general position relative to $J$.
First, if $\brk{x',y}\leq 0$, then we have
\[
\mathbf 1_{C_{J^{(i)}}} (x')=0, \ \forall i \in \{0, \dots, g\} 
\]
because $\brk{\alpha_i,y}>0$ for all $i=0, \dots, g$. Therefore, we may assume $\brk{x',y}>0$.
In this case, the identity \eqref{eqn:cocycle} follows from \cite[Proposition 6.2]{Y10:On-}.

Indeed, let $\gamma \in GL_g(\R)$ such that $\tp \gamma e_g=y$, where $e_g=\tp (0, \dots, 0, 1) \in \R^g$.
Then we have:
\begin{itemize}
\item $\gamma x', \gamma\alpha_0, \dots, \gamma\alpha_g \in \mcH := \{v \in \R^g \mid \brk{v, e_g}>0 \}$, 
\item $\gamma x'$ is in general position relative to $\gamma J$, 
\item $\sgn(\gamma J^{(i)})= \sgn (\det (\gamma)) \sgn (J^{(i)})$, 
\item $\mathbf 1_{C_{J^{(i)}}}(x')=\mathbf 1_{C_{\gamma J^{(i)}}}(\gamma x')$,
\end{itemize}
and hence we can use \cite[Proposition 6.2]{Y10:On-}. This completes the proof.
\end{proof}

\begin{rmk}
It is also possible to prove the last part using \cite[Theorem 2.1]{CDG15:Int}.
\end{rmk}

%%%%%%%%%%%%%%%%%%%%%%%%%%%%%%%%%%%%%%%%%%%%%%%%%%%%%%%%%%%%%%%%%%%%%%%%%%%%%%%%%%%%%%%%%%%%%%%%%%%%%%%%%%%%%%%%%%%%%%%%%%%%%%%%%%%%%%%%%%%%%%%%%%%%%%%%%%%%%%%%%%%%%%%%%%%%%%%%%%%%%%%%%%%%%%%

\section{Construction of the Shintani-Barnes cocycle}\label{sect:shi}

Recall that for $d \geq 0$, we have sheaves
\begin{align*}
\msF_d &= \pi_{\C}^{-1} \Omega^{g-1}_{\PPC}(-d)\big|_{\Yo}, \\
\msF_d^{\Xi} &=  \underline{\Hom}(\underline{\Z}[\Xi], \msF_d) \simeq \prod_{Q\in \Xi} \msF_d,
\end{align*}
on $\Yo= \C^g \setm \bmi \R^g$.
In this section we construct a certain cohomology class in $H^{g-1}(\Yo, SL_g(\Z), \msF_d^{\Xi})$ using the \v{C}ech complex $\msC^{\bullet}(\mcX_{\Q}, \msF_d)^{\Xi}$.

\subsection{Barnes zeta function associated to $C_I^Q$}

Recall that for $I=(\alpha_{1}, \dots, \alpha_{g}) \in (\XQ)^g$, the open subset $V_I \subset \Yo$ is defined as
\[
V_I=\{y \in \Yo \mid \re (\brk{\alpha_i, y})>0,\  i=1, \dots, g \}, 
\]
and we have
\begin{align*}
 \Gamma \big(V_I, \msF_d \big)=
 \Bigg\{ f\omega \ \Bigg| \
\begin{split}
f: \text{ holomorphic function on } \pi_{\C}^{-1}(\pi_{\C}(V_I)) \\
\text{s.t. } f(\lambda y)=\lambda^{-g-d} f(y),\ \forall \lambda \in \C^{\times}
\end{split}
\Bigg\}
 \end{align*}
 by Proposition \ref{local vanishing}. 
Note that $\pi_{\C}^{-1}(\pi_{\C}(V_I)) \subset \C^g \setm \{0\}$ is an open subset of the following form:
\[
%\pi_{\C}^{-1}(\pi_{\C}(V_I))=\{y \in \Yo \mid \exists \lambda \in \C^{\times}, \ \re (\lambda \brk{\alpha_i, y})>0,\  i=1, \dots, g \} \subset \C^g.
\pi_{\C}^{-1}(\pi_{\C}(V_I))=\{y \in \C^g \mid \exists \lambda \in \C^{\times}, \ \lambda y \in V_I \} \subset \C^g \setm \{0\}.
\]

\begin{dfn}\label{dfn:psi}
For $d \geq 1$,  $I=(\alpha_{1}, \dots, \alpha_{g}) \in (\XQ)^g$, $Q \in \Xi$, and $y \in \pi_{\C}^{-1}(\pi_{\C}(V_I))$, set %define
\begin{align}
\psi_{d, I}^Q(y):= \sgn (I) \sum_{\substack{x \in C_I^Q \cap \Z^g \setm \{0\}}} \frac{1}{\brk{x,y}^{g+d}}, \label{eqn:psi}
\end{align}
where $\sgn (I)=\sgn \det (\alpha_1 \dots \alpha_g) \in \{-1, 0, 1\}$, cf. Section \ref{subsec:coc}.
\end{dfn}

\begin{prop}\label{prop:conv}
The infinite series \eqref{eqn:psi} converges absolutely and locally uniformly for $y \in \pi_{\C}^{-1}(\pi_{\C}(V_I))$. In particular, $\psi_{d,I}^Q$ is a holomorphic function on $\pi_{\C}^{-1}(\pi_{\C}(V_I))$. Moreover, we have
\[
\psi_{d,I}^Q(\lambda y)= \lambda^{-g-d}\psi_{d,I}^Q(y)
\]
for all $\lambda \in \C^{\times}$ and $y \in \pi_{\C}^{-1}(\pi_{\C}(V_I))$.
\end{prop}

\begin{proof}
If $\sgn (I)=0$, then by Proposition \ref{prop:rat} (1), we see that $C_I^Q \cap \Z^g \setm \{0\}=\emptyset$, and hence the sum is zero. (In particular, the series converges.)
Therefore, we may assume that $\alpha_1, \dots, \alpha_g$ form a basis of $\Q^g$.
%First if $\sgn (I)=0$, then we have $\psi_{d,I}^Q=0$ by definition. Therefore, we may assume that $\alpha_1, \dots, \alpha_g$ form a basis of $\Q^g$.
%
Furthermore, since $\sgn (I)$ and $C_I^Q$ do not change if we replace $\alpha_i$ by its multiple by positive integers, we may assume that $\alpha_1, \dots, \alpha_g \in \Z^g \setm \{0\}$. 

Let $y \in \pi_{\C}^{-1}(\pi_{\C}(V_I))$ and take $\lambda \in \C^{\times}$ such that $\lambda y \in V_I$. Then take a relatively compact open neighborhood $U \subset V_I$ of $\lambda y$, 
i.e., $U$ is an open neighborhood of $\lambda y$ such that its closure $\overline U$ is compact and $\overline U \subset V_I$. 
Since $y \in \lambda^{-1} \overline{U} \subset \pi_{\C}^{-1}(\pi_{\C}(V_I))$, it suffices to show that \eqref{eqn:psi} converges absolutely and uniformly on $\lambda^{-1}\overline U$.

First, note that by the definition of $C_I^Q$, we have
\[
C_I^Q \subset \overline{C_I}=\sum_{i=1}^g \R_{\geq 0} \alpha_i,
\]
%where $\overline{C_I}$ is the (topological) closure of $C_I$ in $\R^g$.
where $\overline{C_I}$ is the closed cone generated by $I$. 
Put
\[
R_I:=\sum_{i=1}^g [0,1) \alpha_i.
\]
Then we see
\begin{itemize}
\item $C_I^Q \cap \Z^g \subset \overline{C_I} \cap \Z^g
= \left\{x + \sum_{i=1}^g n_i \alpha_i  \ \Bigg|\  x \in  R_I \cap \Z^g,  n_i \in \Z_{\geq 0} \right\}$, 
\item $R_I \cap \Z^g$ is a finite set, 
%\item $\{\re (\brk{\alpha_i, y'}) \mid y' \in \overline U \}$ is a compact subset of $\Rpos$ for $i=1, \dots, g$, 
\item $\{ \re (\brk{x, y'}) \ |\  x \in R_I \cap \Z^g \setm \{0\}, y' \in \overline U \}$ is a compact subset of $\Rpos$. 
\end{itemize}
Therefore, set
\[
b:=\min \{ \re (\brk{x, y'}) \ |\  x \in R_I \cap \Z^g \setm \{0\}, y' \in \overline U \} >0.
\]
Moreover, for $i=1, \dots, g$, set
\[
a_i:= \min \{\re (\brk{\alpha_i, y'}) \mid y' \in \overline U \} >0.
\]
Then for $y'' =\lambda^{-1}y' \in \lambda^{-1}\overline U$, where $y' \in \overline U$, we have
\begin{align*}
\sum_{\substack{x \in C_I^Q \cap \Z^g \setm \{0\}}} \left| \frac{1}{\brk{x,y''}^{g+d}} \right|
&\leq
\left| \lambda \right|^{g+d} \sum_{\substack{x \in \overline{C_I} \cap \Z^g \setm \{0\}}} \frac{1}{\left|\brk{x,y'}\right|^{g+d}}
\leq
\left| \lambda \right|^{g+d} \sum_{\substack{x \in \overline{C_I} \cap \Z^g \setm \{0\}}} \frac{1}{(\re (\brk{x,y'}))^{g+d}} \\
%&\leq
%\left| \lambda \right|^{-g-d} \sum_{\substack{x' \in R_I \cap \Z^g, (n_1, \dots, n_g) \in (\Z_{\geq 0})^g \\ x' + \sum_{i=1}^g n_i \alpha_i \neq 0}} \frac{1}{( \brk{x', y'} + \sum_{i=1}^g n_i \brk{\alpha_i, y'})^{g+d}} \\
&\leq
\left| \lambda \right|^{g+d} \sum_{\substack{x' \in R_I \cap \Z^g, (n_1, \dots, n_g) \in (\Z_{\geq 0})^g \\ x' + \sum_{i=1}^g n_i \alpha_i \neq 0}} \frac{1}{(\re (\brk{x', y'}) + \sum_{i=1}^g n_i \re (\brk{\alpha_i, y'}))^{g+d}} \\
&\leq
\left| \lambda \right|^{g+d} \sum_{(n_1, \dots, n_g) \in (\Z_{\geq 0})^g\setm \{0\}} \frac{1}{(\sum_{i=1}^g n_i a_i)^{g+d}} \\
&~ + \left| \lambda \right|^{g+d}\#(R_I \cap \Z^g\setm \{0\}) \sum_{(n_1, \dots, n_g) \in (\Z_{\geq 0})^g}  \frac{1}{( b + \sum_{i=1}^g n_i a_i)^{g+d}}, 
\end{align*}
where $\#(R_I \cap \Z^g\setm \{0\})$ is the order of the finite set $R_I \cap \Z^g\setm \{0\}$. 
It is now clear that the last two series converge for $d\geq 1$. The last statement in the proposition follows directly from the definition. %This completes the proof.
\end{proof}

\begin{rmk}\label{rmk:decomp eis}
Since $C_I^Q$ is a rational constructible cone, cf. Proposition \ref{prop:rat}, we see that $\psi_{d,I}^Q$ can be written as a sum of a finite number of the Barnes zeta functions, cf.~\cite{10019989677}, \cite{Y10:On-}. 
Conceptually, we may also view $\psi_{d,I}^Q$ as a decomposed piece of the ``Eisenstein series'' along the cone $C_I^Q$. 
%Conceptually, $\psi_{d,I}^Q$ can be seen as a decomposed piece of the ``Eisenstein series'' along the cone $C_I^Q$. 
\end{rmk}

\begin{cor}\label{cor:psi}
Let $d\geq 1$. For $I=(\alpha_{1}, \dots, \alpha_{g}) \in (\XQ)^g$ and $Q \in \Xi$, we have
\[
\psi_{d,I}^Q \omega  \in \Gamma(V_I, \msF_d), % \simeq \Gamma (\pi_{\C}(V_I), \Omega^{g-1}_{\PPC}(-d)),
\]
where
\[
\omega(y)= \sum_{i=1}^g (-1)^{i-1} y_i dy_1 \wedge \cdots \wedge \check{dy_i} \wedge \cdots \wedge dy_g.
\]
\end{cor}

\begin{proof}
This follows directly from Proposition \ref{local vanishing} (2) and Proposition \ref{prop:conv}.
\end{proof}

%%%%%%%%%%%%%%%%%%%%%%%%%%%%%%%%%%%%%%%%%%%%%%%%%%%%%%%%%%%%%%%%%%%%%%%%%%%%%%%%%%%%

\subsection{The Shintani-Barnes cocycle}

\begin{dfn}
For $d\geq 1$, we define a map
\[
%\Psi_d \in \Map (\Xi, C^{g-1}(\mcX_{\Q}, \msF_d)) = \Hom (\Z[\Xi], C^{g-1}(\mcX_{\Q}, \msF_d))
\Psi_d: \Xi \ra  C^{g-1}(\mcX_{\Q}, \msF_d)
%\Psi_d: \Xi \ra  C^{g-1}(\mcX_{\Q}, \msF_d)=\prod_{I \in (\XQ)^g} \Gamma(V_I, \msF_d)
\]
by
\[
%\Psi_d: \Xi \ra C^{g-1}(\mcX_{\Q}, \msF_d)= \prod_{I \in (\XQ)^g} \Gamma(V_I, \msF_d); \ Q \mapsto \psi_{d,I}^Q \omega.
%\Psi_d(Q) := (\psi_{d,I}^Q \omega)_{I \in (\XQ)^g} \in \prod_{I \in (\XQ)^g} \Gamma(V_I, \msF_d)=  C^{g-1}(\mcX_{\Q}, \msF_d)
\Psi_d(Q) := (\psi_{d,I}^Q \omega)_{I \in (\XQ)^g} \in C^{g-1}(\mcX_{\Q}, \msF_d) = \prod_{I \in (\XQ)^g} \Gamma(V_I, \msF_d)
%\Psi_d(Q) := (\psi_{d,I}^Q \omega)_{I \in (\XQ)^g} 
\]
for $Q \in \Xi$.
\end{dfn}

The aim of this subsection is to show that $\Psi_d$ defines a class in $H^{g-1}(\Yo, SL_g(\Z), \msF_d^{\Xi})$ via Corollary \ref{cor:equiv coh}.

\begin{prop}\label{prop:psi eqiv}
The map $\Psi_d$ is a $SL_g(\Z)$-equivariant map, i.e., we have
\[
\Psi_d([\gamma](Q)) = [\gamma](\Psi_d(Q))
\]
for $Q \in \Xi$ and $\gamma \in SL_g(\Z)$. In other words, we have
\[
%\Psi_d \in  \Hom_{SL_g(\Z)} (\Z[\Xi], C^{g-1}(\mcX_{\Q}, \msF_d))=\Gamma(\Yo, SL_g(\Z), \msC^{g-1}(\mcX_{\Q}, \msF_d)^{\Xi}).
\Psi_d \in  \Map_{SL_g(\Z)} (\Xi, C^{g-1}(\mcX_{\Q}, \msF_d))=\Gamma(\Yo, SL_g(\Z), \msC^{g-1}(\mcX_{\Q}, \msF_d)^{\Xi}).
\]
\end{prop}

\begin{proof}
Let $I=(\alpha_{1}, \dots, \alpha_{g}) \in (\XQ)^g$. We need to show
\[
\Psi_d([\gamma](Q))_I = ([\gamma](\Psi_d(Q)))_I \in \Gamma(V_I, \msF_d),
\]
where $\Psi_d([\gamma](Q))_I $ (resp. $ ([\gamma](\Psi_d(Q)))_I$) is the $I$-th component of $\Psi_d([\gamma](Q))$ (resp. $ [\gamma](\Psi_d(Q))$) as always.
%
%Indeed, by Lemma \ref{action on omega}, Definition \ref{dfn gamma on U}, \eqref{eqn gamma sI}, and Lemma \ref{lem:gammaCQ}, we have
Indeed, we have
\begin{align*}
([\gamma](\Psi_d(Q)))_I (y)&=( [\gamma](\psi_{d,\gamma^{-1}I}^Q\omega))(y) \\
&=
\psi_{d,\gamma^{-1}I}^Q(\tp{\gamma}y)\omega(\tp{\gamma}y) \\
&=
\sgn (\gamma^{-1}I) \sum_{\substack{x \in C_{\gamma^{-1}I}^Q \cap \Z^g \setm \{0\}}} \frac{ \omega (\tp{\gamma}y)}{\brk{x,\tp{\gamma}y}^{g+d}} \\
&=
\sgn (\det (\gamma^{-1}))\sgn (I) \det (\tp{\gamma}) \sum_{\substack{x \in C_{\gamma^{-1}I}^Q \cap \Z^g \setm \{0\}}} \frac{\omega (y) }{\brk{\gamma x, y}^{g+d}} \\
&=
\sgn (I) \sum_{\substack{x \in \gamma(C_{\gamma^{-1}I}^Q) \cap \Z^g \setm \{0\}}} \frac{\omega (y) }{\brk{x, y}^{g+d}} \\
&=
\sgn (I) \sum_{\substack{x \in C_{I}^{[\gamma](Q)} \cap \Z^g \setm \{0\}}} \frac{\omega (y) }{\brk{x, y}^{g+d}}
=
\Psi_d([\gamma](Q))_I(y)
\end{align*}
for $y \in \pi_{\C}^{-1}(\pi_{\C}(V_I))$. 
%Here, the first equality follows from \eqref{eqn gamma sI}, the second equality follows from Definition \ref{dfn gamma on U}, the fourth equality follows from Lemma \ref{action on omega}, and the last equality follows from Lemma \ref{lem:gammaCQ}. 
Here, the first and the second equality follow from the definition of $[\gamma]$ (cf. \eqref{eqn gamma sI} and Definition \ref{dfn gamma on U}), the fourth equality follows from Lemma \ref{action on omega}, and the sixth equality follows from Lemma \ref{lem:gammaCQ}. 
%This proves the proposition.
\end{proof}

\begin{cor}\label{cor:psi inv}
For $Q \in \Xi$, we have
\[
\Psi_d(Q) \in C^{g-1}(\mcX_{\Q}, \msF_d)^{\Gamma_Q}=\Gamma(\Yo, \Gamma_Q, \msC^{g-1}(\mcX_{\Q}, \msF_d)).
\]
\end{cor}

\begin{proof}
Since $\Gamma_Q$ is the stabilizer of $Q$ in $SL_g(\Z)$ and $\Psi_d$ is a $SL_g(\Z)$-equivariant map, it follows that $\Psi_d(Q)$ is a $\Gamma_Q$-invariant element.
\end{proof}

\begin{prop}\label{prop:psi cocycle}
\begin{enumerate}
\item Let $Q \in \Xi$. 
%Under the differential map
%\[
%\dd^{g-1}: C^{g-1}(\mcX_{\Q}, \msF_d) \ra C^{g}(\mcX_{\Q}, \msF_d),
%\]
%cf. Section \ref{subsect:cech cpx}, we have
We have 
\[
\dd^{g-1}(\Psi_d(Q))=0
\]
under the differential map
\[
\dd^{g-1}: C^{g-1}(\mcX_{\Q}, \msF_d) \ra C^{g}(\mcX_{\Q}, \msF_d). 
\]
%for all $Q \in \Xi$. %In particular,
\item 
We have
\[
\dd^{g-1}(\Psi_d)=0 
\]
under the differential map
\[
\dd^{g-1}: \Gamma \big(\Yo, SL_g(\Z), \msC^{g-1}(\mcX_{\Q}, \msF_d)^{\Xi} \big) \ra \Gamma \big(\Yo, SL_g(\Z), \msC^{g}(\mcX_{\Q}, \msF_d)^{\Xi}\big). 
\]
In the following, we refer to $\Psi_d$ as the Shintani-Barnes cocycle.
\end{enumerate}
\end{prop}

\begin{proof}
%(2) follows directly from (1). We prove (1).
(1) 
Let $J=(\alpha_{0}, \dots, \alpha_{g}) \in (\XQ)^{g+1}$. For $i=0, \dots, g$, put $J^{(i)} = (\alpha_{0}, \dots, \check{\alpha_{i}}, \dots, \alpha_{g}) \in (\XQ)^g$. We need to show
\begin{align}\label{eqn:Psi cocycle}
(\dd^{g-1}(\Psi_d(Q)))_J=\sum_{i=0}^g (-1)^i \Psi_d(Q)_{J^{(i)}} \big|_{V_J}=0.
\end{align}
First if $V_J=\emptyset$, then \eqref{eqn:Psi cocycle} is obvious because $\Gamma(\emptyset, \msF_d)=0$. Assume $V_J \neq \emptyset$, and take $y' \in V_J$. 
%Then we have $\brk{\alpha_i, \re(y')}=\re(\brk{\alpha_i,y'})>0$ for all $i=0, \dots, g$. 
Then we have $\brk{\alpha_i, \re(y')}=\re(\brk{\alpha_i,y'})>0$ for all $i=0, \dots, g$, and hence the assumption in Proposition \ref{prop:cocycle} is satisfied. 
%Hence by Proposition \ref{prop:cocycle}, we have
Therefore, by Proposition \ref{prop:cocycle}, we find
\begin{align*}
\sum_{i=0}^g (-1)^i \Psi_d(Q)_{J^{(i)}} \big|_{V_J}(y)
&=
%\left( \sum_{i=0}^g (-1)^i \sgn (J^{(i)}) \sum_{\substack{x \in C_{J^{(i)}}^Q \cap \Z^g\\ x \neq 0}} \frac{1}{\brk{x,y}^{g+d}} \right)\omega(y)=0
\sum_{i=0}^g (-1)^i \sgn (J^{(i)}) \sum_{\substack{x \in C_{J^{(i)}}^Q \cap \Z^g \setm \{0\}}} \frac{1}{\brk{x,y}^{g+d}} \omega(y) \\
&=
\sum_{x \in \Z^g \setm \{0\}} \left( \sum_{i=0}^g (-1)^i  \sgn (J^{(i)}) \mathbf 1_{C_{J^{(i)}}^Q} (x)\right)\frac{\omega(y)}{\brk{x,y}^{g+d}} \\
&=0 
\end{align*}
for $y \in \pi_{\C}^{-1}(\pi_{\C}(V_J))$. This proves (1). 

(2) follows from (1).
%This proves the proposition.
\end{proof}

We obtain the following.
%By combining Corollary \ref{cor:equiv coh}, Proposition \ref{prop:psi eqiv}, Corollary \ref{cor:psi inv} and Proposition \ref{prop:psi cocycle}, we obtain the following. 
\begin{thm}\label{thm:SB class}
 %For $d\geq 1$, the map $\Psi_d$ defines a class
 For $d\geq 1$, the Shintani-Barnes cocycle $\Psi_d$ defines a class
 \[
 [\Psi_d] \in H^{g-1} \big(\Yo, SL_g(\Z), \msF_d^{\Xi}\big).
 \]
%We refer to this class $[\Psi_d]$ as the Shintani-Barnes cocycle.
Moreover, for $Q \in \Xi$, the element $\Psi_d(Q) \in C^{g-1}(\mcX_{\Q}, \msF_d)^{\Gamma_Q}$ defines a class
 \[
 [\Psi_d(Q)] \in H^{g-1}(\Yo, \Gamma_Q, \msF_d),
 \]
 and we have
 \[
 \ev_Q([\Psi_d])=[\Psi_d(Q)]. 
 \]
\end{thm}

\begin{proof}
This follows from Corollary \ref{cor:equiv coh}, Proposition \ref{prop:psi eqiv}, Corollary \ref{cor:psi inv}, and Proposition \ref{prop:psi cocycle}. 
%The first statement follows from Proposition \ref{prop:psi eqiv} and Proposition \ref{prop:psi cocycle}, and the
\end{proof}

%end of part I

%%%%%%%%%%%%%%%%%%%%%%%%%%%%%%%%%%%%%%%%%%%%%%%%%%%%%%%%%%%%%%%%%%%%%%%%%%%%%%%%%%%%%%%%%%%%%%%%%%%%%%%%%%%%%%%%%%%%%%%%%%%%%%%%%%%%%%%%

\section{Integration}\label{sect:int}
The goal of the remaining sections is to construct a specialization map \eqref{eqn:spec}, and prove that the Shintani-Barnes cocycle class $[\Psi_{d}]$ specializes to the special value of the zeta functions of number fields, cf. Theorem \ref{thm:main thm}. 
%The goal of the remaining sections is to construct a specialization map \eqref{eqn:spec}, and prove our main theorem: Theorem \ref{thm:main thm}. 

%The aim of the remaining sections is to obtain the special values of Dedekind zeta function of number fields of degree $g$ as a specialization of the class $[\Psi_d]$.
%In the remaining sections, we show that we can get the values of the zeta functions of number fields as a specialization of the Shintani-Barnes cocycle $[\Psi_d]$.

Let $Q \in \Xi$ be fixed throughout this section. 
In this section we define an integral map
\[
\int_Q: H^q (\Yo, \Gamma_Q, \msF_0) \ra H^q_Q(\Yo, \Gamma_Q, \shfC), 
\]
where $H^q_Q(\Yo, \Gamma_Q, \shfC)$ is a certain auxiliary cohomology group defined later, cf. Section \ref{subsect:intQ}. 
%This group $H^q_Q(\Yo, \Gamma_Q, \shfC)$ will be studied more closely in Section \ref{subsect:comp HQ} using a topological method.  
This group $H^q_Q(\Yo, \Gamma_Q, \shfC)$ will be studied more closely in Section \ref{sect:spec} using a topological method.  
%, which terns out to be isomorphic to the singular cohomology of 

\subsection{Integration and the Hurwitz formula}\label{subsect:Hur}

For $q \geq 0$, let
\[
\Delta^q:=\left\{(t_1, \dots, t_{q+1}) \in \R^{q+1} \ \Bigg| \  \sum_{i=1}^{q+1}t_i=1,\ t_i \geq 0  \right\}
\]
denote the standard $q$-simplex. 
%20210228
Note that we can also embed $\Delta^q$ into $\R^q$ by
\begin{align*}
\Delta^q \hookrightarrow \R^q; (t_1, \dots, t_{q+1}) \mapsto (t_2, \dots, t_{q+1}), 
\end{align*}
and we equip $\Delta^q$ with an orientation induced from the standard orientation of $\R^q$. 
Moreover, for $\xi_1, \dots, \xi_{q+1} \in \C^g \setm \{0\}$, let
\[
\sigma_{(\xi_1, \dots, \xi_{q+1})}: \Delta^q \ra \C^g; (t_1, \dots, t_{q+1}) \mapsto \sum_{i=1}^{q+1} t_i \xi_i
\]
%denote the affine $q$-simplex with vertices $\xi_1, \dots, \xi_{q+1}$.
denote the affine $q$-simplex with vertices $\xi_1, \dots, \xi_{q+1}$, and let
\[
|\sigma_{(\xi_1, \dots, \xi_{q+1})}| := \sigma_{(\xi_1, \dots, \xi_{q+1})} (\Delta^q) \subset \C^g
\]
denote the image of $\sigma_{(\xi_1, \dots, \xi_{q+1})}$.

%Now, let $U \subset \C^g \setm \{0\}$ be a convex open subset and let $\xi_1, \dots, \xi_g \in \C^g\setm \{0\}$ be a basis of $\C^g$ such that
%\begin{align*}
%\xi_1, \dots, \xi_g \in U.
%\end{align*}
Now, let $U \subset \C^g \setm \{0\}$ be a convex open subset and let $\xi_1, \dots, \xi_g \in U$ be a basis of $\C^g$. 
Then for a homogeneous holomorphic function $f$ on $ \pi_{\C}^{-1}(\pi_{\C}(U))$ of degree $-g$, (i.e., $f(\lambda y)=\lambda^{-g} f(y)$ for all $\lambda \in \C^{\times}$,) we consider the integral
\begin{align}\label{eqn:int1}
\int_{\sigma_{(\xi_1, \dots, \xi_g)}} f\omega := \int_{\Delta^{g-1}}  (\sigma_{(\xi_1, \dots, \xi_g)})^{\ast}(f\omega), 
\end{align}
where
\[
\omega(y)= \sum_{i=1}^g (-1)^{i-1} y_i dy_1 \wedge \cdots \wedge \check{dy_i} \wedge \cdots \wedge dy_g.
\]
Here note that $f\omega$ is a holomorphic $(g-1)$-form on $\pi_{\C}^{-1}(\pi_{\C}(U)) \supset U$, and we have $|\sigma_{(\xi_1, \dots, \xi_g)}| \subset U$ since $U$ is convex.

\begin{rmk}\label{rmk:int}
Note that via the identification \eqref{eqn:hol forms}, the above $f\omega$ corresponds to a holomorphic $(g-1)$-form on $\pi_{\C}(U) \subset \PPC$.   
More precisely, there exists a holomorphic $(g-1)$-form $\eta$ on $\pi_{\C}(U) \subset \PPC$ such that
\[
(\pi_{\C})^*\eta = f\omega.
\]
Then we see that the integral \eqref{eqn:int1} is actually an integral on $\PPC$:
\[
\int_{\sigma_{(\xi_1, \dots, \xi_g)}} f\omega = \int_{\pi_{\C} \circ \sigma_{(\xi_1, \dots, \xi_g)}} \eta.
\]
\end{rmk}

\begin{lem}\label{lem:int indep}
Let $U \subset \C^g \setm \{0\}$ be a convex open subset, and let $\xi_1, \dots, \xi_g \in \C^g\setm \{0\}$ be a basis of $\C^g$ such that
%Let $U \subset \C^g \setm \{0\}$ be a convex open subset, and let $\xi_1, \dots, \xi_g \in \C^g\setm \{0\}$ be a basis of $\C^g$ such that
\begin{align*}
\xi_1, \dots, \xi_g \in U.
\end{align*}
Furthermore, let $\lambda_1, \dots, \lambda_g \in \C^{\times}$ be any complex numbers such that
\begin{align*}
\lambda_1 \xi_1, \dots, \lambda_g \xi_g \in U.
\end{align*}
Then for a homogeneous holomorphic function $f$ on $ \pi_{\C}^{-1}(\pi_{\C}(U))$ of degree $-g$,
we have
\begin{align*}
\int_{\sigma_{(\xi_1, \dots, \xi_g)}} f\omega = \int_{\sigma_{(\lambda_1 \xi_1, \dots, \lambda_g \xi_g)}} f\omega.
\end{align*}
\end{lem}

\begin{proof}
Let
\begin{align}\label{eqn:h}
\begin{split}
h: [0,1] \times \Delta^{g-1} \ra U; (u,t) \mapsto &u \sigma_{( \xi_1, \dots,  \xi_g)}(t) + (1-u) \sigma_{(\lambda_1 \xi_1, \dots, \lambda_g \xi_g)}(t) \\
&= \sum_{i=1}^g (u +(1-u)\lambda_i)t_i \xi_i
\end{split}
\end{align}
be a homotopy between $\sigma_{( \xi_1, \dots,  \xi_g)}$ and $\sigma_{(\lambda_1 \xi_1, \dots, \lambda_g \xi_g)}$. Note that we have
$h(u,t) \in U$
because $U$ is convex. 
We regard $h$ as a singular $g$-chain in a usual way using the standard decomposition of the prism $[0,1]\times \Delta^{g-1}$, cf. \cite[Section 2.1, Proof of 2.10]{H02:Alg}. 
Then we have
\[
\partial h= \sigma_{( \xi_1, \dots,  \xi_g)} - \sigma_{(\lambda_1 \xi_1, \dots, \lambda_g \xi_g)} + h',
\]
%where $h'$ is a singular $(g-1)$-chain associated to 
where
\[
%h': [0,1]\times \partial \Delta^{g-1} \ra U; (u,t) \mapsto h(u,t).
h': [0,1]\times \partial \Delta^{g-1} \ra U; (u,t) \mapsto h(u,t), 
\]
which is also regarded as a singular $(g-1)$-chain. 
Let $\xi_1^*, \dots, \xi_g^* \in \C^g$ be the dual basis of $\xi_1, \dots, \xi_g$, and let
\[
Z:=\bigcup_{i=1}^g \{y \in \C^g \mid \brk{\xi^*_i,y}=0 \}
\]
be the union of hyperplanes defined by $\xi_1^*, \dots, \xi_g^*$. Then by \eqref{eqn:h}, we easily see
\[
h' ( [0,1]\times \partial \Delta^{g-1}) \subset Z.
\]

Now, by Remark \ref{rmk:int}, there exists a holomorphic $(g-1)$-form $\eta$ on $\pi_{\C}(U)$ such that
\[
(\pi_{\C})^*\eta =f\omega.
\]
In particular, we have
\[
d(f\omega)=(\pi_{\C})^*(d\eta)=0,
\]
where $d$ is the usual derivative of differential forms.
Moreover, we also have
\[
\int_{h'} f\omega=\int_{\pi_{\C}\circ h'}\eta = 0
\]
because $\pi_{\C}\circ h'$ is contained in a divisor $\pi_{\C}(Z\setm\{0\}) \subset \PPC$.
Therefore, we obtain
\begin{align*}
0 &= \int_{h} d(f\omega)
=
\int_{\partial h} f\omega\\
&=
\int_{ \sigma_{( \xi_1, \dots,  \xi_g)}} f\omega -\int_{\sigma_{(\lambda_1 \xi_1, \dots, \lambda_g \xi_g)}}f\omega +\int_{h'} f\omega \\
&=
\int_{ \sigma_{( \xi_1, \dots,  \xi_g)}} f\omega -\int_{\sigma_{(\lambda_1 \xi_1, \dots, \lambda_g \xi_g)}}f\omega. 
\end{align*}
%This completes the proof.
This completes the proof.
\end{proof}

%Actually the Hurwitz formula motivates the definition of the Shintani-Barnes cocycle
%Actually the definition of the Shintani-Barnes cocycle is motivated by the Hurwitz formula

%An important example of such an integral is the following Hurwitz formula, cf. \cite{MR1512118}. 
An important example of such an integral is the following Hurwitz formula, cf. \cite{MR1512118}, \cite{S93:Eis}, which is also known as the Feynman parametrization. 
%Actually, the definition of our Shintani-Barnes cocycle is motivated by the Hurwitz formula. 

\begin{prop}[\cite{MR1512118}]\label{prop:Hur formula}
Let $x \in \C^g \setm \{0\}$, and let $\xi_1, \dots, \xi_g \in \C^g \setm \{0\}$ be a basis of $\C^g$ such that
\[
\xi_1, \dots, \xi_g \in V_x=\{y \in \C^g \setm \{0\} \ |\ \re (\brk{x,y})>0 \}. %\text{, i.e., } \re (\brk{x, \xi_i})>0, \ \forall i.
\]
%i.e., $\re (\brk{x, \xi_i})>0$ for $i=1, \dots, g$. 
\begin{enumerate}
\item We have
\[
\int_{\sigma_{(\xi_1, \dots, \xi_g)}} \frac{\omega(y)}{\brk{x,y}^g} = \frac{1}{(g-1)!} \frac{\det (\xi_1, \dots, \xi_g)}{\brk{x,\xi_1}\cdots \brk{x,\xi_g}}.
\]
\item Let $\xi_1^*, \dots, \xi_g^* \in \C^g$ be the dual basis of $\xi_1, \dots, \xi_g$, and let $\underline k =(\tup{k}{g}) \in (\Z_{\geq 0})^g$. Then we have
\[
\int_{\sigma_{(\xi_1, \dots, \xi_g)}} \brk{\xi^*_1,y}^{k_1}\cdots  \brk{\xi^*_g,y}^{k_g}  \frac{\omega(y)}{\brk{x,y}^{g+\abs{\underline k}}} = \frac{\underline k!}{(g+\abs{\underline k}-1)!} \frac{\det (\xi_1, \dots, \xi_g)}{\brk{x,\xi_1}^{k_1+1}\cdots \brk{x,\xi_g}^{k_g+1}},
\]
where $\abs{\underline k}:=k_1+\cdots +k_g$ and $\underline k!:= k_1!\cdots k_g!$.
\end{enumerate}
\end{prop}

\begin{proof}
(1) Let $W:=(\xi_1, \dots, \xi_g) \in GL_g(\C)$ be the matrix whose columns  are $\xi_1, \dots, \xi_g$ so that the $(g-1)$-simplex $\sigma_{(\xi_1, \dots, \xi_g)}$ is represented by the linear transformation $W$.  Then we have
\begin{align*}
\int_{\sigma_{(\xi_1, \dots, \xi_g)}} \frac{\omega(y)}{\brk{x,y}^g}
=
\int_{\Delta^{g-1}} \frac{\omega(Wy)}{\brk{x,Wy}^g}
=
\det W \int_{\Delta^{g-1}} \frac{\omega(y)}{\brk{\tp{W}x,y}^g}. 
\end{align*}
For $i=1, \dots, g$, put
\[
a_i:=\brk{x,\xi_i} \neq 0,
\]
and let $\tup{e}{g} \in \C^g$ be the standard basis, i.e., $e_i=\tp{(0, \dots, 0, \stackrel{i}{1}, 0, \dots, 0)}$.
Then we find
\begin{align*}
\det W \int_{\Delta^{g-1}} \frac{\omega(y)}{\brk{\tp{W}x,y}^g}
&=
\frac{\det W}{a_1\cdots a_g} \int_{\Delta^{g-1}} \frac{\omega((a_1y_1, \dots, a_gy_g))}{(a_1y_1+\cdots +a_gy_g)^g}\\
&=
\frac{\det W}{a_1\cdots a_g} \int_{\sigma_{(a_1e_1, \dots, a_ge_g)}} \frac{\omega(y)}{(y_1+\cdots +y_g)^g} \\
&=
\frac{\det W}{a_1\cdots a_g} \int_{\sigma_{(e_1, \dots, e_g)}} \frac{\omega(y)}{(y_1+\cdots +y_g)^g} \\
&=
\frac{\det W}{a_1\cdots a_g} \int_{\sigma_{(e_1, \dots, e_g)}} \omega(y) \\
&=
\frac{1}{(g-1)!}\frac{\det W}{a_1\cdots a_g}.
\end{align*}
Here, the third equality follows from Lemma \ref{lem:int indep}, and the last equality follows from an elementary computation.
This proves (1).

(2) First note that for fixed $\xi_1, \dots, \xi_g$, the formula (1) can be seen as an equality of holomorphic functions in $x$-variable.
Thus, for $1 \leq i \leq g$, we consider a linear differential operator
\[
D_i:=\brk{\xi_i^*, \frac{\partial}{\partial x}}=\xi^*_{i1}\frac{\partial}{\partial x_1} +\cdots + \xi^*_{ig}\frac{\partial}{\partial x_g},
\]
where $\xi^*_{ij}$ is the $j$-th component of $\xi^*_i$. 
Then we can compute the action of $D_i$ on the both sides of (1) using the following formula:
\begin{align*}
D_i \frac{1}{\brk{x,y}^n} =-n \brk{\xi^*_i, y} \frac{1}{\brk{x,y}^{n+1}},
\end{align*}
where $y \in \C^g$, $\brk{x,y}\neq 0$, and $n\geq 1$. 
Now (2) follows from (1) by applying the operator
\[
D_1^{k_1}\cdots D_g^{k_g}
\]
to (1). 
This completes the proof.
\end{proof}

\begin{rmk}
%In \cite{S93:Eis} Sczech mentions the formula (1). % in the case of $\R$-coefficient. Although Sczech does not use these formulas, the right hand side of these formulas are exactly the building blocks of Sczech's Eisenstein cocycle. %In addition, the idea of considering the linear differential operator $D_i$ is also inspired by Sczech's argument in \cite{S93:Eis}. 
%It should be noted that the Hurwitz formula is also the starting point of Sczech's Eisenstein cocycle \cite{S93:Eis}. Indeed, the right hand side of Proposition \ref{prop:Hur formula} (1) is exactly the building block of Sczech's Eisenstein cocycle. 
It should be noted that the right hand side of the Hurwitz formula (Proposition \ref{prop:Hur formula}) is exactly the building block of Sczech's Eisenstein cocycle~\cite{S93:Eis}. 
\end{rmk}

%%%%%%%%%%%%%%%%%%%%%%%%%%%%%%%%%%%%%%%%%%%%%%%%%%%%%%%%%%%%%%%%%%%%%%%%%%%%%%%%%

\subsection{The Integral map $\int_Q$}\label{subsect:intQ}

Let $Q \in \Xi$, and let $\theta^{(1)}, \dots, \theta^{(g)} \in \C$ be the distinct eigenvalues of $\tp{Q}$. Note that by Proposition \ref{lem:tor} (1), $\tp{Q}$ has $g$ distinct eigenvalues.
%In this subsection we fix the order of the eigenvalues $\theta^{(1)}, \dots, \theta^{(g)}$. Cf. Remark \ref{rmk:intQ}

%In this subsection we define an integral map $\int_Q$ from $H^q(\Yo, \Gamma_Q, \msF_0)$ to a certain cohomology group $H_Q^q(\Yo, \Gamma_Q, \C)$.
%We define an integral map $\int_Q$ from $H^q(\Yo, \Gamma_Q, \msF_0)$ to a certain cohomology group $H_Q^q(\Yo, \Gamma_Q, \C)$.
%To this end we construct the integral map locally on $V_I$ for a certain suitable family of $I$:
%
In this subsection we introduce an auxiliary cohomology group $H_Q^q(\Yo, \Gamma_Q, \shfC)$ and define the integral map $\int_Q$. 
\begin{dfn}\label{dfn:Q-int}
%Let $q \geq 0$. We say that $I \in (\XQ)^{q+1}$ is $Q$-admissible if we can take the system of eigenvectors $\xi_1, \dots, \xi_g$ of $\tp{Q}$ in $V_I$.
%Let $q \geq 0$. We say that $I \in (\XQ)^{q+1}$ is $Q$-admissible if for all $i=1, \dots, g$ we can take an eigenvector $\xi_i $ of $\tp{Q}$ with eigenvalue $\theta^{(i)}$ such that $\xi_i \in V_I$.
%Let $q \geq 0$. We say that $I \in (\XQ)^{q+1}$ is $Q$-admissible if there exist $\xi_1, \dots, \xi_g \in V_I$ such that $\xi_i$ is an eigenvector of $\tp{Q}$ with eigenvalue $\theta^{(i)}$ for $i=1, \dots, g$.
%Let $q \geq 0$. We say that $I \in (\XQ)^{q+1}$ is \textit{$Q$-admissible} if a system of eigenvectors $\xi_1, \dots, \xi_g$ of $\tp{Q}$ can be taken in $V_I$, i.e., 
%
%%%%Let $q \geq 0$. We say that $I \in (\XQ)^{q+1}$ is \textit{$Q$-admissible} if there exists a system of eigenvectors $\xi_1, \dots, \xi_g$ of $\tp{Q}$ in $V_I$, i.e., 
%
Let $q \geq 0$. We say that $I \in (\XQ)^{q+1}$ is \textit{$Q$-admissible} if we can take a system of eigenvectors $\xi_1, \dots, \xi_g$ of $\tp{Q}$ in $V_I$, i.e., 
%Let $q \geq 0$. We say that $I \in (\XQ)^{q+1}$ is \textit{$Q$-admissible} if the eigenvectors $\xi_1, \dots, \xi_g$ of $\tp{Q}$ can be taken in $V_I$, i.e., 
%\[
%\exists \ \xi_1, \dots, \xi_g \in V_I \text{ s.t. } \tp{Q}\xi_i=\theta^{(i)} \xi_i \text{ for } i=1, \dots, g.
%\]
%Let $q \geq 0$. We say that $I \in (\XQ)^{q+1}$ is \textit{$Q$-admissible} if there exist $\xi_1, \dots, \xi_g \in V_I$ such that 
%Let $q \geq 0$. We say that $I \in (\XQ)^{q+1}$ is \textit{$Q$-admissible} if there exist $\xi_1, \dots, \xi_g \in V_I$ such that 
\[
\exists \ \xi_1, \dots, \xi_g \in V_I \text{ s.t. } \tp{Q}\xi_i=\theta^{(i)} \xi_i \text{ for } i=1, \dots, g.
%\tp{Q}\xi_i=\theta^{(i)} \xi_i %\text{ for } i=1, \dots, g.
\]
%for $i=1, \dots, g$. 
We define $(\XQ)^{q+1}_Q$ to be the set of all $Q$-admissible elements of $(\XQ)^{q+1}$.
\end{dfn}

%Now, recall that by Proposition \ref{local vanishing} we have
Recall that we have
 \begin{align*}
 \Gamma(V_I, \msF_0)=
 \Bigg\{ f\omega \ \Bigg| \
\begin{split}
f: \text{ holomorphic function on }\pi_{\C}^{-1}(\pi_{\C}(V_I)) \\
\text{s.t. } f(\lambda y)=\lambda^{-g} f(y),\ \forall \lambda \in \C^{\times}
\end{split}
\Bigg\}.
\end{align*}

\begin{dfn}
For $q \geq 0$ and a $Q$-admissible $I \in (\XQ)^{q+1}_Q$, we define a map
%For $q \geq 0$ and a $Q$-admissible $I \in (\XQ)^{q+1}_Q$, we define an integral oparator
\begin{align}\label{eqn:map intQI}
%\int_{Q, I}: \Gamma (V_I, \msF_0) \ra \C; \ f\omega \mapsto \int_{\sigma_{(\xi_1, \dots, \xi_g)}} f\omega.
\int_{Q, I}: \Gamma (V_I, \msF_0) \ra \C;\ s \mapsto \int_{Q, I} s 
\end{align}
as follows. 
%First, take $\xi_1, \dots, \xi_g \in V_I$ such that $\tp{Q}\xi_i=\theta^{(i)} \xi_i$ for $i=1, \dots, g$. Then we define 
Take $\xi_1, \dots, \xi_g \in V_I$ such that $\tp{Q}\xi_i=\theta^{(i)} \xi_i$ for $i=1, \dots, g$, and define 
\[
%%\int_Q: \Gamma (V_I, \msF_0) \ra \C; \ f\omega \mapsto \int_{\sigma_{(\xi_1, \dots, \xi_g)}} f\omega.
\int_{Q, I} f\omega := \int_{\sigma_{(\xi_1, \dots, \xi_g)}} f\omega 
\]
for $f \omega \in \Gamma (V_I, \msF_0)$. 
%Cf. Section \ref{subsect:Hur} Equation \eqref{eqn:int1}.
Note that by Lemma \ref{lem:int indep}, the map $\int_{Q, I}$ is independent of the choice of the eigenvectors $\xi_1, \dots, \xi_g$. 
\end{dfn}

\begin{rmk}\label{rmk:intQ}
Strictly speaking, the map $\int_Q$ is depending on the (fixed) choice of the order of the eigenvalues $\theta^{(1)}, \dots, \theta^{(g)}$ up to sign. %(We have fixed such an ordering at the beginning of this subsection. )
%In this subsection we fix the order of the eigenvalues $\theta^{(1)}, \dots, \theta^{(g)}$.
\end{rmk}

\begin{ex}\label{ex:Hur formula}
Let the notations be the same as in Section \ref{subsect:number field}. Furthermore, let $\theta \in F^{\times}$ and $Q=\rho_w(\theta) \in \Xi$ be as in Lemma \ref{lem:rev}, and let $I \in (\XQ)^{g}_Q$. 
%
%Then for $k \geq 0$ and $x \in C_I^Q \setm \{0\}$, we have
%
\begin{enumerate}
\item For $k \geq 0$ and $x \in C_I^Q \setm \{0\}$, we have
\begin{align*}
N_{w^*}(y)^k \frac{\omega(y)}{\brk{x,y}^{g+kg}} \in \Gamma(V_I, \msF_0), 
\end{align*}
and
\begin{align*}
\int_{Q, I} N_{w^*}(y)^k \frac{\omega(y)}{\brk{x,y}^{g+kg}}
= \frac{(k!)^g}{(g+kg-1)!} \frac{\det (w^{(1)}, \dots, w^{(g)})}{N_w(x)^{k+1}}. 
\end{align*}
%In particular, for $k \geq 1$, we find
\item For $k \geq 1$, we have
\begin{align*}
N_{w^*}(y)^k \psi_{kg, I}^Q(y)\omega(y) \in \Gamma(V_I, \msF_0), 
\end{align*}
and
\begin{align*}
\int_{Q, I} N_{w^*}(y)^k \psi_{kg, I}^Q(y)\omega(y)
= 
\frac{(k!)^g \det (w^{(1)}, \dots, w^{(g)})}{(g+kg-1)!}
%\sgn (I)\sum_{\substack{x \in C_I^Q \cap \Z^g\\ x \neq 0}} \frac{1}{N_w(x)^{k+1}}. 
\sgn (I)\sum_{\substack{x \in C_I^Q \cap \Z^g \setm \{0\}}} \frac{1}{N_w(x)^{k+1}}. 
\end{align*}
\end{enumerate}
\begin{proof}
(1) First, since $x \in C_I^Q \setm \{0\}$, we easily see that $\re (\brk{x,y})>0$ for all $y \in V_I$, i.e., $V_I \subset V_x$. In particular, we have $\brk{x,y}\neq 0$ for all $y \in \pi_{\C}^{-1}(\pi_{\C}(V_I))$, and hence we obtain the first assertion. 
Now, by Lemma \ref{lem:rev} (5), we know that $w^{(1)}, \dots, w^{(g)} \in \C^g$ are the eigenvectors of $\tp{Q}$ with eigenvalues $\theta^{(1)}:=\tau_1(\theta), \dots, \theta^{(g)}:=\tau_g(\theta) \in \C$ respectively. %, and that $w^{*(1)}, \dots, w^{*(g)}$ are the dual basis of 
Take $\mu_1, \dots, \mu_g \in \C^{\times}$ so that $\xi_1:= \mu_1w^{(1)}, \dots, \xi_g:=\mu_g w^{(g)} \in V_I$. This is possible because $I$ is $Q$-admissible. 
Then by Lemma \ref{lem:rev} (3), we see that $\xi_1^*:=\mu_1^{-1}w^{*(1)}, \dots, \xi_g^*:=\mu_g^{-1}w^{*(g)}$ form the dual basis of $\xi_1, \dots, \xi_g$. 
Thus by Proposition \ref{prop:Hur formula}, we find
\begin{align*}
\int_{Q, I} N_{w^*}(y)^k \frac{\omega(y)}{\brk{x,y}^{g+kg}} 
&=
\int_{\sigma(\xi_1,\dots, \xi_g)}  \prod_{i=1}^g\brk{\mu_i \xi_i^*,y}^k \frac{\omega(y)}{\brk{x,y}^{g+kg}} \\
&=
(\mu_1 \dots \mu_g)^k 
\frac{(k!)^g}{(g+kg-1)!} 
\frac{\det (\xi_1, \dots, \xi_g)}{\prod_{i=1}^g \brk{x, \xi_i}^{k+1}} \\
&=
(\mu_1 \dots \mu_g)^k 
\frac{(k!)^g}{(g+kg-1)!} 
\frac{\det (\mu_1w^{(1)}, \dots, \mu_g w^{(g)})}{\prod_{i=1}^g \brk{x, \mu_iw^{(i)}}^{k+1}} \\
&=\frac{(k!)^g}{(g+kg-1)!} \frac{\det (w^{(1)}, \dots, w^{(g)})}{N_w(x)^{k+1}}.  %\qedhere
\end{align*}

(2) The first assertion follows from Proposition \ref{prop:conv}. The integral formula follows from (1) by taking the sum over $x \in C_I^Q \cap \Z^g \setm \{0\}$. %$x \in C_I^Q \setm \{0\}$. 
\end{proof}
\end{ex}

Next, we extend the map \eqref{eqn:map intQI} to the cohomology group. 

\begin{lem}\label{lem:Q-int} Let  $I =(\alpha_{0}, \dots, \alpha_{q}) \in (\XQ)^{q+1}$.
\begin{enumerate}
%\item If $q\geq 1$ and $I$ is $Q$-admissible, then $I^{(i)}=(\alpha_{0}, \dots, \check{\alpha_{i}}, \dots, \alpha_{q})$ is also $Q$-admissible for $i=0, \dots, q$.
\item If $q\geq 1$ and $I$ is $Q$-admissible, then so is $I^{(i)}=(\alpha_{0}, \dots, \check{\alpha_{i}}, \dots, \alpha_{q})$ for $i=0, \dots, q$.
\item Let $\gamma \in \Gamma_Q$. If $I$ is $Q$-admissible, then so is $\gamma I$, i.e., $(\XQ)^{q+1}_Q$ is a $\Gamma_Q$-stable subset of $(\XQ)^{q+1}$.
\end{enumerate}
\end{lem}
\begin{proof}
(1) follows from the fact $V_I= V_{I^{(i)}} \cap V_{\alpha_{i}} \subset V_{I^{(i)}}$.

(2) Take $\xi_1, \dots, \xi_g \in V_I$ such that $\tp{Q}\xi_i=\theta^{(i)} \xi_i$ for $i=1, \dots, g$. 
%Then since $\tp{Q}\tp{\gamma}=\tp{\gamma}\tp{Q}$, we see that $\xi_1, \dots, \xi_g$ are also the eigenvectors of $\tp{\gamma}$. Let $\lambda_i \in \C$ be the eigenvalue of $\tp{\gamma}$ associated with $\xi_i$. Then by  Lemma \ref{lem stable cov}, we have
Then since $\tp{Q}\tp{\gamma}=\tp{\gamma}\tp{Q}$, we see that $\tp{\gamma}^{-1}\xi_1, \dots, \tp{\gamma}^{-1}\xi_g$ are again eigenvectors of $\tp{Q}$ with eigenvalues $\theta^{(1)}, \dots, \theta^{(g)}$ respectively. On the other hand, by Lemma \ref{lem stable cov}, we have
\[
%\lambda_i^{-1}\xi_i =\tp{\gamma}^{-1} \xi_i \in \tp{\gamma}^{-1} V_I =V_{\gamma I}
\tp{\gamma}^{-1} \xi_i \in \tp{\gamma}^{-1} V_I =V_{\gamma I}
\]
%for $i=1, \dots, g$. This shows that $\lambda_1^{-1}\xi_1 , \dots, \lambda_g^{-1}\xi_g$ form a system of eigenvectors of $\tp{Q}$ in $V_{\gamma I}$. 
for $i=1, \dots, g$. Thus we find that $\tp{\gamma}^{-1}\xi_1 , \dots, \tp{\gamma}^{-1}\xi_g$ are a system of eigenvectors of $\tp{Q}$ in $V_{\gamma I}$. 
\end{proof}

For a $\Gamma_Q$-equivariant sheaf $\msF$ on $\Yo$, set
\begin{align*}
%C^q_Q(\mcX_{\Q}, \msF_0) &:=\prod_{I \in (\XQ)^{q+1}_Q} \Gamma(V_I, \msF_0), \\
%C^q_Q(\mcX_{\Q}, \C) &:=\prod_{I \in (\XQ)^{q+1}_Q}\C.
C^q_Q(\mcX_{\Q}, \msF) &:=\prod_{I \in (\XQ)^{q+1}_Q} \Gamma(V_I, \msF), \\
%{}_QC^q(\mcX_{\Q}, \msF) &:=\prod_{I \notin (\XQ)^{q+1}_Q} \Gamma(V_I, \msF). 
{}_QC^q(\mcX_{\Q}, \msF) &:=\prod_{\substack{I \in (\XQ)^{q+1} \\ I \notin (\XQ)^{q+1}_Q}} \Gamma(V_I, \msF). 
%{}_QC^q(\mcX_{\Q}, \msF) &:=\prod_{I \in (\XQ)^{q+1} \setm (\XQ)^{q+1}_Q} \Gamma(V_I, \msF). 
\end{align*}
%so that we have a natural short exact sequence
Then we have a natural short exact sequence
\begin{align}\label{eqn:Q-int es}
%0\ra {}_QC^q(\mcX_{\Q}, \msF) \ra C^q(\mcX_{\Q}, \msF) \stackrel{{p}_Q}{\ra} C^q_Q(\mcX_{\Q}, \msF) \ra 0.
0\ra {}_QC^q(\mcX_{\Q}, \msF) \ra C^q(\mcX_{\Q}, \msF) \stackrel{{p}_Q}{\ra} C^q_Q(\mcX_{\Q}, \msF) \ra 0, 
\end{align}
where $p_Q$ is the natural projection. 
By Lemma \ref{lem:Q-int}, we easily see that ${}_QC^{\bullet}(\mcX_{\Q}, \msF)$ becomes a $\Gamma_Q$-equivariant subcomplex of $C^{\bullet}(\mcX_{\Q}, \msF)$, and hence $C^{\bullet}_Q(\mcX_{\Q}, \msF)$ has a natural structure of $\Gamma_Q$-equivariant complex induced from that of $C^{\bullet}(\mcX_{\Q}, \msF)$.
%
%replace20201009
%We define
%\[
%H_Q^q(\Yo, \Gamma_Q, \msF):=H^q \big(C^{\bullet}_Q(\mcX_{\Q}, \msF)^{\Gamma_Q}\big)
%\]
%to be the $q$-th cohomology group of the complex $C^{\bullet}_Q(\mcX_{\Q}, \msF)^{\Gamma_Q}$.
%Note that we have a natural projection obtained from \eqref{eqn:Q-int es}: 
%\[
%{p}_Q: H^q(\Yo, \Gamma_Q, \msF) \ra H_Q^q(\Yo, \Gamma_Q, \msF).
%\]
%replacement canceled%-->replace20201009
For a subgroup $\Gamma \subset \Gamma_Q$, we define
\[
H_Q^q(\Yo, \Gamma, \msF):=H^q \big(C^{\bullet}_Q(\mcX_{\Q}, \msF)^{\Gamma}\big)
\]
to be the $q$-th cohomology group of the complex $C^{\bullet}_Q(\mcX_{\Q}, \msF)^{\Gamma}$.
%

%Now, by taking the product of $\int_{Q, I}$ \eqref{eqn:map intQI} over $I \in (\XQ)^{q+1}_Q$, we define
Now, by taking the product of \eqref{eqn:map intQI} over $I \in (\XQ)^{q+1}_Q$, we define
\[
%\int_Q: C^q_Q(\mcX_{\Q}, \msF_0) \ra C^q_Q(\mcX_{\Q}, \shfC).
\int_Q: C^q_Q(\mcX_{\Q}, \msF_0) \ra C^q_Q(\mcX_{\Q}, \shfC);\ (s_I)_{I \in (\XQ)^{q+1}_Q} \mapsto \left( \int_{Q, I} s_I \right)_{I \in (\XQ)^{q+1}_Q}.
\]
%Here $\shfC$ is the constant sheaf associated to $\C$ with the trivial $\Gamma_Q$-equivariant structure. 
Here $\shfC$ is regarded as a constant sheaf associated to $\C$ with the trivial $\Gamma_Q$-equivariant structure.

\begin{prop}\label{prop:intQ}
The map
\[
\int_Q: C^{\bullet}_Q(\mcX_{\Q}, \msF_0) \ra C^{\bullet}_Q(\mcX_{\Q}, \C)
\]
is a morphism of $\Gamma_Q$-equivariant complexes, and hence induces a map
\[
\int_Q: H^q_Q(\Yo, \Gamma_Q, \msF_0) \ra H^q_Q(\Yo, \Gamma_Q, \C)
\]
for $q \geq 0$.
\end{prop}
\begin{proof}
First we must show $\int_Q \circ \dd^q=\dd^q \circ \int_Q$ for $q \geq 0$.
Let $J=(\alpha_{0}, \dots, \alpha_{q+1}) \in (\XQ)^{q+2}_Q$, and let $\xi_1, \dots, \xi_g \in V_J$ be a system of eigenvectors of $\tp{Q}$ with eigenvalues $\theta^{(1)}, \dots, \theta^{(g)}$ respectively.
Then for $s=(s_I)_{I\in (\XQ)^{q+1}_Q} \in C^q_Q(\mcX_{\Q}, \msF_0)$, we have
\begin{align*}
\left( \int_Q d^q(s) \right)_J &= \int_{Q, J} (d^q(s))_J \\
&=
\int_{\sigma_{(\xi_1, \dots, \xi_g)}} \sum_{i=0}^{q+1} (-1)^i s_{J^{(i)}}|_{V_J} \\
&=
\sum_{i=0}^{q+1} (-1)^i  \int_{\sigma_{(\xi_1, \dots, \xi_g)}} s_{J^{(i)}} \\
&=
\sum_{i=0}^{q+1} (-1)^i  \left(\int_Q s\right)_{J^{(i)}}
=
\left(d^q \left(\int_Q s\right)\right)_J,
\end{align*}
where $J^{(i)}=(\alpha_{0}, \dots, \check{\alpha_{i}}, \dots, \alpha_{q+1})$.

Next we must show $\int_Q \circ [\gamma]=[\gamma] \circ \int_Q$ for $\gamma \in \Gamma_Q$. Let $J=(\alpha_{1}, \dots, \alpha_{q+1}) \in (\XQ)^{q+1}_Q$, and let again $\xi_1, \dots, \xi_g \in V_J$ be a system of eigenvectors of $\tp{Q}$ with eigenvalues $\theta^{(1)}, \dots, \theta^{(g)}$ respectively. 
%Moreover, as in the proof of Lemma \ref{lem:Q-int}, let $\lambda_1, \dots, \lambda_g \in \C$ be the eigenvalues of $\tp{\gamma}$ associated with $\xi_1, \dots, \xi_g$ respectively.
%Then for $s=(s_I)_{I\in (\XQ)^{q+1}_Q} \in C^q_Q(\mcX_{\Q}, \msF_0)$, we have
Then as in the proof of Lemma \ref{lem:Q-int}, we see that $\tp{\gamma}\xi_1, \dots, \tp{\gamma}\xi_g$ are eigenvectors of $\tp{Q}$ in $V_{\gamma^{-1}J}$ with eigenvalues $\theta^{(1)}, \dots, \theta^{(g)}$ respectively. 
Therefore, for $s=(s_I)_{I\in (\XQ)^{q+1}_Q} \in C^q_Q(\mcX_{\Q}, \msF_0)$, we have
\begin{align*}
\left( \int_Q [\gamma](s) \right)_J
&=
\int_{Q, J} ([\gamma](s))_J \\
&=
\int_{\sigma_{(\xi_1, \dots, \xi_g)}} s_{\gamma^{-1}J}(\tp{\gamma}y) \\
&=
\int_{\sigma_{(\tp{\gamma}\xi_1, \dots, \tp{\gamma}\xi_g)}} s_{\gamma^{-1}J}(y) \\
%&=
%\int_{\sigma_{(\lambda_1\xi_1, \dots, \lambda_g\xi_g)}} s_{\gamma^{-1}J}(y)\\
&=
\int_{Q, \gamma^{-1}J} s_{\gamma^{-1}J}
=\left(\int_Q s\right)_{\gamma^{-1}J}
=\left([\gamma] \left(\int_Q s \right)\right)_J.
\end{align*}
This completes the proof.
\end{proof}

Let $\int_Q$ also denote the composition %$\int_Q \circ {p}_Q$: 
\begin{align}\label{map:intQ}
\int_Q: H^q(\Yo, \Gamma_Q, \msF_0) \stackrel{{p}_Q}{\ra} H_Q^q(\Yo, \Gamma_Q, \msF_0) \stackrel{\int_Q}{\ra} H_Q^q(\Yo, \Gamma_Q, \shfC), 
\end{align}
where $p_Q$ is the natural map induced from the projection $p_Q$ in \eqref{eqn:Q-int es}. See also Corollary \ref{cor:equiv coh}.

%%%%%%%%%%%%%%%%%%%%%%%%%%%%%%%%%%%%%%%%%%%%%%%%%%%%%%%%%%%%%%%%%%%%%%%%%%%%%%%%%%%%%%%%%%%%%%%%%%%%%%%%%%%%%%%%%%%%%%%%%%%%%%%%%%%%%%%%

\section{Specialization to the zeta values}\label{sect:spec}
In this section we compute the group $H^q_Q(\Yo, \Gamma_Q, \shfC)$ explicitly, and show that we can get the values of the zeta function as a specialization of the Shintani-Barnes cocycle $[\Psi_d]$.

%First we recall the notations in Section \ref{subsect:number field}. 
First we return to the setting in Section \ref{subsect:number field}. 
%
%Let $F$ be a number field of degree $g$, 
Let 
\begin{itemize}
%\item $F/\Q$: a number field of degree $g$,
\item $F/\Q$ be a number field of degree $g$,
%\item $\tau_1, \dots, \tau_g: F \hookrightarrow \C$: the field embeddings of $F$ into $\C$, 
\item $\tau_1, \dots, \tau_g: F \hookrightarrow \C$ be the field embeddings of $F$ into $\C$, 
%\item $\mcO \subset F$: an order, 
\item $\mcO \subset F$ be an order, 
%\item $\mfa \subset F$: a proper fractional $\mcO$-ideal,
\item $\mfa \subset F$ be a proper fractional $\mcO$-ideal,
\item $w_1, \dots, w_g \in \mfa$ be a basis of $\mfa$ over $\Z$, 
\item $w:=\tp{(w_1, \dots, w_g)} \in F^g$, and $w^{(i)}:=\tau_i(w)=\tp{(\tau_1(w_1), \dots, \tau_g(w_g))} \in \C^g$, 
\item $\rho_w: F \ra M_g(\Q)$ be the regular representation with respect to 
\begin{align*}%\label{isom w}
w: \Q^g \isomto F; x \mapsto \brk{x,w},
\end{align*}
\item $N_w(x_1, \dots, x_g) \in \Q[x_1, \dots, x_g]$ be the norm polynomial with respect to $w$, 
%\item $\rho_w: F \ra M_g(\Q)$ be the regular representation with respect to $w$, 
\item $w_1^*, \dots, w_g^* \in F$ be the dual basis of $w_1, \dots, w_g$ with respect to the trace $Tr_{F/\Q}$, 
\item $w^*$, $w^{*(i)}$, $N_{w^*}$, $\rho_{w^*}$ be the dual objects obtained from $w_1^*, \dots, w_g^*$. 
%\item Take $\theta \in F^{\times}$ such that $F=\Q(\theta)$ and put $Q:=\rho_w(\theta) \in \Xi$, cf. Lemma \ref{lem:rev}. %In the following we fix these notations.  
\end{itemize}
%
%We assume that $\tau_1, \dots, \tau_{r_1}: F \hookrightarrow \R$ are the real embeddings, and $\tau_{r_1+1}=\overline{\tau_{r_1+r_2+1}}, \dots, \tau_{r_1+r_2}=\overline{\tau_{r_1+2r_2}}: F \hookrightarrow \C$ are the non-real embeddings. 
%
%

%Take $\theta \in F^{\times}$ such that $F=\Q(\theta)$ and put $Q:=\rho_w(\theta) \in \Xi$, cf. Lemma \ref{lem:rev}. In the following we fix these notations.  
Take $\theta \in F^{\times}$ such that $F=\Q(\theta)$ and put $Q:=\rho_w(\theta) \in \Xi$. 
%Note that by Lemma \ref{lem:rev} (1), we may assume that $Q$ is of this form without loss of generality. 
Moreover, we set $\theta^{(1)}:=\tau_1(\theta), \dots, \theta^{(g)}:=\tau_g(\theta) \in \C^{\times}$ to be the eigenvalues of $\tp{Q}$. 
We fix these notations.

%%%%%%%%%%%%%%%%%%%%%%%%%%%%%%%%%%%%%%%%%%%%%%%%%%%%%%%%%%%%%%%%%%%%%%%%%%%%%%%%%%%%%

\subsection{Computation of $H^q_Q(\Yo, \Gamma_Q, \C)$}\label{subsect:comp HQ}
Define 
\[
T_w:=\{x \in \R^g \mid N_w(x) \neq 0\} \subset \R^g \setm \{0\}
\]
to be the set of real vectors whose norm is non-zero.
%By Lemma \ref{lem:rev} we see that $\Gamma_Q$ is acting on $T_w$.
By Lemma \ref{lem:rev} (7), it is clear that $T_w$ is a $\Gamma_Q$-stable subset of $\R^g \setm \{0\}$. 
Note that under the isomorphism
\begin{align*}%\label{eqn:isom w1}
w: \R^g \isomto F_{\R}:=F\otimes_{\Q}\R; x \mapsto \brk{x,w}, 
\end{align*}
$T_w$ corresponds to $F_{\R}^{\times}=\{\alpha \in F_{\R} \mid N_{F/\Q}(\alpha)\neq 0\}$, i.e., 
\begin{align}\label{eqn:isom w2}
w: T_w \isomto F_{\R}^{\times}.
\end{align}
The aim of this subsection is to obtain the following isomorphism:
\begin{align}\label{eqn:isom HQ}
%H^q_Q(\Yo, \Gamma_Q, \shfC) \isomto H^q(T_w/\Gamma_Q, \C) \simeq H^q(F_{\R}^{\times}/ \mcO^1, \C),
H^q_Q(\Yo, \Gamma_Q, \shfC) \isomfrom H^q(T_w/\Gamma_Q, \C) \simeq H^q(F_{\R}^{\times}/ \mcO^1, \C),
\end{align}
where the last two cohomology groups are the usual singular cohomology groups.

As in Section \ref{sect:int}, for $I=(\alpha_{1}, \dots, \alpha_{q+1}) \in (\XQ)^{q+1}$, let 
\begin{align*}
\sigma_I: \Delta^q \ra \R^g; \ t=(t_1, \dots, t_{q+1})\mapsto \sum_{i=1}^{q+1} \alpha_i t_i 
\end{align*}
denote the affine $q$-simplex with vertices $\alpha_1, \dots, \alpha_{q+1}$, and let $|\sigma_I|:=\sigma_I (\Delta^q) \subset \R^g$ denote the image of $\sigma_I$.  
The following lemma enables us to compute the group $H^q_Q(\Yo, \Gamma_Q, \shfC)$ using these simplices.

\begin{lem}\label{lem:Q-int simpl}
Let $q\geq 0$, $I=(\alpha_{1}, \dots, \alpha_{q+1}) \in (\XQ)^{q+1}$. The following conditions are equivalent:
\begin{enumerate}
\item[\textrm{(i)}] $I$ is $Q$-admissible.
\item[\textrm{(ii)}] $|\sigma_I| \subset T_w$.
\end{enumerate}
\end{lem}

To prove this lemma, recall the following fact:
\begin{lem}\label{lem:comp 1}
Let $A \subset \C$ be a convex compact subset. The following conditions are equivalent:
\begin{enumerate}
\item[\textrm{(i)}] $0 \notin A$.
\item[\textrm{(ii)}] There exists $\lambda \in \C^{\times}$ such that $\re (\lambda A) \subset \Rpos$.
\end{enumerate}
\end{lem}
\begin{proof}
This follows from \cite[Theorem 3.4 (b)]{R91:Fun}.
\end{proof}

\begin{proof}[Proof of Lemma \ref{lem:Q-int simpl}]
First, by Lemma \ref{lem:rev} (5), we know that $w^{(1)}, \dots, w^{(g)}$ are the eigenvectors of $\tp{Q}$ with eigenvalues $\theta^{(1)}, \dots, \theta^{(g)}$ respectively. 
Therefore, $I$ is $Q$-admissible if and only if %for all $j=1, \dots, g$
\begin{align*}
&\forall  j\in \{1, \dots, g\}, \exists \lambda_j \in \C^{\times}, \lambda_j w^{(j)} \in V_I \\
\Lra&
\forall  j\in \{1, \dots, g\}, \exists \lambda_j \in \C^{\times}, \forall i\in \{1, \dots, q+1\}, \re (\brk{\alpha_i, \lambda_j w^{(j)}})>0 \\
\Lra&
\forall  j\in \{1, \dots, g\}, \exists \lambda_j \in \C^{\times}, \re (\lambda_j \brk{|\sigma_I|, w^{(j)}}) \subset \Rpos \\
\stackrel{\ast}{\Lra}&
\forall  j\in \{1, \dots, g\}, 0 \notin \brk{|\sigma_I|, w^{(j)}} \\
\Lra&
%\forall x \in \sigma_I(\Delta^q), N_w(x)=\prod_{i=1}^g \brk{x, w^{(i)}} \neq 0\\
\forall x \in |\sigma_I|, N_w(x) \neq 0\\
\Lra&
|\sigma_I| \subset T_w
\end{align*}
Note that the third equivalence $\stackrel{\ast}{\Lra}$ follows from Lemma \ref{lem:comp 1} since $\brk{|\sigma_I|, w^{(i)}} \subset \C$ is a convex compact subset.
This proves the lemma.
\end{proof}

%Lemma \ref{lem:Q-int simpl} enables us to compute $H^q_Q(\Yo, \Gamma_Q, \shfC)$ using the singular (co)homology. 
%For $q\geq 0$, let 
For $q\geq 0$, let $\Sigma_q:=\big\{\sigma: \Delta^q \ra T_w \ \big| \  \text{continuous} \big\}$ denote the set of singular $q$-simplices in $T_w$, and let
\begin{align*}
%S_q:=\Z \left[\sigma: \Delta^q \ra T_w \ \big| \  continuous \right] 
S_q:=\Z \left[ \Sigma_q \right] 
\end{align*}
denote the group of singular $q$-chains of $T_w$. 
For $j=1, \dots, q+1$, let 
\[
\delta^q_j: \Delta^{q-1} \ra \Delta^q; (t_1, \dots, t_q) \mapsto (t_1, \dots, t_{j-1}, 0, t_{j}, \dots, t_q) 
\]
denote the $j$-th face map. 
Then we have a boundary map $\partial: S_q \ra S_{q-1}$ which maps $\sigma \in \Sigma_q$ to
\[
\partial \sigma = \sum_{j=1}^{q+1} (-1)^{j-1} \sigma \circ \delta_j^q \in S_{q-1}. 
\]
%where $\delta_j^q$ is the $j$-th face map
%\[
%\delta^q_j: \Delta^{q-1} \ra \Delta^q; (t_1, \dots, t_q) \mapsto (t_1, \dots, t_{j-1}, 0, t_{j}, \dots, t_q). 
%\]
%
The action of $\Gamma_Q$ on $T_w$ naturally induces an action of $\Gamma_Q$ on $S_q$ and we have a $\Gamma_Q$-equivariant singular chain complex $S_{\bullet}$. 
%The action of $\Gamma_Q$ on $T_w$ naturally induces a $\Gamma_Q$-equivariant structure on the singular chain complex $S_{\bullet}$. 
Moreover, let 
%Now, let 
\begin{align*}
K_{q}:=\Z \left[(\XQ)^{q+1}_Q\right] 
\end{align*}
denote the free abelian group generated by $(\XQ)^{q+1}_Q$. 
By Lemma \ref{lem:Q-int} (2), we have a natural action of $\Gamma_Q$ on $K_q$. 
%
%Now, by Lemma \ref{lem:Q-int simpl}, we have a natural injective homomorphism
Then by Lemma \ref{lem:Q-int simpl}, we have a natural injective homomorphism
\begin{align*}
K_{q} \hookrightarrow S_{q}; I \mapsto \sigma_I, 
\end{align*}
which is clearly a $\Gamma_Q$-equivariant map. 
%Furthermore, by Lemma \ref{lem:Q-int}, we see that this map induces a injective map $K_{\bullet} \hookrightarrow S_{\bullet}$ of $\Gamma_Q$-equivariant complexes. 
In the following, we identify $K_{q}$ with a $\Gamma_Q$-submodule 
\begin{align*}
%\Z \left[\sigma_I \ \Big| \  I \in (\XQ)^{q+1}_Q\right] \subset S_q 
\Z \left[\sigma_I \ \Big| \  I \in (\XQ)^{q+1}_Q\right] 
=
\Z \left[\sigma_I \ \Big| \  I \in (\XQ)^{q+1}, |\sigma_I| \subset T_w \right] \subset S_q 
\end{align*}
%of $S_{q}$ via this map. 
of $S_{q}$ via this injective map. 
%
%By Lemma \ref{lem:Q-int} (1), we see that the boundary map $\partial$ maps $K_q$ to $K_{q-1}$, and hence $K_{\bullet}$ becomes a $\Gamma_Q$-equivariant subcomplex of $S_{\bullet}$. 
%Then by Lemma \ref{lem:Q-int} (1), we see that the boundary map $\partial$ maps $K_q$ to $K_{q-1}$, and hence we obtain a $\Gamma_Q$-equivariant subcomplex $K_{\bullet} \subset S_{\bullet}$. 
Then by Lemma \ref{lem:Q-int} (1), we see that the boundary map $\partial$ maps $K_q$ to $K_{q-1}$, and hence $K_{\bullet} \subset S_{\bullet}$ becomes a $\Gamma_Q$-equivariant subcomplex of $S_{\bullet}$. 

Note that we have a natural isomorphism
\begin{align*}
K_{\C}^{\bullet}:= \Hom_{\Z}(K_{\bullet}, \C) \simeq \prod_{I \in (\XQ)^{\bullet +1}_Q}\C =  C^{\bullet}_Q (\mcX_{\Q}, \shfC) 
\end{align*}
of $\Gamma_Q$-equivariant complexes, 
and hence 
\begin{align*}
H^q_Q(\Yo, \Gamma_Q, \shfC) \simeq H^q \big((K_{\C}^{\bullet})^{\Gamma_Q} \big).
\end{align*}
Therefore, in order to obtain \eqref{eqn:isom HQ}, we compare $K_{\bullet}$ and $S_{\bullet}$.

\begin{prop}\label{prop:qis1}
\begin{enumerate}
%\item For any subgroup $\Gamma \subset \Gamma_Q$ and $q \geq 0$, the quotient group $S_{q}/ K_{q}$ is an induced $\Gamma$-module.
%\item For $q \geq 0$, the quotient group $S_{q}/ K_{q}$ is an induced $\Gamma_Q$-module.
\item Let $\Gamma \subset \Gamma_Q$ be a subgroup. For $q \geq 0$, the quotient group $S_{q}/ K_{q}$ is an induced $\Gamma$-module.
\item The inclusion map
\[
K_{\bullet} \hookrightarrow S_{\bullet}
\]
is a quasi-isomorphism. In other words, the quotient complex $S_{\bullet}/ K_{\bullet}$ is exact.
%\item For $q \geq 0$ the quotient $S_{q}/ S_{\Q, q}$ is an induced $\Gamma_Q$-module.
\end{enumerate}
\end{prop}

\begin{proof}
(1) is clear since we have
\[
%S_{q}/ K_{q} \simeq \Z \left[\sigma: \Delta^q \ra T_w \ \Big|\  \sigma \neq \sigma_I, \forall  I \in (\XQ)^{q+1}_Q \right],
S_{q}/ K_{q} \simeq \Z \left[\sigma \in \Sigma_q \ \Big|\  \sigma \notin K_q \right] 
\]
%and $\Gamma_Q$ acts freely on the basis $\big\{\sigma: \Delta^q \ra T_w \ \big|\  \sigma \neq \sigma_I, \forall  I \in (\XQ)^{q+1}_Q\big\}$.
and $\Gamma \subset \Gamma_Q$ acts freely on the basis $\big\{\sigma \in \Sigma_q \ \big|\  \sigma \notin K_q \big\}$.

(2) 
%The author believes that this kind of fact is well-known to experts, but here we give a proof for the sake of completeness of the paper. 
This kind of fact may be well-known to experts, but here we give a proof for the sake of completeness of the paper. 
First take any finite open covering
\[
T_w=\bigcup_{k=1}^NU_k
\]
of $T_w$ such that $U_k$ is a convex open subset of $T_w$ for all $k$. 
The existence of such a covering can be easily seen from the identification $w: T_w \isomto F_{\R}^{\times}$.

%
%We will prove the exactness of the quotient complex $S_{\bullet}/K_{\bullet}$. 
We will prove that the quotient complex $S_{\bullet}/K_{\bullet}$ is exact. 
Let $q \geq 0$ and let ${a} \in S_q$ such that $\partial {a} \in K_{q-1}$. 
%We assume $K_{-1}=S_{-1}=0$. 
We need to show the following:
\begin{aim}
%There exist ${b}_1 \in S_{q+1}$ and ${b}_2 \in K_{q}$ such that ${a}=\partial {b}_1+{b}_2$.
There exist $\eta \in S_{q+1}$ and ${b} \in K_{q}$ such that ${a}=\partial \eta+{b}$.
\end{aim}
Suppose ${a} \in S_q$ is of the form
\[
{a}=\sum_{i=1}^r c_i\sigma_i,
\]
%where $\sigma_i$ are distinct singular $q$-simplices in $T_w$.
where $\sigma_i$ are distinct singular $q$-simplices in $T_w$, and $c_i \in \Z$.
By using the barycentric subdivision if necessary, without loss of generality we may assume
%
%\[
%\forall i \in \{1, \dots, r\}, \exists \kappa_i \in \{1, \dots, N\}, \sigma_i(\Delta^q) \subset U_{\kappa_i}. 
%\]
%
\begin{align}\label{eqn:small chain complex}
\forall i \in \{1, \dots, r\}, \exists \kappa_i \in \{1, \dots, N\}, \sigma_i(\Delta^q) \subset U_{\kappa_i}. 
\end{align}
%where $|\sigma_i| =\sigma_i(\Delta^q)$ is the image of $\sigma_i$. 
%28
Indeed, let
\begin{align*}
&S: S_n \ra S_n, \\
&T: S_n \ra S_{n+1}
%S: S_n \ra S_n, T: S_n \ra S_{n+1}
\end{align*}
be the subdivision operator and the chain homotopy between $S$ and $\id_{S_n}$ defined as in \cite[Section 2.1, Proof of Proposition 2.21]{H02:Alg}. Then taking into account the fact that the barycenter of any $\sigma_I \in K_n$ ($I \in (\XQ)^{n+1}_Q$) belongs to $\Q^g \cap |\sigma_I|$, we easily see that $S$ (resp. $T$) maps $K_n$ to $K_n$ (resp. $K_{n+1}$). 
Hence we have
\begin{align*}
\partial S(a) &= S(\partial a) \in K_{q-1}, \\
a-S(a) &= \partial T(a) + T(\partial a) \in \partial S_{q+1} + K_q. 
\end{align*}
%Therefore, we can replace $a$ with its (iterated) barycentric subdivision $S^m(a)$ ($m$: sufficiently large) without loss of generality. 
%Therefore, we can freely replace $a$ with its barycentric subdivision $S(a)$ until we have \eqref{eqn:small chain complex}. 
Therefore, we can replace $a$ with its (iterated) barycentric subdivision $S^m(a)$ ($m$: sufficiently large) until we have \eqref{eqn:small chain complex}.

We fix such $\kappa_i$ for each $i=1, \dots, r$.

\vspace{2mm}
\noindent
\textbf{Step 1.} In order to ``approximate'' $\sigma_i$ by the elements in $K_{q}$, we first approximate their vertices ``simultaneously''. For $i=1, \dots, r$ and $j=1, \dots, q+1$, let ${v}_{ij} \in U_{\kappa_i} \subset T_w$ denote the $j$-th vertex of $\sigma_i$, i.e.,
\[
{v}_{ij} = \sigma_i(0, \dots, 0,\stackrel{j}{1},0,\dots, 0) \in T_w.
\]
Then for $i=1, \dots, r$ and $j=1, \dots, q+1$, take ${v}'_{ij} \in U_{\kappa_i} \cap \Q^g$ satisfying the following conditions:
%
%\begin{enumerate}
%\item[(V):] If ${v}_{ij}={v}_{mn}$ for some $i,m \in \{1, \dots, r\}$ and $j,n \in \{1, \dots, q+1\}$, then ${v}'_{ij}={v}'_{mn}$.
%\end{enumerate}
%
%
%20210324_cut
%\vspace{2mm}
%\noindent
%\textbf{(V):} If ${v}_{ij}={v}_{mn}$ for some $i,m \in \{1, \dots, r\}$ and $j,n \in \{1, \dots, q+1\}$, then ${v}'_{ij}={v}'_{mn}$.
%

%\vspace{2mm}
%\noindent
%In other words, if the $j$-th vertex of $\sigma_i$ and the $n$-th vertex of $\sigma_m$ are the same, then ${v}'_{ij}$ and ${v}'_{mn}$ are the same as well. 
%
%
%

\begin{enumerate}
\item[\textbf{(V1):}] If ${v}_{ij} \in \Q^g$, then ${v}'_{ij}={v}_{ij}$. 
\item[\textbf{(V2):}] If ${v}_{ij}={v}_{mn}$ for some $i,m \in \{1, \dots, r\}$ and $j,n \in \{1, \dots, q+1\}$, then ${v}'_{ij}={v}'_{mn}$. 
(In other words, if the $j$-th vertex of $\sigma_i$ and the $n$-th vertex of $\sigma_m$ are the same, then ${v}'_{ij}$ and ${v}'_{mn}$ are the same as well.) 
\end{enumerate}
This is possible because $\Q^g$ is dense in $\R^g$. Then for $i=1, \dots, r$, set %define
\begin{align*}
%\sigma_i': \Delta^q \ra T_w; \ t=(t_1, \dots, t_{q+1}) \mapsto \sum_{j=1}^{q+1} {v}'_{ij} t_j, 
%I_i &:=({v}'_{i1}, \dots, {v}'_{i,q+1}) \in (\XQ)^{q+1}, %\\
I_i :=({v}'_{i1}, \dots, {v}'_{i,q+1}) \in (\XQ)^{q+1} %\\
%a' &:= \sum_{i=1}^r c_i \sigma_{I_i}. 
\end{align*}
and %set $
\[
a':= \sum_{i=1}^r c_i \sigma_{I_i}.
\]%$. 
%By Lemma \ref{lem:Q-int simpl}, we have $I_i \in (\XQ)^{q+1}_Q$, and hence 
Since $U_{\kappa_i}$ is convex, we have $\sigma_{I_i} \subset U_{\kappa_i} \subset T_w$, and hence $\sigma_{I_i} \in K_q$. Therefore, we see that 
%\[%$
$a' \in K_{q}$.    %
%\] 

%Note that by the condition (V), we have the following:

Now, recall that for $j=1, \dots, q+1$, 
\[
\delta^q_j: \Delta^{q-1} \ra \Delta^q; (t_1, \dots, t_q) \mapsto (t_1, \dots, t_{j-1}, 0, t_{j}, \dots, t_q)
\]
denotes the $j$-th face map. 
Then by the conditions (V1) and (V2), we have the following:

%\vspace{2mm}
%\noindent
%\textbf{(F):} If $\sigma_i \circ \delta^q_j= \sigma_m \circ \delta^q_n$ for $i,m \in \{1, \dots, r\}$ and $j,n \in \{1, \dots, q+1\}$, then $\sigma_{I_i} \circ \delta^q_j=\sigma_{I_m} \circ \delta^q_n$. 

%\vspace{2mm}
%\noindent
%In other words, if the $j$-th face of $\sigma_i$ and the $n$-th face of $\sigma_m$ are the same, then the $j$-th face of $\sigma_{I_i}$ and the $n$-th face of $\sigma_{I_m}$ are the same as well. 
%

\begin{enumerate}
\item[\textbf{(F1):}] If $\sigma_i \circ \delta^q_j \in K_{q-1}$, then $\sigma_{I_i} \circ \delta^q_j=\sigma_i \circ \delta^q_j$. 
\item[\textbf{(F2):}] If $\sigma_i \circ \delta^q_j= \sigma_m \circ \delta^q_n$ for $i,m \in \{1, \dots, r\}$ and $j,n \in \{1, \dots, q+1\}$, then $\sigma_{I_i} \circ \delta^q_j=\sigma_{I_m} \circ \delta^q_n$. 
(In other words, if the $j$-th face of $\sigma_i$ and the $n$-th face of $\sigma_m$ are the same, then the $j$-th face of $\sigma_{I_i}$ and the $n$-th face of $\sigma_{I_m}$ are the same as well.)  
\end{enumerate}

%%%%Step2 here
\vspace{2mm}
\noindent
\textbf{Step 2.} 
Next we consider the homotopy between $a$ and $a'$. 
%Let $I=[0,1]$ denote the closed interval. 
For $i=1, \dots, r$, let
\[
h_i: [0,1]\times \Delta^q \ra T_w; (u,t) \mapsto u \sigma_i(t) +(1-u) \sigma_{I_i}(t)
\]
be a homotopy between $\sigma_i$ and $\sigma_{I_i}$. Note that since $U_{\kappa_i}$ is convex, we have
\[
h_i([0,1]\times \Delta^q) \subset U_{\kappa_i}.
\]
%We regard $h_i \in S_{q+1}$ in a standard way using the usual decomposition of the prism $[0,1]\times \Delta^q$: %. 
The homotopy $h_i$ defines a $(q+1)$-chain $\eta_i \in S_{q+1}$in a usual way using the standard decomposition of the prism $[0,1]\times \Delta^q$.  %. 
More precisely, for $j=1, \dots, q+1$, put
\begin{align*}
\epsilon^q_j: \Delta^{q+1} \ra [0,1] \times \Delta^q; (t_1, \dots, t_{q+2}) \mapsto
\left(
\sum_{m\geq j+1} t_m, \Big(t_1, \dots, t_{j-1}, t_j+t_{j+1}, t_{j+2}, \dots, t_{q+2}\Big) 
\right).
\end{align*}
Using these maps, the $(q+1)$-chain $\eta_i \in S_{q+1}$ is defined as
\[
\eta_i:=\sum_{j=1}^{q+1} (-1)^{j-1}  h_i\circ \epsilon^q_j. 
%\eta_i:=\sum_{j=1}^{q+1} (-1)^{j-1}  \epsilon_j^*h_i. 
\]
%Set 
%\[
%\eta:=\sum_{i=1}^r c_i \eta_i \in S_{q+1}. 
%\]
Set $\eta:=\sum_{i=1}^r c_i \eta_i \in S_{q+1}$.

\vspace{2mm}
\noindent
\textbf{Step 3.}
Now we examine the assumption $\partial a \in K_{q-1}$. 
First, we have
\[
\partial {a}=\sum_{i=1}^r \sum_{j=1}^{q+1} (-1)^{j-1} c_i  \sigma_i \circ \delta^q_j.
\]
%Then for each $q$-simplex $\sigma: \Delta^q \ra T_w$, put
For each singular $(q-1)$-simplex $\sigma \in \Sigma_{q-1}$, set %$\sigma: \Delta^{q-1} \ra T_w$, set
\[
C_{\sigma} := \sum_{\substack{i=1, \dots, r,\\ j=1, \dots, q+1,\\ \sigma_i \circ \delta^q_j=\sigma}}(-1)^{j-1} c_i \in \Z, 
\]
%Then we have
so that we have
\[
\partial {a}=\sum_{\sigma \in \Sigma_{q-1}} C_{\sigma} \sigma. 
\]
Then by the assumption $\partial a \in K_{q-1}$, we find that $C_{\sigma}=0$ for all $\sigma \notin K_{q-1}$ since the set $\Sigma_{q-1}$ of singular $(q-1)$-simplices is a basis of $S_{q-1}$.

\vspace{2mm}
\noindent
\textbf{Step 4.}
Next we compute the boundary of the homotopy $\eta \in S_{q+1}$. 
By an elementary computation we see 
\[
%\partial \eta_i=\sigma_i -\sigma_{I_i} - \sum_{j=1}^{q+1}\sum_{m=1}^{q}(-1)^{j+m} h_i \circ \delta^{q}_j \circ \epsilon^{q-1}_{m},
\partial \eta_i=\sigma_i -\sigma_{I_i} - \sum_{j=1}^{q+1}\sum_{m=1}^{q}(-1)^{j+m} h_{ij} \circ \epsilon^{q-1}_{m},
\]
where
\[
h_{ij}: [0,1] \times \Delta^{q-1} \ra T_w;\ (u,t) \mapsto u \sigma_i \circ \delta^q_j(t) +(1-u) \sigma_{I_i}\circ \delta^q_j(t).
\]
is a homotopy between $\sigma_i \circ \delta^q_j$ and $\sigma_{I_i} \circ \delta^q_j$, cf. \cite[Section 2.1, Proof of 2.10]{H02:Alg}

%Here, note that by the property (F), the homotopy map $h_{ij}$ depends only on the $(q-1)$-simplex $\sigma_i \circ \delta^q_j$, i.e., if $\sigma_i \circ \delta^q_j= \sigma_m \circ \delta^q_n$ for $i,m \in \{1, \dots, r\}$ and $j,n \in \{1, \dots, q+1\}$, then $h_{ij}=h_{mn}$. 
%Therefore, for a singular $(q-1)$-simplex $\sigma \in \Sigma_{q-1}$ of the form $\sigma = \sigma_i \circ \delta_j^q$ for some $i \in \{1, \dots, r\}$ and $j \in \{1, \dots, q+1\}$, we define $h_{\sigma} := h_{ij}$ to be this homotopy map between $\sigma_i \circ \delta^q_j$ and $\sigma_{I_i} \circ \delta^q_j$, which is well-defined. 

Now, by the properties (F1) and (F2), we see the following: 
\begin{enumerate}
\item[\textbf{(H1):}] If $\sigma_i \circ \delta^q_j \in K_{q-1}$, then $h_{ij}(u,t)=\sigma_i \circ \delta^q_j(t)$ for $(u,t) \in [0,1] \times \Delta^{q-1}$. 
\item[\textbf{(H2):}] If $\sigma_i \circ \delta^q_j= \sigma_m \circ \delta^q_n$ for $i,m \in \{1, \dots, r\}$ and $j,n \in \{1, \dots, q+1\}$, then $h_{ij}=h_{mn}$. 
\end{enumerate}
%
%In particular, by (H2), the homotopy map $h_{ij}$ depends only on the $(q-1)$-simplex $\sigma_i \circ \delta^q_j$. 
%Therefore, for a singular $(q-1)$-simplex $\sigma \in \Sigma_{q-1}$ of the form $\sigma = \sigma_i \circ \delta_j^q$ for some $i \in \{1, \dots, r\}$ and $j \in \{1, \dots, q+1\}$, we define $h_{\sigma} := h_{ij}$ to be this homotopy map between $\sigma_i \circ \delta^q_j$ and $\sigma_{I_i} \circ \delta^q_j$, which is well-defined. 
Then for each singular $(q-1)$-simplex $\sigma \in \Sigma_{q-1}$, we define a map
\[
h_{\sigma}: [0,1] \times \Delta^{q-1} \ra T_w
\]
as follows: If $\sigma$ is of the form $\sigma = \sigma_i \circ \delta_j^q$ for some $i \in \{1, \dots, r\}$ and $j \in \{1, \dots, q+1\}$, we set $h_{\sigma} := h_{ij}$. This is well-defined by the property (H2). If $\sigma$ is not of the form $\sigma_i \circ \delta_j^q$, then simply set $h_{\sigma}(u,t):=\sigma(t)$ for $(u,t) \in [0,1] \times \Delta^{q-1}$.

Then we find
\begin{align*}
\partial \eta &= a -a' - \sum_{i=1}^r\sum_{j=1}^{q+1}\sum_{m=1}^q (-1)^{j+m} c_i h_{ij} \circ \epsilon^{q-1}_m  \\
&=
a -a' - \sum_{i=1}^r\sum_{j=1}^{q+1}\sum_{m=1}^q (-1)^{j+m} c_i h_{\sigma_i \circ \delta^q_j} \circ \epsilon^{q-1}_m \\
%&=
%a -a' - \sum_{\substack{i=1, \dots, r,\\ j=1, \dots, q+1,\\ \sigma_i \circ \delta^q_j=\sigma}} (-1)^{j-1}c_i \sum_{m=1}^q (-1)^{m-1} h_{\sigma} \circ \epsilon^{q-1}_m \\
&=
a-a' -\sum_{\sigma \in \Sigma_{q-1}}  \sum_{m=1}^q (-1)^{m-1} h_{\sigma} \circ \epsilon^{q-1}_m  \sum_{\substack{i=1, \dots, r,\\ j=1, \dots, q+1,\\ \sigma_i \circ \delta^q_j=\sigma}} (-1)^{j-1}c_i \\
&=
a-a' -\sum_{\sigma \in \Sigma_{q-1}} C_{\sigma}  \sum_{m=1}^q (-1)^{m-1} h_{\sigma} \circ \epsilon^{q-1}_m \\%.
&= 
a-a' -\sum_{\sigma \in \Sigma_{q-1}\cap K_{q-1}} C_{\sigma}  \sum_{m=1}^q (-1)^{m-1} h_{\sigma} \circ \epsilon^{q-1}_m. 
\end{align*}
%
%Since we have $C_{\sigma}=0$ for $\sigma \notin K_{q-1}$, we see that
Note that the last equality holds since we have $C_{\sigma}=0$ for $\sigma \notin K_{q-1}$. 
%
%
%Moreover, for $\sigma \in \Sigma_{q-1}\cap K_{q-1}$ and $m=1, \dots, q$, we easily see that $h_{\sigma} \circ \epsilon^{q-1}_m \in K_q$. 
%Moreover, by the property (H1), we easily see that if $\sigma_i \circ \delta^q_j \in K_{q-1}$, then $h_{\sigma_i \circ \delta^q_j} \circ \epsilon^{q-1}_m \in K_q$ for all $m=1, \dots, q$. 
Moreover, by the property (H1), we easily see that if $\sigma = \sigma_i \circ \delta^q_j \in K_{q-1}$, then $h_{\sigma} \circ \epsilon^{q-1}_m \in K_q$ for all $m=1, \dots, q$. 
Therefore, by setting
\begin{align*}
%b:= a' +\sum_{\sigma \in \Sigma_{q-1}} C_{\sigma}  \sum_{m=1}^q (-1)^{m-1} h_{\sigma} \circ \epsilon^{q-1}_m \in K_q. 
%b &:= a' +\sum_{\sigma \in \Sigma_{q-1}} C_{\sigma}  \sum_{m=1}^q (-1)^{m-1} h_{\sigma} \circ \epsilon^{q-1}_m \\
%&= 
b:= a' +\sum_{\sigma \in \Sigma_{q-1}\cap K_{q-1}} C_{\sigma}  \sum_{m=1}^q (-1)^{m-1} h_{\sigma} \circ \epsilon^{q-1}_m
 \in K_q, 
\end{align*}
we obtain the desired identity $a=\partial \eta +b$. 
%This completes the proof. 
\end{proof}

%For $q \geq 0$, let $S^q_{\C}:= \Hom_{\Z}(S_q, \C)$ denote the group of singular $q$-cochains with coefficients in $\C$. 
Let $S^{\bullet}_{\C}:= \Hom_{\Z}(S_{\bullet}, \C)$ denote the singular cochain complex of $T_w$ with coefficients in $\C$. 

\begin{cor}\label{cor:qis}
Let $\Gamma \subset \Gamma_Q$ be a subgroup. 
\begin{enumerate}
\item The map $K_{\bullet} \hookrightarrow S_{\bullet}$ induces a quasi-isomorphism
\[
(K_{\bullet})_{\Gamma} \ra (S_{\bullet})_{\Gamma}, 
\]
where $(-)_{\Gamma}$ denotes the $\Gamma$-coinvariant part. In particular, we obtain an isomorphism
\[
H_q \big( (K_{\bullet})_{\Gamma}\big) \isomto H_q(T_w/\Gamma, \Z).
\]
\item The map $K_{\bullet} \hookrightarrow S_{\bullet}$ induces a quasi-isomorphism
\[
(S^{\bullet}_{\C})^{\Gamma} \ra (K^{\bullet}_{\C})^{\Gamma}.
\]
In particular, we obtain an isomorphism
\[
%H^q(T_w/\Gamma_Q, \C) \isomto H^q \big( (K^{\bullet}_{\C})^{\Gamma_Q}\big) \simeq H^q_Q(\Yo, \Gamma_Q, \shfC).
H^q(T_w/\Gamma, \C) \isomto H^q \big( (K^{\bullet}_{\C})^{\Gamma}\big) \simeq H^q_Q(\Yo, \Gamma, \shfC).
\]
\end{enumerate}
\end{cor}
\begin{proof}
First note that since the action of $\Gamma_Q$ on $T_w$ is free and properly discontinuous, the singular homology $H_q(T_w/\Gamma, \Z)$ (resp. singular cohomology $H^q(T_w/\Gamma, \C)$) can be computed by the equivariant singular homology (resp. equivariant singular cohomology), i.e., we have
\begin{align*}
H_q(T_w/\Gamma, \Z) &\simeq H_q\big( (S_{\bullet})_{\Gamma} \big), \\
H^q(T_w/\Gamma, \C) &\simeq H^q\big( (S^{\bullet}_{\C})^{\Gamma}\big). 
\end{align*}
Cf. \cite[Chapter XVI Section 9]{MR0077480}. 

%First, we have a short exact sequence
(1) We consider the tautological exact sequence
\begin{align}\label{eqn:ses}
0 \ra K_{q} \ra S_{q} \ra S_{q}/K_{q} \ra 0. 
\end{align}
By Proposition \ref{prop:qis1} (1), we obtain a short exact sequence
\[
0=H_1(\Gamma, S_q/K_q) \ra (K_q)_{\Gamma} \ra (S_q)_{\Gamma} \ra (S_q/K_q)_{\Gamma} \ra 0, 
\]
where $H_1(\Gamma, -)$ is the first group homology of $\Gamma$. 
This induces a long exact sequence
\[
\cdots \ra
H_{q+1} \big((S_{\bullet}/K_{\bullet})_{\Gamma} \big) \ra H_{q} \big( (K_{\bullet})_{\Gamma} \big) \ra H_{q} \big( (S_{\bullet})_{\Gamma} \big) \ra H_{q} \big((S_{\bullet}/K_{\bullet})_{\Gamma} \big)
%H_{q} \big( \big)
\ra \cdots
\]
Therefore, it remains to show
\[
H_{q} \big((S_{\bullet}/K_{\bullet})_{\Gamma} \big) = 0
\]
for $q \geq 0$. 
Indeed, by Proposition \ref{prop:qis1}, we see that 
\begin{align}\label{eqn:exact seq}
\cdots \ra 
S_2/K_2 \ra S_1/K_1 \ra S_0/K_0 \ra 0
\end{align}
is an exact sequence of induced $\Gamma$-modules. Therefore, \eqref{eqn:exact seq} can be seen as a $(-)_{\Gamma}$-acyclic resolution of $0$. Thus we see
\[
H_q \big( (S_{\bullet}/K_{\bullet})_{\Gamma} \big) =H_q(\Gamma, 0)= 0
\]
for all $q \geq 0$. 

(2) can be proved similarly. 
%Set
%\[
%(S_q/K_q)^{\vee}_{\C} := \Hom_{\Z}(S_q/K_q, \C). 
%\]
%Since $\Hom_{\C}(-, \C)$ is an exact functor, we have a short exact sequence
%We have a short exact sequence
By applying $\Hom_{\Z}(-,\C)$ to \eqref{eqn:ses}, we obtain a short exact sequence
\[
0 \ra (S_{q}/ K_{q})^{\vee}_{\C} \ra S^{q}_{\C} \ra K^{q}_{\C} \ra 0, 
\]
where $(S_{q}/K_{q})^{\vee}_{\C} := \Hom_{\Z}(S_{q}/K_{q}, \C)$. 
Then by Proposition \ref{prop:qis1} (1), we see that $(S_{q}/ K_{q})^{\vee}_{\C}$ is a co-induced $\Gamma$-module, and hence we obtain another short exact sequence
\[
0 \ra ((S_{q}/ K_{q})^{\vee}_{\C})^{\Gamma} \ra (S^{q}_{\C})^{\Gamma} \ra (K^{q}_{\C})^{\Gamma} \ra H^1(\Gamma, (S_{q}/ K_{q})^{\vee}_{\C})=0.
\]
Furthermore, this exact sequence induces a long exact sequence
\begin{align*}
%H^q \big( \big)
\cdots \ra 
H^q\big(((S_{\bullet}/ K_{\bullet})^{\vee}_{\C})^{\Gamma} \big) \ra H^q \big((S^{\bullet}_{\C})^{\Gamma} \big) \ra H^q \big( (K^{\bullet}_{\C})^{\Gamma} \big) \ra H^{q+1}\big(((S_{\bullet}/ K_{\bullet})^{\vee}_{\C})^{\Gamma} \big)
\ra \cdots. 
\end{align*} 
Therefore, it remains to show that 
\[
H^q\big(((S_{\bullet}/ K_{\bullet})^{\vee}_{\C})^{\Gamma} \big)=0
\] 
for $q \geq 0$. 
%
%Indeed, by \eqref{eqn:ses} and Proposition \ref{prop:qis1} (2), we see that the complex $(S_{\bullet}/K_{\bullet})$ is exact. Therefore, by applying $\Hom_{\Z}(-, \C)$, we obtain an exact sequence
Indeed, by applying $\Hom_{\Z}(-, \C)$ to \eqref{eqn:exact seq}, we see that
\[
0 \ra (S_{0}/ K_{0})^{\vee}_{\C} \ra (S_{1}/ K_{1})^{\vee}_{\C} \ra (S_{2}/ K_{2})^{\vee}_{\C} \ra \cdots. 
\]
is a $(-)^{\Gamma}$-acyclic resolution of $0$, and hence
\[
H^q\big( ((S_{\bullet}/ K_{\bullet})^{\vee}_{\C})^{\Gamma} \big) \simeq H^q(\Gamma, 0)=0
\]
for all $q \geq 0$. %This completes the proof.
\end{proof}

%20201009
%20201011
As a result, for a subgroup $\Gamma \subset \Gamma_Q$ and a homology class $\mfz \in H_{g-1}(T_w/\Gamma, \Z)$, we can define an evaluation map
\begin{align}\label{eqn:pairing z}
%\brk{\mfz, \ }: H_Q^{g-1} (\Yo, \Gamma, \shfC) \simeq H^{g-1}(T_w/\Gamma, \C) \ra \C. 
\brk{\mfz, \ }: H_Q^{g-1} (\Yo, \Gamma_Q, \shfC) \simeq H^{g-1}(T_w/\Gamma_Q, \C)\ra H^{g-1}(T_w/\Gamma, \Z) \stackrel{\brk{\mfz, \ }}{\longrightarrow} \C 
\end{align}
by taking the paring with $\mfz$.

\subsection{Shintani decomposition}
%\subsection{Cone decomposition}
%\subsection{Signed fundamental domain}
%20201009: tried to deal with all connected components
%20201010: changed to deal only with totally positive part
%\subsection{Weighted fundamental domain}
Using Corollary \ref{cor:qis}, here we construct a cone decomposition of a homology class $\mfz \in H_{g-1}(T_w/\Gamma, \Z)$. Cf. Proposition \ref{prop:sfd} and Remark \ref{rmk:sfd}. 
%Using Corollary \ref{cor:qis}, here we construct a version of Shintani decomposition of a homology class $\mfz \in H_{g-1}(T_w/\Gamma, \Z)$. Cf. Proposition \ref{prop:sfd} and Remark \ref{rmk:sfd}. 
We need such a cone decomposition in order to compute the specialization of the Shintani-Barnes cocycle.

Recall that $\tau_1, \dots, \tau_g$ are the field embeddings of $F$ into $\C$. Clearly, $\tau_i$ extends to 
\[
\tau_i : F_{\R}=F \otimes \R \ra \C. 
\]
Let $F_{\tau_i}$ denote the completion of $F$ with respect to the embedding $\tau_i$. In the following, we assume $\tau_1, \dots, \tau_{r_1}$ are the real embeddings, i.e., $F_{\tau_i}=\R$ for $i=1\dots, r_1$, and $\tau_{r_1+1}, \dots, \tau_g$ are the non-real embeddings, i.e., $F_{\tau_i}=\C$ for $i=r_1+1, \dots, g$, for simplicity. 

For $\mu=(\mu_1, \dots, \mu_{r_1}) \in \{\pm 1\}^{r_1}(:=\{-1, 1\}^{r_1})$, set
\begin{align*}
F_{\R, \mu}^{\times}:=\{x \in F_{\R}^{\times} \ |\ \mu_i \tau_i(x)>0, \ i=1, \dots, r_1 \}. 
\end{align*}
Clearly, $F_{\R, \mu}^{\times}$ are the connected components of $F_{\R}^{\times}$, and we have $F_{\R}^{\times}=\coprod_{\mu \in \{\pm 1\}^{r_1}} F_{\R, \mu}^{\times}$. 
Then let $T_{w, \mu} \subset T_w$ be the connected component of $T_w$ corresponding to $F_{\R,\mu}^{\times}$ via the identification \eqref{eqn:isom w2}:
\[
w: T_w \isomto F_{\R}^{\times}. 
\]

If $\mu=(1,1, \dots, 1)$, then $F_{\R, \mu}^{\times}$ is the totally positive component of $F_{\R}^{\times}$, and simply denoted by $F_{\R, +}^{\times}$. Furthermore, let 
\begin{align*}
F_{+}^{\times} &:= F^{\times} \cap F_{\R,+}^{\times}=\{x \in F^{\times} \mid \tau_i(x)>0, \  i=1, \dots, r_1  \}, \\
\mcO_{+}^{\times} &:=\mcO^{\times} \cap F_{\R,+}^{\times}=\{u \in \mcO^{\times} \mid \tau_i(u)>0, \ i=1, \dots, r_1  \} 
\end{align*}
%denote the totally positive parts of $F^{\times}$ and $\mcO^{\times}$ respectively, and let $\Gamma_Q^+ \subset \Gamma_Q$ be the subgroup corresponding to $\mcO_{+}^{\times}$ via the isomorphism
denote the totally positive parts of $F^{\times}$ and $\mcO^{\times}$ respectively, and let $\Gamma_Q^+ \subset \Gamma_Q$ be the image of $\mcO_{+}^{\times}$ under the isomorphism
\[
\rho_w: \mcO^1 \isomto \Gamma_Q,
\]
cf. Section \ref{subsect:number field}.

By Dirichlet's unit theorem, we know that 
\[
T_w/\Rpos \Gamma_Q^+ \simeq F_{\R}^{\times}/\Rpos \mcO^{\times}_+
\]
%is compact, and each of the connected component, including
is compact, and its connected components 
\[
T_{w, \mu}/\Rpos \Gamma_Q^+ \simeq F_{\R,\mu}^{\times}/\Rpos \mcO^{\times}_+, \ \mu \in \{\pm 1\}^{r_1}
\]
are homeomorphic to $(g-1)$-dimensional topological tori. 
%
%We equip $T_w/\Rpos \Gamma_Q^+$ with an orientation induced from  the standard orientation of $T_w \subset \R^g$. 
Therefore, we have
\begin{align}\label{eqn:ori}
H_{g-1}(T_w/ \Gamma_Q^+, \Z) \simeq H_{g-1}(T_w/\Rpos \Gamma_Q^+, \Z) \simeq \Z^{\{\pm 1\}^{r_1}}.
\end{align}
%
%In order to fix the second isomorphism, we equip $T_w/\Rpos \Gamma_Q^+$ with an orientation induced from  the standard orientation of $T_w \subset \R^g$. 
%
%
%20200927
Here the first isomorphism is a canonical isomorphism induced from the projection
\[
T_w/\Gamma_Q^+ \ra T_w/\Rpos \Gamma_Q^+,
\]
which is clearly a homotopy equivalence. 
In order to fix the second isomorphism of \eqref{eqn:ori}, we equip $T_w/\Rpos \Gamma_Q^+$ with an orientation as follows. 
%We first equip the $(g-1)$-sphere $(\R^g \setm \{0\})/\Rpos$ with the usual orientation, i.e., 
%

\vspace{2mm}
\noindent
\textbf{Orientation.}
%For simplicity, set 
Set 
\begin{align*}
\mathbf T_{\mu}:=  T_{w, \mu}/\Rpos \Gamma_Q^+ \subset \mathbf T:=  T_w/\Rpos \Gamma_Q^+ 
\end{align*}
for simplicity. 
Recall that an orientation of a $(g-1)$-dimensional manifold $X$ is defined as a system $(\nu_x)_{x\in X}$ of generators $\nu_x \in H_{g-1}(X, X \setm \{x\}, \Z) \simeq \Z$ satisfying certain compatibility, cf. \cite[Section 3.3]{H02:Alg}. Note that giving a generator $\nu_x$ of $H_{g-1}(X, X\setm \{x\},\Z) \simeq \Z$ is equivalent to giving an isomorphism
\begin{align*}
o_x: H_{g-1}(X, X\setm \{x\}, \Z) \isomto \Z; \ \nu_x \mapsto 1. 
\end{align*}
We first fix an orientation of the $(g-1)$-sphere $\bfS^{g-1}=(\R^g \setm \{0\})/\Rpos$ as follows. 
%Let $\bar x \in \bfS^{g-1}$, where $x \in \R^g \setm \{0\}$, and let $I=(\alpha_1, \dots, \alpha_g) \in (\XQ)^g$ such that $x \in C_I \setm  \partial C_I$. Then we easily see
Let $x \in \R^g \setm \{0\}$ and let $\bar x \in \bfS^{g-1}$ be its image. Moreover, let $I=(\alpha_1, \dots, \alpha_g) \in (\XQ)^g$ such that $0 \notin |\sigma_I|$ and $x \notin  \partial C_I$, where $\partial C_I$ is the boundary of the cone $C_I$.  
Then we see
\[
\overline{\sigma}_I: \Delta^{g-1} \stackrel{\sigma_I}{\ra} \R^g\setm \{0\} \ra \bfS^{g-1}
\]
defines a class $[\overline{\sigma_I}] \in H_{g-1}(\bfS^{g-1}, \bfS^{g-1}\setm \{\bar x\},\Z)$. We fix the isomorphism $o_{\bar x}$ so that we have
\begin{align*}
o_{\bar x}([\overline{\sigma_I}])=\sgn (I) \mathbf{1}_{C_I}(x)
\end{align*}
for all such $I$, 
where $\sgn (I)= \sgn (\det I) \in \{0, \pm 1\}$. 
This defines an orientation of $\bfS^{g-1}$. 
Then this orientation of $\bfS^{g-1}$ induces orientations of $T_w/\Rpos \subset \bfS^{g-1}$ and $\bfT=T_w/ \Rpos\Gamma_Q^+$ because the action of $\Gamma_Q^+$ on $T_w/\Rpos$ is free, properly discontinuous, and orientation preserving. 
More explicitly, for $x \in T_w$ and its image $\bfx \in \bfT$, the local orientation isomorphism
\begin{align*}
o_{\bfx}: H_{g-1}(\bfT, \bfT \setm \{\bfx\}, \Z) \isomto \Z  
\end{align*}
can be computed as follows. 
let $I=(\alpha_1, \dots, \alpha_g) \in (\XQ)_Q^g$ such that $\gamma x \notin \partial C_I$ for all $\gamma \in \Gamma_Q^+$. Then 
\[
\bm{\sigma}_I: \Delta^{g-1} \stackrel{\sigma_I}{\ra} T_w \ra \bfT
\]
defines a class $[\bm{\sigma}_I] \in H_{g-1}(\bfT, \bfT \setm \{\bfx\}, \Z)$, and we have
\begin{align}\label{eqn:loc ori T}
o_{\bfx}([\bm{\sigma}_I])= \sgn(I) \sum_{\gamma \in \Gamma_Q^+} \mathbf{1}_{C_I}(\gamma x). 
\end{align}
%Here, note that since the action of $\Gamma_Q^+$ is properly discontinuous, we have $\mathbf{1}_{C_I}(\gamma x)=0$ for all but finite number of $\gamma \in \Gamma_Q^+$. 
%

Now, since $\bfT_{\mu}$ are the connected components of $\bfT$, this orientation defines isomorphisms
\[
o_{\mu}: H_{g-1}(\bfT_{\mu}, \Z) \isomto \Z, \ \mu \in \{\pm 1\}^{r_1}, 
\]
\[
o=\oplus_{\mu}o_{\mu}:  H_{g-1}(\bfT, \Z) \simeq \bigoplus_{\mu \in \{\pm 1\}^{r_1}} H_{g-1}(\bfT_{\mu}, \Z) \isomto \bigoplus_{\mu \in \{\pm 1\}^{r_1}}\Z
\]
such that for all $\bfx \in \bfT_{\mu}$, the following diagram is commutative:
\begin{align}\label{eqn:loc glo ori}
\begin{split}
\xymatrix{
H_{g-1}(\bfT_{\mu}, \Z) \ar[r]^-{\stackrel{o_{\mu}}{\sim}}\ar[d]_-{\loc_{\bfx}}& \Z \ar@{=}[d]\\
H_{g-1}(\bfT, \bfT \setm \{\bfx\}, \Z) \ar[r]^-{\stackrel{o_{\bfx}}{\sim}}&\Z \\
%H_{g-1}(\bfT_{\mu}, \Z) \ar[rr]^-{\stackrel{o_{\mu}}{\sim}}\ar[d]_-{\loc_{\bfx}}&& \Z \ar@{=}[d]\\
%H_{g-1}(\bfT_{\mu}, \bfT_{\mu} \setm \{\bfx\}, \Z) \ar@{=}[r]&H_{g-1}(\bfT, \bfT \setm \{\bfx\}, \Z) \ar[r]^-{\stackrel{o_{\bfx}}{\sim}}&\Z \\
}
\end{split}
\end{align}
where the left vertical arrow is the natural localization map, 
%and the lower left equality is actually the excision isomorphism, 
cf. \cite[Theorem 3.26, Lemma 3.27]{H02:Alg}.

\vspace{2mm}
%\noindent
%\textbf{Fundamental classes.}
%Now, 

For $\chi =(\chi_{\mu})_{\mu} \in \bigoplus_{\mu \in \{\pm 1\}^{r_1}} \Z$, let
\[
\mfz_{\chi} \in H_{g-1}(\bfT, \Z) \simeq H_{g-1}(T_w/\Gamma_Q^+, \Z)
\]
denote the class such that $o(\mfz_{\chi})=\chi$. 
%
%Note that if $\chi=(\delta_{\mu_0, \mu})_{\mu \in \{\pm 1\}^{r_1}}$ for some $\mu_0 \in \{\pm 1\}^{r_1}$, where $\delta_{\mu_0, \mu}$ is the Kronecker delta, then $\mfz_{\chi}$ is the fundamental class of $\bfT_{\mu_0}$, and simply denoted by $\mfz_{\mu_0}$.  
Note that if $\mfz_{\mu}$ denotes the fundamental class of $\bfT_{\mu}$, then $\mfz_{\chi}$ can be written as $\mfz_{\chi}=\sum_{\mu}\chi_{\mu} \mfz_{\mu}$.

\begin{prop}\label{prop:sfd}
Let $\chi =(\chi_{\mu})_{\mu} \in \bigoplus_{\mu \in \{\pm 1\}^{r_1}} \Z$. 
\begin{enumerate}
%\item The homology class $\mfz_{\chi} \in H_{g-1}(T_w/\Gamma_Q^+, \Z)$ can be represented by
\item There exists 
\[
\Phi = \sum_{i=1}^{r} c_i \sigma_{I_i} \in K_{g-1}=\Z \big[\sigma_I \ \big|\ I \in (\XQ)^g_Q \big] \subset S_{g-1}
\]
which represents the homology class $\mfz_{\chi} \in H_{g-1}(T_w/\Gamma_Q^+, \Z)$, where  $I_1, \dots, I_{r} \in (\XQ)^g_Q$, and $c_i \in \Z$. 
\item Then for $x \in \R^g \setm \{0\}$, we have
\[
\sum_{\gamma \in \Gamma_Q^+} \sum_{i=1}^{r} c_i \sgn (I_i) \mathbf{1}_{C_{I_i}^Q}(\gamma x) =\chi(x)\mathbf{1}_{T_w}(x),
\]
%where $\sgn (I_i)= \sgn (\det I_i) \in \{0, \pm 1\}$, and $\mathbf{1}_A$ denotes the characteristic function of $A \subset \R^g$. 
where $\chi$ is regarded as a locally constant function $\chi: T_w \ra \Z$ which has value $\chi_{\mu}$ on $T_{w,\mu}$, i.e., $\chi(x)=\chi_{\mu}$ for $x \in T_{w, \mu}$. 
%\begin{align*}
%\chi: \R^g\setm \{0\} \ra \Z; x \mapsto 
%\chi: T_w \ra \Z; \ x \mapsto 
%\begin{cases}
%0 & \text{ if } x \notin T_w \\
%\chi_{\mu} & \text{ if } x \in T_{w, \mu} \text{ for }  \mu \in \{\pm 1\}^{r_1}.
%\end{cases}
%\sum_{\mu \in \{\pm 1\}^{r_1}} \chi_{\mu} \mathbf{1}_{T_{w, \mu}}(x). 
%\end{align*}
\end{enumerate}
\end{prop}

\begin{proof}
(1) is a direct consequence of Corollary \ref{cor:qis} (1). 

(2) First note that we have
\[
\mathbf{1}_{C_{I_i}^Q}(\gamma x) =\mathbf{1}_{\gamma^{-1} C_{I_i}^Q}(x)=\mathbf{1}_{C_{\gamma^{-1}I_i}^Q}(x)
\]
for $\gamma \in \Gamma_Q^+$. 
Now, since the action of $\Gamma_Q^+$ on $T_w/\Rpos$ is properly discontinuous, the collection $\{\gamma^{-1} C_{I_i}\}_{i, \gamma}$ of subsets of $T_w$ is locally finite.
%Hence, as in the proof of Proposition \ref{prop:cocycle}, by Lemma \ref{lem:gen} there exists $\delta>0$ such that
%Therefore, as in the proof of Proposition \ref{prop:cocycle}, by using Lemma \ref{lem:gen}, we see that there exists $\delta>0$ such that
Therefore, as in the proof of Proposition \ref{prop:cocycle}, by using Lemma \ref{lem:gen}, we can find $\delta>0$ such that
\begin{align*}
\exp (\varepsilon Q)x \notin  \partial C_{\gamma^{-1}I_i}
\end{align*}
for all $\varepsilon \in (0, 2\delta)$, $i=1, \dots, {r}$, and $\gamma \in \Gamma_Q^+$. 
%Here $\partial C_{\gamma^{-1}I_i}$ denotes the (topological) boundary of the cone $C_{\gamma^{-1}I_i}$. 
Set
\[
x':= \exp(\delta Q)x.
\]
Then we have
\begin{align*}
\mathbf{1}_{C_{I_i}^Q}(\gamma x) =\mathbf{1}_{C_{\gamma^{-1} I_i}^Q}(x) =\mathbf{1}_{C_{\gamma^{-1} I_i}}(x')=\mathbf{1}_{C_{I_i}}(\gamma x').
\end{align*}
Moreover, by using Lemma \ref{lem:rev} (5), we see that $\exp(\delta Q)$ preserves the connected components $T_{w, \mu}$ of $T_w$, and hence we have
\[
\chi(x) \mathbf{1}_{T_w}(x)= \chi(x') \mathbf{1}_{T_w}(x'). 
\]
%since all the real eigenvalues of $\exp(\delta Q)$ is positive, and $\exp(\delta Q)$ preserves $T_w^+$.
Therefore, it suffices to show
\begin{align}\label{eqn:sum}
\sum_{\gamma \in \Gamma_Q^+} \sum_{i=1}^{r} c_i \sgn (I_i) \mathbf{1}_{C_{I_i}}(\gamma x') =\chi(x')\mathbf{1}_{T_w}(x').
\end{align}
First, by Lemma \ref{lem:Q-int simpl}, all of the terms in \eqref{eqn:sum} are $0$ if $x' \notin T_w$. Therefore, we assume $x' \in T_{w,\mu}$ for some $\mu \in \{\pm 1\}^{r_1}$.
Set
\begin{align*}
%&\overline{T}:=T_w/\Rpos \Gamma_Q^+, \\
&\bfT_{\mu} :=T_{w, \mu}/\Rpos \Gamma_Q^+, \\
%&\overline{x'}:= \Rpos \Gamma_Q^+ x' \in \overline{T}
&\bfx' := \Rpos \Gamma_Q^+ x' \in \bfT_{\mu}. 
\end{align*}
%as before. %for simplicity. 
%
%Then by the definition of the orientation of $\bfT=T_w/\Rpos \Gamma_Q^+$, we see that the image of $\Phi$ under the natural localization map
Then by \eqref{eqn:loc ori T}, we see that the image of $\Phi$ under the localization map
\[
%H_{g-1}(T_w/\Gamma_Q^+, \Z)  \ra H_{g-1}(T_w/\Rpos\Gamma_Q^+, (T_w \setm x'\Rpos\Gamma_Q^+)/\Rpos\Gamma_Q^+, \Z) \simeq \Z
%H_{g-1}(T_w/\Gamma_Q^+, \Z)  \ra H_{g-1}(T_w/\Gamma_Q^+,  T_w/\Gamma_Q^+ \setm x' \Gamma_Q^+, \Z) \simeq \Z
%H_{g-1}(T_w/\Gamma_Q^+, \Z) \simeq H_{g-1} \big(\overline{T}, \Z \big)  \ra H_{g-1}\big(\overline{T}, \overline{T} \setm \{\overline{x'}\}, \Z \big) \simeq \Z
o_{\bfx'} \circ \loc_{\bfx'}: H_{g-1}(T_{w, \mu}/\Gamma_Q^+, \Z) \simeq H_{g-1} \big(\bfT_{\mu}, \Z \big)  \stackrel{\loc_{\bfx'}}{\longrightarrow} H_{g-1}\big(\bfT, \bfT \setm \{\bfx'\}, \Z \big) \stackrel{o_{\bfx'}}{\isomto} \Z
\]
is equal to
\begin{align*}
\sum_{\gamma \in \Gamma_Q^+} \sum_{i=1}^{r} c_i \sgn (I_i) \mathbf{1}_{C_{I_i}}(\gamma x').
%\sum_{\gamma \in \Gamma_Q^+} \sum_{i=1}^{r} c_i \sgn (I_i) \mathbf{1}_{C_{I_i}}(\gamma x') =
%\begin{cases}
%1 &\text{ if } x' \in T_w^+ \\
%0 &\text { otherwise}
%\end{cases}
\end{align*}
On the other hand, by \eqref{eqn:loc glo ori}, we have 
\[
o_{\bfx'} \circ \loc_{\bfx'}(\Phi)=o_{\mu}(\mfz_{\chi})=\chi_{\mu}
\]
because $\Phi$ represents $\mfz_{\chi}$. 
This completes the proof. 
\end{proof}

\begin{rmk}\label{rmk:sfd} 
In the case where $\mfz_{\chi}=\mfz_{\mu}$ is the fundamental class of a connected component $\bfT_{\mu}$, % , i.e., $\chi=(\delta_{\mu_0, \mu})_{\mu}$, 
Proposition \ref{prop:sfd} says that
\[
\sum_{i=1}^{r} c_i \sgn (I_i) \mathbf{1}_{C_{I_i}^Q}
\]
gives a signed fundamental domain for $T_{w, \mu}/\Gamma_Q^+$ in the sense of Charollois, Dasgupta, and Greenberg~\cite[Definition 2.4]{CDG15:Int}, which is a ``weighted version'' of the Shintani cone decomposition, cf. also \cite{MR3198753}, \cite{MR4105945}. 
%Note that the number fields treated in these previous works are required to have at least one real embedding. 
\end{rmk}

\begin{rmk}\label{rmk:eval comp}
Let the notations $\chi, \mfz_{\chi}$, and $\Phi=\sum_{i=1}^r c_i \sigma_{I_i}$ be the same as in Proposition \ref{prop:sfd}. 
We can compute the evaluation map 
\[
\brk{\mfz_{\chi}, \ }: H_Q^{g-1} (\Yo, \Gamma_Q, \shfC) \simeq H^{g-1}(T_w/\Gamma_Q, \C)\ra H^{g-1}(T_w/\Gamma_Q^+, \Z) \stackrel{\brk{\mfz_{\chi}, \ }}{\longrightarrow} \C, 
\]
cf. \eqref{eqn:pairing z}, explicitly as follows. Let 
\[
%s=(s_I)_{I \in (\XQ)^{g}_Q} \in  C_Q^{g-1}(\mcX_{\Q}, \shfC)^{\Gamma_Q} =\left( \prod_{I\in (\XQ)^{g}_Q} \C \right)^{\Gamma_Q}
s=(s_I)_{I \in (\XQ)^{g}_Q} \in  C_Q^{g-1}(\mcX_{\Q}, \shfC) = \prod_{I\in (\XQ)^{g}_Q} \C 
\]
be a $\Gamma_Q$-invariant cocycle and let $[s] \in H_Q^{g-1}(\Yo, \Gamma_Q, \shfC)$ be the class represented by $s$. Then we have
\[
\brk{\mfz_{\chi}, [s]}=\sum_{i=1}^{r} c_i s_{I_i}. 
\]
\end{rmk}

%%%%%%%%%%%%%%%%%%%%%%%%%%%%%%%%%%%%%%%%%%%%%%%%%%%%%%%%%%%%%%%%%%%%%%%%%%%%%%%

\subsection{Values of the zeta functions}
%We keep the notations above. 
Recall that $F$ is a number field of degree $g$, $\mcO$ is an order in $F$, and $\mfa \subset F$ is a proper fractional $\mcO$-ideal. 
\begin{dfn}
\begin{enumerate}
\item 
For a continuous map
\[
\chi: F_{\R}^{\times}=(F \otimes_{\Q} \R)^{\times} \ra \Z, 
\]
let
\[
\zeta_{\mcO}(\chi, \mfa^{-1}, s):= 
\sum_{x \in (\mfa- \{0\})/\mcO^{\times}_+} \frac{\chi(x)}{|N_{F/\Q}(x)|^s}, \quad \re(s)>1
\]
denote the partial zeta function associated to $\chi$ and a proper fractional $\mcO$-ideal $\mfa^{-1}$. 
Here, note that $\chi$ is constant on each connected component of $F_{\R}^{\times}$, and hence invariant under the action of $\mcO^{\times}_+$. 
\item Let
\[
\bm{\varepsilon}:  F_{\R}^{\times} \ra \{\pm 1\};\  x \mapsto \frac{N_{F/\Q}(x)}{|N_{F/\Q}(x)|}
\]
denote the sign character. 
\end{enumerate}
\end{dfn}

Now, let $k \geq 1$, and let $\chi \in \bigoplus_{\mu \in \{\pm 1\}^{r_1}}\Z$. 
Note that $\chi$ can be regarded as a continuous map $\chi: F_{\R}^{\times} \ra \Z$ via
\[
\chi: F_{\R}^{\times} \ra F_{\R}^{\times}/F_{\R, +}^{\times} \simeq \{\pm 1\}^{r_1} \stackrel{\chi}{\longrightarrow} \Z. 
\]
So far, we have defined the following series of maps between cohomology groups:
\begin{align}\label{eqn:spec}
\vcenter{
%\xymatrix@C=0pt@R=20pt{
%H^{g-1}(\Yo, SL_g(\Z), \msF_{kg}^{\Xi}) \ar[d]^{\ev_Q} & \ni [\Psi_{kg}]\\
%H^{g-1}(\Yo, \Gamma_Q, \msF_{kg}) \ar[d]^{N_{w^*}^k} & \\
%H^{g-1}(\Yo, \Gamma_Q, \msF_0) \ar[d]^{\int_Q} & \\
%H_Q^{g-1}(\Yo, \Gamma_Q, \C) \ar[d]^{\brk{\mfz_{\chi}, \ }} & \\
%\C &
%}
\xymatrix@C=0pt@R=20pt{
H^{g-1}(\Yo, SL_g(\Z), \msF_{kg}^{\Xi}) \ar[d]^{\ev_Q} & \ni &[\Psi_{kg}] \ar@{|->}[dddd]\\
H^{g-1}(\Yo, \Gamma_Q, \msF_{kg}) \ar[d]^{N_{w^*}^k} && \\
H^{g-1}(\Yo, \Gamma_Q, \msF_0) \ar[d]^{\int_Q} && \\
H_Q^{g-1}(\Yo, \Gamma_Q, \C) \ar[d]^{\brk{\mfz_{\chi}, \ }} && \\
\C & \ni &\left\langle \mfz_{\chi}, {\int_Q}  N_{w^*}^k  \ev_Q ([\Psi_{kg}]) \right\rangle
}
}
\end{align}
%
%See Corollary \ref{cor:equiv coh}, Corollary \ref{cor:eq coh1.5}, Example \ref{ex:norm}, \eqref{map:intQ}, and \eqref{map:zQ} for the definitions of these maps. 
See Corollary \ref{cor:equiv coh}, Example \ref{ex:norm}, \eqref{map:intQ}, and Remark \ref{rmk:eval comp} for the definitions of these maps.

\begin{thm}\label{thm:main thm}
We have
\[
\left\langle \mfz_{\chi}, {\int_Q}  N_{w^*}^k  \ev_Q ([\Psi_{kg}]) \right\rangle
=
%\frac{(k!)^g}{(g+gk-1)!}\frac{\sqrt{|D_{\mcO}|}}{N\mfa^{k}} \zeta_{\mcO, +}(\mfa^{-1}, 1+k),
%\frac{\det (w^{(1)}, \dots, w^{(g)})(k!)^g}{(g+gk-1)! N\mfa^{k+1}}
\frac{(k!)^g\det (w^{(1)}, \dots, w^{(g)})}{(g+gk-1)!}
%\frac{(k!)^g}{(g+gk-1)!}\frac{\det (w^{(1)}, \dots, w^{(g)})}{N\mfa^{k+1}}
\zeta_{\mcO}(\bm{\varepsilon}^{k+1}\chi, \mfa^{-1}, k+1), 
\]
where $\bm{\varepsilon}^{k+1}\chi (x)=\bm{\varepsilon}(x)^{k+1}\chi(x)$. 
\end{thm}

\begin{proof}
%By Hurwitz' formula, cf. Proposition \ref{prop:Hur formula} (2) and Example \ref{ex:Hur formula}, we see that the class ${\int_Q}  N_{w^*}^k  \ev_Q ([\Psi_{kg}]) \in H_Q^{g-1}(\Yo, \Gamma_Q, \C)$ is represented by
By Hurwitz' formula Example \ref{ex:Hur formula}, we see that the class \[
{\int_Q}  N_{w^*}^k  \ev_Q ([\Psi_{kg}]) \in H_Q^{g-1}(\Yo, \Gamma_Q, \C)
\] 
is represented by
\begin{align*}
&\left( \int_{Q, I} N_{w^*}(y)^k \psi_{kg, I}^Q(y)\omega(y) \right)_{I \in (\XQ)_Q^g} \\
=&
\left( \frac{(k!)^g\det (w^{(1)}, \dots, w^{(g)})}{(g+gk-1)!} \sgn(I) \sum_{x \in C_I^Q \cap \Z^g\setm \{0\}} \frac{1}{N_w(x)^{k+1}} \right)_{I \in (\XQ)_Q^g}.
\end{align*}

On the other hand, by Proposition \ref{prop:sfd} (1), we can take a representative
\[
\Phi=\sum_{i=1}^r c_i \sigma_{I_i} \in K_{g-1}=\Z \big[\sigma_I \ \big|\ I \in (\XQ)^g_Q \big] \subset S_{g-1}
\]
of $\mfz_{\chi} \in H_{g-1}(T_w/\Gamma_Q^+, \Z)$. 
Then by using Remark \ref{rmk:eval comp} and Proposition \ref{prop:sfd} (2), we find
\begin{align*}
&\left\langle \mfz_{\chi}, {\int_Q}  N_{w^*}^k  \ev_Q ([\Psi_{kg}]) \right\rangle \\
=&
\frac{(k!)^g\det (w^{(1)}, \dots, w^{(g)})}{(g+gk-1)!} \sum_{i=1}^{r} c_i \sgn(I_i) \sum_{x \in C_{I_i}^Q \cap \Z^g\setm \{0\}} \frac{1}{N_w(x)^{k+1}} \\
=&
\frac{(k!)^g\det (w^{(1)}, \dots, w^{(g)})}{(g+gk-1)!} \sum_{x \in \Z^g\setm \{0\}} \sum_{i=1}^{r} c_i \sgn(I_i) \mathbf{1}_{C_{I_i}^Q}(x) \frac{1}{N_w(x)^{k+1}} \\
=&
\frac{(k!)^g\det (w^{(1)}, \dots, w^{(g)})}{(g+gk-1)!} \sum_{x \in (\Z^g\setm \{0\})/\Gamma_Q^+} \sum_{\gamma \in \Gamma_Q^+} \sum_{i=1}^{r} c_i \sgn(I_i) \mathbf{1}_{C_{I_i}^Q}(\gamma x) \frac{1}{N_w(x)^{k+1}} \\
=&
%\stackrel{\ast}{=}&
\frac{(k!)^g\det (w^{(1)}, \dots, w^{(g)})}{(g+gk-1)!} \sum_{x \in (\Z^g \setm \{0\})/\Gamma_Q^+}  \frac{\chi(x)}{N_w(x)^{k+1}} \\
=&
\frac{(k!)^g\det (w^{(1)}, \dots, w^{(g)})}{(g+gk-1)!} \sum_{x \in (\mfa - \{0\})/\mcO^{\times}_+}  \frac{\bm{\varepsilon}(x)^{k+1} \chi(x)}{|N_{F/\Q}(x)|^{k+1}}. \qedhere
\end{align*}
%Note that the $4$-th equality follows from Proposition \ref{prop:sfd}. 
\end{proof}

\begin{rmk}\label{rmk:main thm}
%We easily see that 
%Note that we have
It is easy to see that
%It is well-known that
\[
\det (w^{(1)}, \dots, w^{(g)})^2 = D_{\mcO} N\mfa^2,  
\]
where $D_{\mcO}$ is the discriminant of the order $\mcO$. 
Moreover, we also know that $\sgn (D_{\mcO})=(-1)^{r_2}$, where $r_2$ is the number of complex places of $F$. 
Therefore, by permuting the order of the embeddings $\tau_1, \dots, \tau_g$ if necessary, we have
%Therefore, by permuting the order of the basis $w_1, \dots, w_g$ if necessary, we have
%replacing $w_1$ with $-w_1$ if necessary, we have
\[
\det (w^{(1)}, \dots, w^{(g)}) = \bmi^{r_2} \sqrt{|D_{\mcO}|} N\mfa, 
\]
where $\bmi \in \C$ is the imaginary unit. 
%Therefore, the right hand side of Theorem \ref{thm:main thm} can be rewritten as
Hence, (under a suitable ordering of $\tau_1, \dots, \tau_g$,) Theorem \ref{thm:main thm} can be also written as
%Hence, (under a suitable ordering of $w_1, \dots, w_g$,) Theorem \ref{thm:main thm} can be also written as
\[
\left\langle \mfz_{\chi}, {\int_Q}  N_{w^*}^k  \ev_Q ([\Psi_{kg}]) \right\rangle
=
\bmi^{r_2}\frac{\sqrt{|D_{\mcO}|}N\mfa (k!)^g}{(g+gk-1)!} 
\zeta_{\mcO}(\bm{\varepsilon}^{k+1}\chi, \mfa^{-1}, k+1). 
%\frac{(k!)^g}{(g+gk-1)!}\frac{\sqrt{D_{\mcO}}}{N\mfa^{k}} \zeta_{\mcO, +}(\mfa^{-1}, 1+k). 
\]
\end{rmk}

\subsection*{Acknowledgements}
I would like to express my gratitude to Takeshi Tsuji for the constant encouragement and the precious advice during the study. 
I am also grateful to Kenichi Bannai, Kei Hagihara, Kazuki Yamada, and Shuji Yamamoto for answering my questions about their work and giving me many valuable comments. 
Thanks are also due to Ryotaro Sakamoto for the helpful discussions during the study. 
Some of the ideas in this paper were obtained during my stay at the Max Planck Institute for Mathematics in 2018. 
I would like to express my appreciation to Don Zagier and G\"unter Harder for the valuable discussions and for their great hospitality during the stay. 
This work is supported by JSPS Overseas Challenge Program for Young Researchers Grant Number 201780267, JSPS KAKENHI Grant Number JP18J12744, KAKENHI 18H05233, and JSPS KAKENHI Grant Number JP20J01008.


\begin{thebibliography}{99}
\bibitem{BHYY19:Can}
K.~Bannai, K.~Hagihara, K.~Yamada, and S.~Yamamoto.
\newblock Canonical equivariant cohomology classes generating zeta values of
  totally real fields, arXiv:1911.02650.


\bibitem{10019989677}
E.~W. Barnes.
\newblock On the theory of the multiple gamma function.
\newblock {\em Trans. Cambridge Philos. Soc.}, 19:374--425, 1904.



\bibitem{BKL18:Top}
A.~Beilinson, G.~Kings, and A.~Levin.
\newblock Topological polylogarithms and {$p$}-adic interpolation of
  {$L$}-values of totally real fields.
\newblock {\em Math. Ann.}, 371(3-4):1449--1495, 2018.





\bibitem{BCG20:Tra}
N.~Bergeron, P.~Charollois, and L.~Garcia~Martinez.
\newblock Transgressions of the Euler class and Eisenstein cohomology of {$\mathrm{GL}_N (\mathbf{Z})$}.
\newblock {\em Japanese Journal of Mathematics}, 2020.




%\bibitem{B82:Coh}
%K.~S. Brown.
%\newblock {\em Cohomology of groups}, volume~87 of {\em Graduate Texts in
%  Mathematics}.
%\newblock Springer-Verlag, New York-Berlin, 1982.



\bibitem{MR0077480}
H.~Cartan and S.~Eilenberg.
\newblock {\em Homological algebra}.
\newblock Princeton University Press, Princeton, N. J., 1956.

  

\bibitem{CDG15:Int}
P.~Charollois, S.~Dasgupta, and M.~Greenberg.
\newblock Integral {E}isenstein cocycles on {$\mathbf{GL}_n$}, {II}:
  {S}hintani's method.
\newblock {\em Comment. Math. Helv.}, 90(2):435--477, 2015.



\bibitem{MR3198753}
F.~Diaz~y Diaz and E.~Friedman.
\newblock Signed fundamental domains for totally real number fields.
\newblock {\em Proc. Lond. Math. Soc. (3)}, 108(4):965--988, 2014.

\bibitem{MR4105945}
M.~Espinoza and E.~Friedman.
\newblock Twisters and signed fundamental domains for number fields.
\newblock {\em Ann. Inst. Fourier (Grenoble)}, 70(2):479--521, 2020.






\bibitem{MR3991431}
J.~Fl\'{o}rez, C.~Karabulut, and T.~A. Wong.
\newblock Eisenstein cocycles over imaginary quadratic fields and special
  values of {$L$}-functions.
\newblock {\em J. Number Theory}, 204:497--531, 2019.



\bibitem{MR0345092}
R.~Godement.
\newblock {\em Topologie alg\'{e}brique et th\'{e}orie des faisceaux}.
\newblock Hermann, Paris, 1973.
\newblock Troisi\`eme \'{e}dition revue et corrig\'{e}e, Publications de
  l'Institut de Math\'{e}matique de l'Universit\'{e} de Strasbourg, XIII,
  Actualit\'{e}s Scientifiques et Industrielles, No. 1252.




%\bibitem{H81:Per}
%G.~Harder.
%\newblock Period integrals of cohomology classes which are represented by
%  {E}isenstein series.
%\newblock In {\em Automorphic forms, representation theory and arithmetic
%  ({B}ombay, 1979)}, volume~10 of {\em Tata Inst. Fund. Res. Studies in Math.},
%  pages 41--115. Tata Inst. Fundamental Res., Bombay, 1981.



\bibitem{MR892187}
G.~Harder.
\newblock Eisenstein cohomology of arithmetic groups. {T}he case {${\rm
  GL}_2$}.
\newblock {\em Invent. Math.}, 89(1):37--118, 1987.


  \bibitem{H02:Alg}
A.~Hatcher.
\newblock {\em Algebraic topology}.
\newblock Cambridge University Press, Cambridge, 2002.




\bibitem{H17:Ube}
E.~Hecke.
\newblock {\"U}ber die kroneckersche grenzformel f{\"u}r reelle quadratische
  k{\"o}rper und die klassenzahl relativ-abelscher k{\"o}rper.
\newblock {\em Verh. d. Naturforsch. Ges. Basel}, 28:363--372, 1917.




\bibitem{H93:Ele}
H.~Hida.
\newblock {\em Elementary theory of {$L$}-functions and {E}isenstein series},
  volume~26 of {\em London Mathematical Society Student Texts}.
\newblock Cambridge University Press, Cambridge, 1993.




\bibitem{H07:Shi}
R.~Hill.
\newblock Shintani cocycles on {${\rm GL}_n$}.
\newblock {\em Bull. Lond. Math. Soc.}, 39(6):993--1004, 2007.



\bibitem{MR0344507}
L.~H\"{o}rmander.
\newblock {\em An introduction to complex analysis in several variables}.
\newblock North-Holland Publishing Co., Amsterdam-London; American Elsevier
  Publishing Co., Inc., New York, revised edition, 1973.
\newblock North-Holland Mathematical Library, Vol. 7.


\bibitem{MR2311920}
L.~H\"{o}rmander.
\newblock {\em Notions of convexity}.
\newblock Modern Birkh\"{a}user Classics. Birkh\"{a}user Boston, Inc., Boston,
  MA, 2007.
%\newblock Reprint of the 1994 edition [of MR1301332].




\bibitem{MR1512118}
A.~Hurwitz.
\newblock \"{U}ber die {A}nzahl der {K}lassen positiver tern\"{a}rer
  quadratischer {F}ormen von gegebener {D}eterminante.
\newblock {\em Math. Ann.}, 88(1-2):26--52, 1922.




%\bibitem{I86:Coh}
%B.~Iversen.
%\newblock {\em Cohomology of sheaves}.
%\newblock Universitext. Springer-Verlag, Berlin, 1986.



\bibitem{KS90:She}
M.~Kashiwara and P.~Schapira.
\newblock {\em Sheaves on manifolds}, volume 292 of {\em Grundlehren der
  Mathematischen Wissenschaften}.
\newblock Springer-Verlag, Berlin, 1990.
%\newblock With a chapter in French by Christian Houzel.



\bibitem{lim2019milnor}
S.~H. Lim and J.~Park.
\newblock The milnor $k$-theory and the shintani cocycle, arXiv:1909.03450.


\bibitem{N95:Som}
M.~V. Nori.
\newblock Some {E}isenstein cohomology classes for the integral unimodular
  group.
\newblock In {\em Proceedings of the {I}nternational {C}ongress of
  {M}athematicians, {V}ol. 1, 2 ({Z}\"{u}rich, 1994)}, pages 690--696.
  Birkh\"{a}user, Basel, 1995.




\bibitem{R91:Fun}
W.~Rudin.
\newblock {\em Functional analysis}.
\newblock International Series in Pure and Applied Mathematics. McGraw-Hill,
  Inc., New York, second edition, 1991.






\bibitem{S93:Eis}
R.~Sczech.
\newblock Eisenstein group cocycles for {$\mathrm{GL}_n$} and values of
  {$L$}-functions.
\newblock {\em Invent. Math.}, 113(3):581--616, 1993.




\bibitem{sharifi2020eisenstein}
R.~Sharifi and A.~Venkatesh.
\newblock Eisenstein cocycles in motivic cohomology, arXiv:2011.07241.




\bibitem{S76:On-}
T.~Shintani.
\newblock On evaluation of zeta functions of totally real algebraic number
  fields at non-positive integers.
\newblock {\em J. Fac. Sci. Univ. Tokyo Sect. IA Math.}, 23(2):393--417, 1976.



\bibitem{MR1631700}
D.~Solomon.
\newblock Algebraic properties of {S}hintani's generating functions: {D}edekind
  sums and cocycles on {${\rm PGL}_2({\bf Q})$}.
\newblock {\em Compositio Math.}, 112(3):333--362, 1998.




\bibitem{stacks-project}
T.~{Stacks project authors}.
\newblock The stacks project.
\newblock \url{https://stacks.math.columbia.edu}, 2020.



\bibitem{VZ13:Hig}
M.~Vlasenko and D.~Zagier.
\newblock Higher {K}ronecker ``limit'' formulas for real quadratic fields.
\newblock {\em J. Reine Angew. Math.}, 679:23--64, 2013.




\bibitem{Y10:On-}
S.~Yamamoto.
\newblock On {S}hintani's ray class invariant for totally real number fields.
\newblock {\em Math. Ann.}, 346(2):449--476, 2010. 


\end{thebibliography}
\end{document}